\documentclass[12pt]{amsart}

\usepackage{amssymb, amsthm, amsmath, amsfonts, amsxtra, mathrsfs, mathtools}
\usepackage{color}
\usepackage{import}
\usepackage{thmtools}
\usepackage{thm-restate}
\usepackage[colorlinks=false, allcolors=blue]{hyperref}
\usepackage{cleveref}
\usepackage{chngcntr}
\usepackage[intoc]{nomencl}
\makenomenclature
\usepackage{stmaryrd}
\usepackage{svg}
\usepackage{tikz-cd}
\usepackage{tikz}
\usepackage{tikzsymbols}
\usetikzlibrary{decorations.pathreplacing,angles,quotes}
\usepackage{enumitem}
\usetikzlibrary{matrix}
\usepackage[ruled,vlined]{algorithm2e}
\usepackage[mathscr]{euscript}
\usepackage[normalem]{ulem}
\usepackage{comment}
\usepackage{multicol}
\usepackage{caption, subcaption}
\declaretheorem[name=Theorem,numberwithin=section]{theorem}
\input xy
\xyoption{all}

\usepackage[letterpaper, top=3cm, bottom=3cm,left=2.7cm, right=2.7cm, heightrounded,bindingoffset=0mm]{geometry}

\newtheorem{thm}{Theorem}[section]
\newtheorem{proposition}[thm]{Proposition}
\newtheorem{corollary}[thm]{Corollary}
\newtheorem{lemma}[thm]{Lemma}
\theoremstyle{definition}
\newtheorem{definition}[thm]{Definition}
\newtheorem{example}[thm]{Example}
\newtheorem{remark}[thm]{Remark}

\newtheorem{notation}[thm]{Notation}
\newtheorem{convention}[thm]{Convention}
\newtheorem{observation}[thm]{Observation}

\def\Z{\mathbb{Z}}
\def\Q{\mathbb{Q}}
\def\R{\mathbb{R}}
\def\C{\mathbb{C}}
\def\P{\mathbb{P}}

\def\Z{\mathbb{Z}}

\def\calC{\mathcal{C}}

\def\calL{\mathcal{L}}
\def\calM{\mathcal{M}}

\def\calS{\mathcal{S}}
\newcommand{\T}{\mathcal{T}}

\def\tbI{\mathbf{I}}
\def\tbJ{\mathbf{J}}
\newcommand{\w}{\mathbf{w}}

\def\a{\mathfrak{a}}
\newcommand{\abs}[1]{\left\lvert#1\right\rvert}
\def\b{\mathfrak{b}}

\def\k{b}
\def\L{\overline{\mathcal{L}}}
\def\M{\overline{\mathcal{M}}}

\newcommand{\Y}{\mathsf{Y}}

\title{Permutohedral Complexes and Rational Curves with Cyclic Action}

\author[E. Clader]{Emily Clader}\address{Emily Clader, Department of Mathematics, San Francisco State University}
\email{\url{eclader@sfsu.edu}}

\author[C. Damiolini]{Chiara Damiolini}
\address{Chiara Damiolini, Department of Mathematics, Rutgers University}
\email{\url{chiara.damiolini@rutgers.edu}}

\author[D. Huang]{Daoji Huang}\address{Daoji Huang, ICERM, Brown University}
\email{\url{daoji_huang@brown.edu}}

\author[S. Li]{Shiyue Li}\address{Shiyue Li, Department of Mathematics, Brown University}
\email{\url{shiyue_li@brown.edu}}

\author[R. Ramadas]{Rohini Ramadas}\address{Rohini Ramadas, Department of Mathematics, Brown University}
\email{\url{rohini_ramadas@brown.edu}}

\begin{document}
\maketitle

\begin{abstract}
    We define a moduli space of rational curves with finite-order automorphism and weighted orbits, and we prove that the combinatorics of its boundary strata are encoded by a particular polytopal complex that also captures the algebraic structure of a complex reflection group acting on the moduli space.  This generalizes the situation for Losev--Manin's moduli space of curves (whose boundary strata are encoded by the permutohedron and related to the symmetric group) as well as the situation for Batyrev--Blume's moduli space of curves with involution, and it extends that work beyond the toric context.
\end{abstract}

\section{Introduction}\label{sec:introduction}

The moduli space $\M_{0,n}$ of genus-zero stable curves with $n$ distinct marked points is a fundamental object in algebraic geometry, in part due to its applicability---to such fields as enumerative geometry, representation theory, and mathematical physics, to name a few---but also because it is an interesting variety in its own right. In particular, while $\M_{0,n}$ is not toric when $n \geq 5$, it shares some of the combinatorial structure that a toric variety would enjoy.  The Chow ring of a toric variety, for example, is generated by the toric boundary (the positive-codimension torus-invariant subvarieties) with relations described combinatorially in terms of fan data; analogously, Keel proved in \cite{keel1992intersection} that the Chow ring of $\M_{0,n}$ is generated by the modular boundary (the positive-codimension boundary strata) with relations described combinatorially in terms of dual graphs.

One perspective on the close connection between $\M_{0,n}$ and toric varieties is that the moduli problem can be tweaked to produce a space that is, in fact, toric.  Losev and Manin studied a particularly significant such modification in \cite{losev2000}, constructing a moduli space $\L_n$ that parameterizes genus-zero curves with marked points $(y_1, y_2, z_1, \ldots, z_n)$, where the marked points $y_1$ and $y_2$ are ``heavy"---that is, they cannot coincide with any other marked points---whereas the marked points $z_1, \ldots, z_n$ are ``light" in the sense that they are allowed to coincide with one another. The space $\L_n$ (which is birational to $\M_{0,n-1}$) is a toric variety, and its associated polytope is the permutohedron $\Pi_n$: the convex hull in $\R^n$ of the $n!$ points obtained by permuting the coordinates of $(1,2, \ldots, n)$.   Moreover, Losev and Manin proved that the torus-invariant strata of $\L_n$ are precisely the boundary strata, which implies that there is a dimension-preserving, inclusion-preserving bijection
\[ \left\{ \substack{\textstyle\text{boundary }\\ \textstyle\text{strata in } \overline{\calL}_n}\right\} \longleftrightarrow \left\{ \substack{\textstyle\text{faces}\\\textstyle\text{ of }\Pi_n}\right\}.\]
From a combinatorial
perspective, on the other hand, the faces of $\Pi_n$ have another interpretation: they encode the generation of the symmetric group $S_n$ by adjacent transpositions.   Namely, the $d$-dimensional faces of $\Pi_n$ are in inclusion-preserving bijection with the right cosets in $S_n$ of subgroups of the form
\[\langle \tau_1, \ldots, \tau_d \rangle \subseteq S_n,\]
where $\tau_1, \ldots, \tau_d$ are adjacent transpositions.

Batyrev and Blume extended the work of Losev and Manin in \cite{Batyrev2009TheFO, Batyrev2011generalisations}, constructing a moduli space $\L^2_n$ that parameterizes genus-zero curves with an involution $\sigma$, two light fixed points of $\sigma$, one heavy marked orbit of $\sigma$, and $n$ light marked orbits.  Again, this moduli space is toric, and its torus-invariant strata are precisely the boundary strata, so one obtains a dimension-preserving, inclusion-preserving bijection
\[ \left\{ \substack{\textstyle\text{boundary }\\ \textstyle\text{strata in } \overline{\calL}^2_n}\right\} \longleftrightarrow \left\{ \substack{\textstyle\text{faces}\\ \textstyle\text{ of }\Delta^2_n}\right\}.\]
Here, $\Delta^2_n$ is the polytope known as the type-$B$ permutohedron, which is the convex hull in $\R^n$ of the $2^nn!$ points obtained by permuting the coordinates of $(\pm 1, \pm 2, \ldots, \pm n)$.  Also analogously to the Losev--Manin case, this polytope has a group-theoretic interpretation, this time in terms of the complex reflection group $S(2,n)$ of $n \times n$ matrices all of whose nonzero entries are $\pm 1$, and with exactly one nonzero entry in each row and column. 

The motivation for Batyrev and Blume's work comes from the theory of root systems.  Specifically, they constructed a toric variety associated to any root system and proved
that Losev--Manin space $\L_n$ is the toric variety associated to the classical root system $A_{n-1}$, while working instead with the root system $B_n$ yields their moduli space $\L^2_n$. From the
perspective of root systems, however, this seems to be the end of the line: Batyrev and Blume proved that the toric varieties in types $C$ and $D$ do not have equally well-behaved modular interpretations.

In this paper, we propose a generalization of Losev--Manin and Batyrev--Blume's story in a different direction.  Namely, rather than preserve the connection to root systems, we preserve from Batyrev--Blume's work the existence of an automorphism $\sigma$ but allow it to have any finite order $r$.  The result is a moduli space $\L^r_n$ that parameterizes genus-zero curves with an order-$r$ automorphism, two light fixed points, one heavy marked orbit, and $n$ light marked orbits.

In one sense, the moduli spaces $\L^r_n$ break the story, because when $r \geq 3$, they are not toric.  In particular, then, their boundary strata are not encoded by the faces of a polytope.  However, something perhaps more intriguing is true: we prove that there exists a polytopal complex $\Delta^r_n$ whose ``$\Delta$-faces" encode the boundary strata of $\L^r_n$.  More specifically, $\Delta^r_n$ is defined as the subset of
\[\Big(\mathbb{R}^{\geq 0} \cdot \mu_r \Big)^n \subseteq \C^n\] bounded by certain hyperplanes (where $\mu_r$ denotes the group of $r$th roots of unity), and its $\Delta$-faces are defined as the intersections of $\Delta^r_n$ with collections of the bounding hyperplanes.  When $r=2$, the complex $\Delta^r_n$ specializes to the type-$B$ permutohedron $\Delta^2_n$ and $\Delta$-faces are ordinary faces, so Batyrev--Blume's result is recovered.

Furthermore, we generalize the group-theoretic interpretation of both Losev--Manin and Batyrev--Blume's moduli spaces.  Namely, let $S(r,n)$ be the group of $n \times n$ matrices all of whose nonzero entries are $r$th roots of unity, and with exactly one nonzero entry in each row and column.  Then $S(r,n)$ is generated by the set
\[\mathcal{T} := \left\{\substack{\textstyle \text{adjacent transpositions}\\ \textstyle \text{ in }S_n \subseteq S(r,n)}\right\} \cup
\left\{\left( \begin{array}{ccccc} \zeta & 0 &0  & \cdots &0\\0& 1 & 0& \cdots & 0\\ 0 & 0 & 1 & \cdots & 0\\ & & & \ddots & \\ 0& 0& \cdots& 0& 1\end{array}\right)\right\},\]
where $\zeta$ is a primitive $r$th root of unity, and we define a {\bf $d$-dimensional $\T$-coset} to be a right coset of the form
\[\langle t_1, \ldots, t_d \rangle \cdot A \subseteq S(r,n)\]
with $t_1, \ldots, t_d \in \mathcal{T}$ and $A \in S(r,n)$.

Our main theorem is the following:

\begin{restatable}{thm}{main}
\label{thm:main}
For any integers $r \geq 2$ and $n \geq 0$, there are dimension-preserving, inclusion-preserving bijections
\[ \left\{ \substack{\textstyle\text{boundary }\\ \textstyle\text{strata in } \L^r_n}\right\} \longleftrightarrow \left\{ \substack{\textstyle \T\text{-cosets}\\ \textstyle\text{ in } S(r,n)}\right\} \longleftrightarrow \left\{ \substack{\textstyle\Delta\text{-faces}\\ \textstyle\text{ of }\Delta^r_n}\right\}.\]
\end{restatable}

\begin{remark}
In fact, these bijections preserve quite a bit more of the structure of the above three types of objects.  In particular, we show that there are product decompositions of boundary strata, $\T$-cosets, and $\Delta$-faces, as well as $S(r,n)$-actions on each, and the bijections of Theorem~\ref{thm:main} respect these.  See Section~\ref{sec:conclusions} for the details.
\end{remark}

When $r=2$, Theorem~\ref{thm:main} specializes to the results of Batyrev and Blume.  For larger values of $r$, the theorem indicates that the spaces $\L^r_n$ are unions of toric varieties with independent torus actions, and in this way they occupy a middle ground between toric varieties and the classical moduli spaces of genus-zero curves---a setting in which the applicability of polyhedral methods is an intriguing new avenue for study. 

Let us illustrate the statement of Theorem~\ref{thm:main} with two special cases, to give the reader a flavor of the combinatorics involved.

\begin{example}
Let $r=2$ and $n=2$, which is part of Batyrev--Blume's work.  Then $\Delta^2_2$ is the octagon in $\R^2$ with vertices $(\pm 2, \pm 1)$ and $(\pm 1, \pm 2)$, and the combinatorial content of Theorem~\ref{thm:main} is that we can label the faces of $\Delta^2_2$ in two different ways, both of which are dimension-preserving and inclusion-preserving and are illustrated in Figures~\ref{fig:Delta22example-boundary} and ~\ref{fig:Delta22example-cosets}.

\begin{figure}[h!]
\begin{subfigure}{\textwidth}
    \centering
    \includegraphics[scale=1.1]{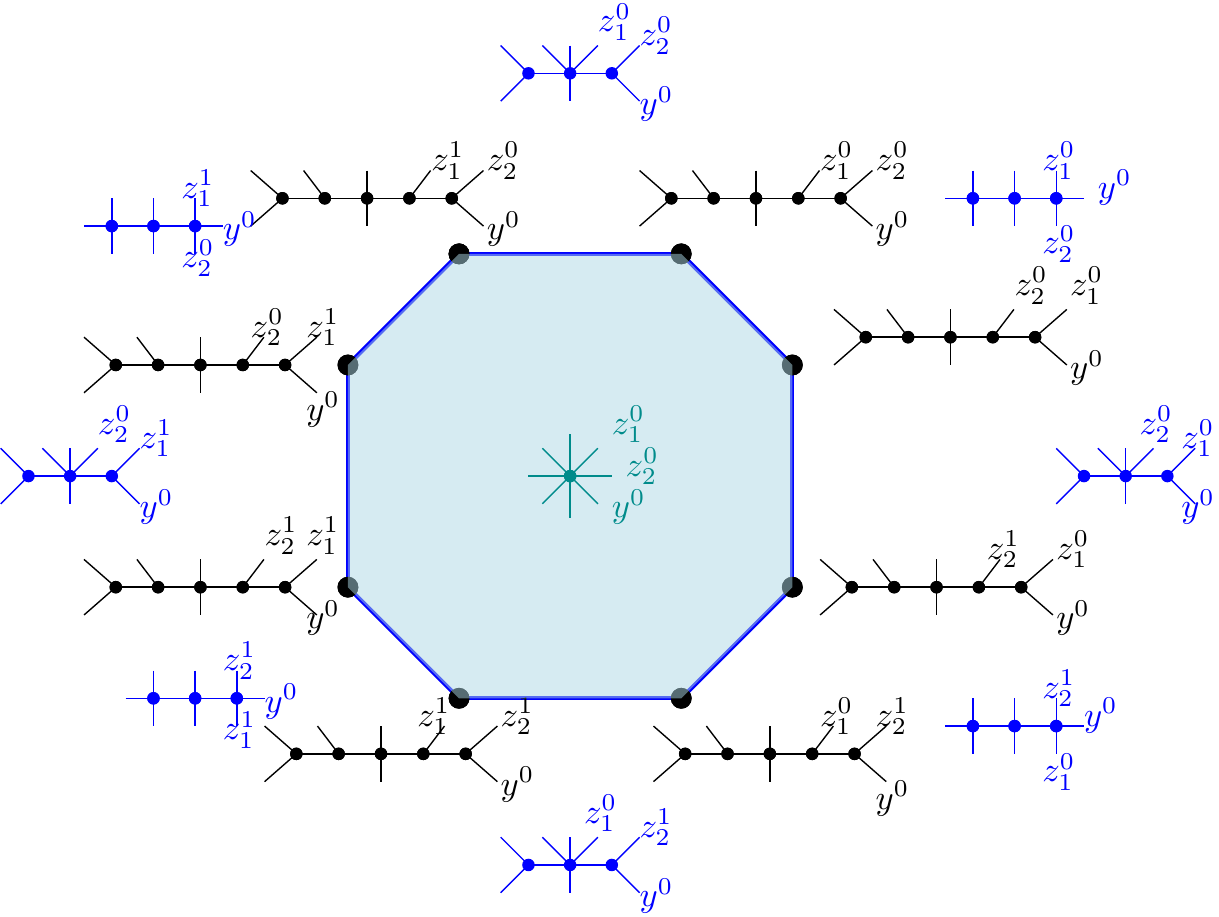}
    \caption{The polytope $\Delta^{2}_2$ with faces labeled by boundary strata in $\overline{\calL}^{2}_2$. The markings on the left side of each stratum are omitted for clarity but are uniquely determined by the involution.}
    \label{fig:Delta22example-boundary}
\end{subfigure}
\begin{subfigure}{\textwidth}
    \centering
    \hspace{-1.25cm}\includegraphics[scale=1.1]{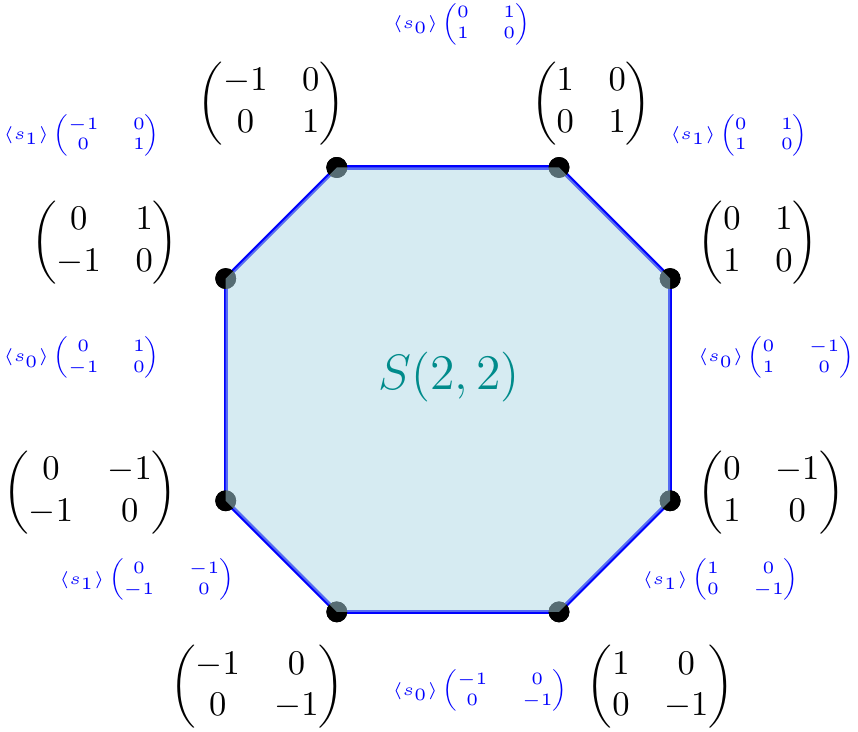}
    \caption{The polytope $\Delta^{2}_2$ with faces labeled by $\T$-cosets in $S(2, 2)$.}
    \label{fig:Delta22example-cosets}
\end{subfigure}
\caption{The example of $r=2$ and $n=2$. }
\label{fig:Delta22example}
\end{figure}

First, in Figure~\ref{fig:Delta22example-boundary}, we label each face of $\Delta^2_2$ by a boundary stratum in $\L^2_2$.  Such a stratum is described by an odd-length chain of rational curves---depicted by its dual graph in the figure---where the involution reflects across the central component, together with two marked fixed points of the involution (on the central component) as well as light orbits $(z_1^0, z_1^1)$ and $(z_2^0, z_2^1)$ and a heavy orbit $(y^0, y^1)$.

Second, as illustrated in Figure~\ref{fig:Delta22example-cosets}, each face of $\Delta^2_2$ can be labeled by a $\T$-coset in $S(2,2)$.  In this case, we have $\T = \{s_0, s_1\}$,
where
\[s_0=\begin{pmatrix} -1 & 0 \\
    0 & 1 \end{pmatrix} \text{ and } s_1=\begin{pmatrix} 0 & 1 \\
    1 & 0\end{pmatrix}.\]
The $0$-dimensional $\T$-cosets (which label the vertices) are the elements of $S(2,2)$, while the $1$-dimensional $\T$-cosets are the right cosets in $S(2,2)$ of subgroups generated by a single element of $\T$.  Since the two elements of $\T$ together generate $S(2,2)$, there is only a single $2$-dimensional $\T$-coset, which is the entire group and labels the unique $2$-dimensional face.
\end{example}

\begin{example}
Now, let $r=3$ and $n=2$, which is a new case in the current work.  In the previous example, the intersection of $\Delta^2_2 \subseteq \R^2$ with each quadrant is a pentagon, and these pentagons meet in pairs when a coordinate changes sign.  When $r=3$, by contrast, the polytopal complex $\Delta^3_2$ is a subset of
\[\Big(\mathbb{R}^{\geq 0} \cdot \{1, \zeta, \zeta^2\}\Big) \times \Big(\mathbb{R}^{\geq 0} \cdot \{1, \zeta, \zeta^2\}\Big)  \subseteq \C^2\]
where $\zeta$ is a primitive third root of unity.  Its intersection with each of the subsets
\[\Big(\mathbb{R}^{\geq 0} \cdot \zeta^a\Big) \times \Big(\mathbb{R}^{\geq 0} \cdot \zeta^b\Big) \subseteq \C^2\]
for $a,b \in \{0,1,2\}$ is a pentagon, and these nine pentagons meet in triples when $a$ or $b$ changes.  This complex is depicted in Figure~\ref{fig:Delta32example}, where we illustrate the statement of Theorem~\ref{thm:main} again by labeling some of the $\Delta$-faces in two ways: first by the boundary strata in $\L^3_2$, which are described by marked trees of rational curves with $\mu_3$-symmetry, and second by $\T$-cosets in $S(3,2)$.

\begingroup%
  \makeatletter%
  \providecommand\color[2][]{%
    \errmessage{(Inkscape) Color is used for the text in Inkscape, but the package 'color.sty' is not loaded}%
    \renewcommand\color[2][]{}%
  }%
  \providecommand\transparent[1]{%
    \errmessage{(Inkscape) Transparency is used (non-zero) for the text in Inkscape, but the package 'transparent.sty' is not loaded}%
    \renewcommand\transparent[1]{}%
  }%
  \providecommand\rotatebox[2]{#2}%
  \ifx\svgwidth\undefined%
    \setlength{\unitlength}{320.64581924bp}%
    \ifx\svgscale\undefined%
      \relax%
    \else%
      \setlength{\unitlength}{\unitlength * \real{\svgscale}}%
    \fi%
  \else%
    \setlength{\unitlength}{\svgwidth}%
  \fi%
  \global\let\svgwidth\undefined%
  \global\let\svgscale\undefined%
  \makeatother%
      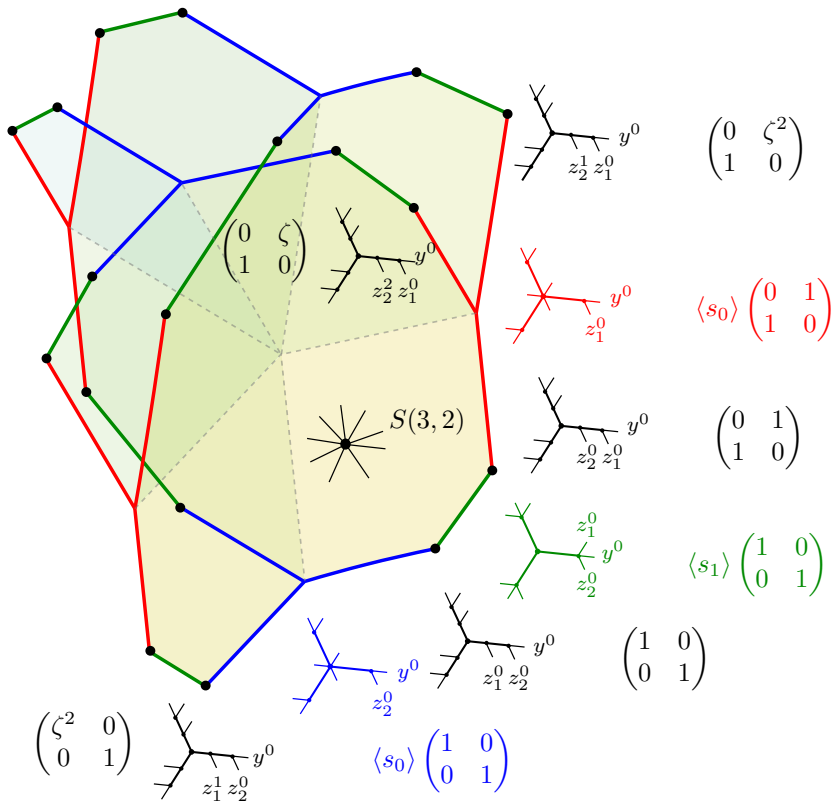
\begin{figure}[h!]
  \begin{picture}(1,1)%\begin{picture}(1,1.0166914)%
    \put(0,0){\includegraphics[width=\unitlength,page=1]{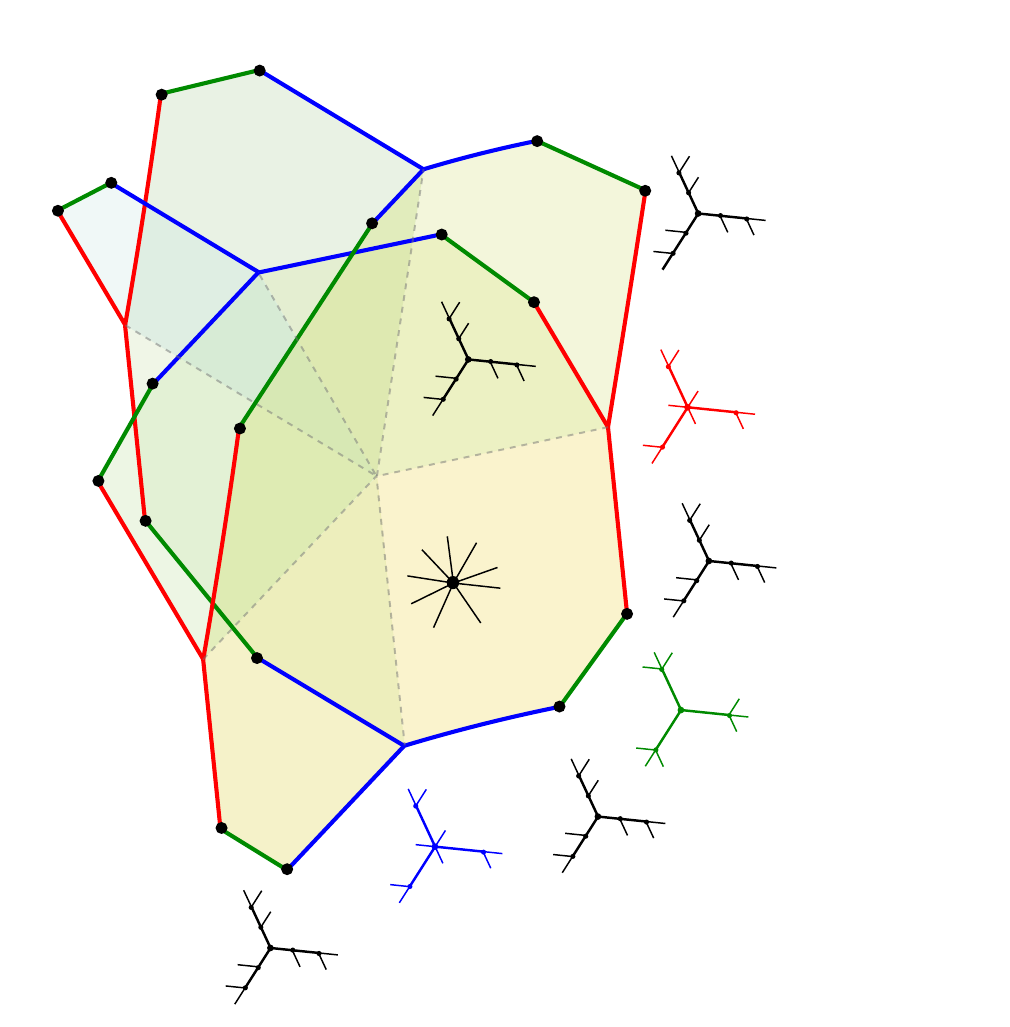}}%
    \put(0.5,0.45){\color[rgb]{0,0,0}\makebox(0,0)[lb]{\footnotesize{$S(3,2)$}}}%
    
    \put(0.76,0.60){\color[rgb]{1,0,0}\makebox(0,0)[lb]{\tiny{$y^0$}}}%
    \put(0.73,0.56){\color[rgb]{1,0,0}\makebox(0,0)[lb]{\tiny{$z^0_1$}}}%
    \put(0.86,0.56){\color[rgb]{1,0,0}\makebox(0,0)[lb]
    {\footnotesize{$\langle s_0 \rangle \begin{pmatrix}0 & 1 \\1 & 0 \end{pmatrix}$}}}%
    
    \put(0.78,0.45){\color[rgb]{0,0,0}\makebox(0,0)[lb]{\tiny{$y^0$}}}%
    \put(0.75,0.41){\color[rgb]{0,0,0}\makebox(0,0)[lb]{\tiny{$z^0_1$}}}%
    \put(0.72,0.41){\color[rgb]{0,0,0}\makebox(0,0)[lb]{\tiny{$z^0_2$}}}%
    \put(0.88,0.41){\color[rgb]{0,0,0}\makebox(0,0)[lb]{\footnotesize{$\begin{pmatrix}0 & 1 \\1 & 0 \end{pmatrix}$}}}%
    
    \put(0.67,0.19){\color[rgb]{0,0,0}\makebox(0,0)[lb]{\tiny{$y^0$}}}%
    \put(0.64,0.15){\color[rgb]{0,0,0}\makebox(0,0)[lb]{\tiny{$z^0_2$}}}%
    \put(0.61,0.15){\color[rgb]{0,0,0}\makebox(0,0)[lb]{\tiny{$z^0_1$}}}%
    \put(0.77,0.15){\color[rgb]{0,0,0}\makebox(0,0)[lb]{\footnotesize{$\begin{pmatrix}1 &0  \\0 &1 \end{pmatrix}$}}}%
    
    \put(0.51,0.16){\color[rgb]{0,0,1}\makebox(0,0)[lb]{\tiny{$y^0$}}}%
    \put(0.48,0.12){\color[rgb]{0,0,1}\makebox(0,0)[lb]{\tiny{$z^0_2$}}}%
    \put(0.48,0.03){\color[rgb]{0,0,1}\makebox(0,0)[lb]{\footnotesize{$\langle s_0 \rangle \begin{pmatrix}1 &0  \\0 &1 \end{pmatrix}$}}}%

    \put(0.34,0.06){\color[rgb]{0,0,0}\makebox(0,0)[lb]{\tiny{$y^0$}}}%
    \put(0.31,0.02){\color[rgb]{0,0,0}\makebox(0,0)[lb]{\tiny{$z^0_2$}}}%
    \put(0.28,0.02){\color[rgb]{0,0,0}\makebox(0,0)[lb]{\tiny{$z^1_1$}}}%
    \put(0.08,0.05){\color[rgb]{0,0,0}\makebox(0,0)[lb]{\footnotesize{$\begin{pmatrix}\zeta^2 &0  \\0 &1 \end{pmatrix}$}}}%

    \put(0.77,0.79){\color[rgb]{0,0,0}\makebox(0,0)[lb]{\tiny{$y^0$}}}%
    \put(0.74,0.75){\color[rgb]{0,0,0}\makebox(0,0)[lb]{\tiny{$z^0_1$}}}%
    \put(0.71,0.75){\color[rgb]{0,0,0}\makebox(0,0)[lb]{\tiny{$z^1_2$}}}%
    \put(0.87,0.75){\color[rgb]{0,0,0}\makebox(0,0)[lb]{\footnotesize{$\begin{pmatrix}0 & \zeta^2  \\ 1 & 0 \end{pmatrix}$}}}%

    \put(0.53,0.65){\color[rgb]{0,0,0}\makebox(0,0)[lb]{\tiny{$y^0$}}}%
    \put(0.51,0.605){\color[rgb]{0,0,0}\makebox(0,0)[lb]{\tiny{$z^0_1$}}}%
    \put(0.48,0.605){\color[rgb]{0,0,0}\makebox(0,0)[lb]{\tiny{$z^2_2$}}}%
    \put(0.30,0.63){\color[rgb]{0,0,0}\makebox(0,0)[lb]{\footnotesize{$\begin{pmatrix}0 & \zeta  \\1 & 0 \end{pmatrix}$}}}%
    
    \put(0.75,0.30){\color[rgb]{0,0.54117647,0}\makebox(0,0)[lb]{\tiny{$y^0$}}}%
    \put(0.72,0.33){\color[rgb]{0,0.54117647,0}\makebox(0,0)[lb]{\tiny{$z^0_1$}}}%
    \put(0.72,0.26){\color[rgb]{0,0.54117647,0}\makebox(0,0)[lb]{\tiny{$z^0_2$}}}%
    \put(0.85,0.26){\color[rgb]{0,0.54117647,0}\makebox(0,0)[lb]{\footnotesize{$\langle s_1 \rangle \begin{pmatrix}1 &0  \\0 &1 \end{pmatrix}$}}}%
      \end{picture}%
        \caption{A projection of the polytope $\Delta^3_2$, with some $\Delta$-faces labeled by both the boundary strata in $\L^3_2$ and by the corresponding $\T$-cosets in $S(3,2)$.}
   \label{fig:Delta32example}
   \end{figure}
\endgroup%

The precise definition of $\Delta$-faces is given in Definition~\ref{def:face}, but, for now, we simply remark that they are themselves polytopal complexes; in particular, the $0$-dimensional $\Delta$-faces are the vertices labeled in black, the the $1$-dimensional $\Delta$-faces are the line segments labeled in green as well as the $\mathsf{Y}$-shapes labeled in red and blue, and there is only a single $2$-dimensional $\Delta$-face, which is the entire complex $\Delta^3_2$.  The key observation is that a $1$-dimensional $\Delta$-face may be adjacent to either two or three $0$-dimensional $\Delta$-faces.  This corresponds precisely to a fact about boundary strata and about $\T$-cosets:
\begin{itemize}
\item in $\L^3_2$, a $1$-dimensional boundary stratum contains either two or three different $0$-dimensional boundary strata, depending on whether there is a light orbit on the central component;
\item in $S(3,2)$, the $1$-dimensional $\T$-cosets $\langle t \rangle \cdot A$ have either two or three elements, depending on whether $t$ is an adjacent transposition or $t=s_0$. 
\end{itemize}
These are special cases of the combinatorial phenomena that arise more generally in what follows.
\end{example}

\subsection*{Future and related work} In future work, we hope to probe further geometric and combinatorial properties of $\L^r_n$ to more fully exploit its proximity to being a toric variety.  For example, one could study the tropical manifestation of $\L^n_r$ (along the lines of the work carried out in \cite{Chan_2021, Ulirsch_2015, cavalieri2016moduli} for $\M_{g,n}$ and Hassett spaces), its symmetries (along the lines of \cite{massarenti2014, massarenti2015automorphisms}), or the combinatorial structure of its Chow ring (along the lines of \cite{keel1992intersection, bergstrom2013cohomology,bergstrom2014cohomology, Kannan_2021}). The birational geometry of $\L^r_n$ would also be very interesting to study.  In the case of $\M_{0,n}$, the parallelism with toric varieties motivated Fulton's famous F-conjecture as well as the (now disproven) conjecture that the Cox ring of $\M_{0,n}$ is finitely-generated \cite{CastravetTevelev, GonzalezKaru}.  Perhaps the fact that $\L^r_n$ is ``more toric" than $\M_{0,n}$---in that it is combinatorially encoded by a polyhedral object---would make these birational-geometric questions more amenable to study in this setting.

\subsection*{Plan of the paper}  In Section~\ref{sec:modulispace}, we precisely define the objects parameterized by the moduli space $\L^r_n$.  The fact that there indeed exists a fine moduli space parameterizing these objects is the content of Section~\ref{sec:construction}; readers wishing to accept the existence of $\L^r_n$ are encouraged to skip that section and proceed directly to the combinatorial material that follows.  Section~\ref{sec:chains} defines the combinatorial data of decorated nested chains of subsets of $\{1,\ldots, n\}$, and Sections \ref{sec:boundarystrata}, \ref{sec:S(r,n)}, and \ref{sec:permutohedral-complex} show that this data can be used to describe the boundary strata, $\T$-cosets, and $\Delta$-faces, respectively.  Finally, in Section~\ref{sec:conclusions}, we combine the results of the previous three sections to deduce Theorem~\ref{thm:main}, and we observe that the boundary strata, $\T$-cosets, and $\Delta$-faces also have product decompositions and $S(r,n)$-actions that are respected by the bijections between them.

\subsection*{Acknowledgments}  The authors would like to thank Dusty Ross for the conversations that inspired the project, and ICERM for hosting the workshop ``Women in Algebraic Geometry" at which the collaboration began.  E.C. was supported by NSF DMS grant 1810969, D.H. was supported by NSF DMS grant 1439786, and R.R. was supported by NSF DMS grant 1703308.

\section{The moduli space}
\label{sec:modulispace}

Fix integers $r \geq 2$ and $n \geq 0$.  Denote by $\mu_r \subseteq \C^*$ the cyclic group of $r$th roots of unity, and denote $\Z_r = \{0,1, \ldots, r-1\}$.  

\subsection{Objects and families}
\label{sec:defLrn}

We begin by specifying the underlying curves of the objects we are interested in parameterizing.  Throughout, varieties are considered over the field $\C$. 

\begin{definition}
\label{def:pinwheelcurve}
An {\bf $r$-pinwheel curve} is a tree of projective lines meeting at nodes, consisting of a central projective line from which $r$ equal-length chains of projective lines (``spokes") emanate.  If each of these spokes has $k$ components, we say that the pinwheel curve has {\bf length} $k$; in the case where $k=0$, the curve is simply $\P^1$.
\end{definition}

The objects of our moduli space are built from $r$-pinwheel curves as follows.

\begin{definition}
An {\bf $(r,n)$-curve} consists of the following data:
\begin{itemize}
    \item an $r$-pinwheel curve $C$;
    \item an order-$r$ automorphism $\sigma: C \rightarrow C$;
    \item two distinct fixed points $x^+$ and $x^-$ of $\sigma$;
    \item $n$ labeled $r$-tuples $(z_1^{0}, \ldots, z_1^{r-1}), \ldots, (z_n^{0}, \ldots, z_n^{r-1})$ of points $z_i^j \in C$ satisfying
    \[\sigma(z_i^j) = z_i^{j+1 \!\!\!\!\mod r}\]
    for each $i$ and $j$, where we allow that $z_i^j = z_{i'}^{j'}$ and that $z_i^j = x^{\pm}$;
    \item an additional labeled $r$-tuple $(y^{0}, \ldots, y^{r-1})$ satisfying
    \[\sigma(y^\ell) = y^{\ell+1 \!\!\!\!\mod r}\]
    for each $\ell$, whose elements are distinct from one another as well as from $x^{\pm}$ and $z_i^j$.
\end{itemize}

\end{definition}

We refer to an $(r,n)$-curve as {\bf stable} if each irreducible component of $C$ contains at least two ``heavy" points---where the ``heavy" points are the half-nodes and the points $y^{\ell}$---and any irreducible component with exactly two heavy points contains at least one of the ``light" points $x^{\pm}$ or $z_i^j$.  (This is a special case of the stability condition for Hassett spaces, which will play a major role in the construction of the moduli space $\overline{\calL}^{r}_n$ in Section \ref{sec:construction} below.)

It is straightforward to see that the stability condition forces $y^0, \ldots, y^{r-1}$ to lie on the $r$ leaves of the pinwheel curve, in which case $\sigma$ must consist of a rotation taking each spoke of the pinwheel to another, and $x^{\pm}$ must both lie on the central component, which we denote $C_{\bullet}$; see Figure~\ref{fig:elementofLrn}.  

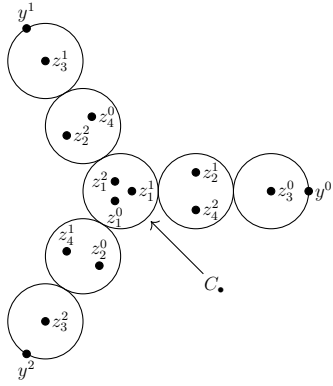
\begin{figure}[h!]
\begin{center}
    \begin{tikzpicture}[scale=1,  every node/.style={scale=0.6}]
    
    \filldraw (0.15,0) circle (0.05cm) node[right]{$z^1_1$};
    \filldraw (-0.075,0.1299) circle (0.05cm) node[left]{$z^2_1$};
    \filldraw (-0.075,-0.1299) circle (0.05cm) node[below]{$z^0_1$};
    
    \filldraw(1,0.25) circle (0.05cm) node[right]{$z^1_2$};
    \filldraw(-0.7165, 0.741) circle (0.05cm) node[right]{$z^2_2$};
    \filldraw(-0.283, -0.991) circle (0.05cm) node[above]{$z^0_2$};
    
    \filldraw(1,-0.25) circle (0.05cm) node[right]{$z^2_4$};
    \filldraw(-0.383,0.991) circle (0.05cm) node[right]{$z^0_4$};
    \filldraw(-0.717,-0.8) circle (0.05cm) node[above]{$z^1_4$};
    
    \filldraw (2,0) circle (0.05cm) node[right]{$z^0_3$};
    \filldraw (-1,1.732) circle (0.05cm) node[right]{$z^1_3$};
    \filldraw (-1,-1.732) circle (0.05cm) node[right]{$z^2_3$};
    
    \draw (0,0) circle (0.5cm);
    \draw (1,0) circle (0.5cm);
    \draw (2,0) circle (0.5cm);
    \filldraw (2.5,0) circle (0.05cm) node[right]{$y^0$};

    \draw (-0.5,0.866) circle (0.5cm);
    \draw (-1,1.732) circle (0.5cm);
    \filldraw (-1.25,2.165) circle (0.05cm) node[above]{$y^1$};

    \draw (-0.5,-0.866) circle (0.5cm);
    \draw (-1,-1.732) circle (0.5cm);
    \filldraw (-1.25,-2.165) circle (0.05cm) node[below]{$y^2$};
    
    \draw[->] (1.1,-1.1) -- (0.4, -0.4);
    \node at (1.25, -1.25) {$C_{\bullet}$};
    \end{tikzpicture}
\end{center}
\caption{A stable length-two $(4,3)$-curve, where each circle represents a $\P^1$ component and $\sigma$ is the rotational automorphism.  Not pictured are the marked points $x^+$ and $x^-$, which are the two fixed points of $\sigma$ and must both lie on the central component $C_{\bullet}$.}
\label{fig:elementofLrn}
\end{figure}

Up to an automorphism of $C$, one can assume that $x^+$ and $x^-$ are the points $\infty$ and $0$, respectively, in the central component $C_{\bullet} \cong \P^1$, and that the node at which the $y^0$-spoke meets $C_{\bullet}$ is the point $1 \in C_{\bullet} \cong \P^1$.  Under this identification, the fact that $\sigma$ has order $r$ ensures that the $y^1$-spoke meets $C_{\bullet}$ at $\zeta \in \P^1$, where $\zeta$ is a primitive $r$th root of unity.  Once $\zeta$ is chosen, the points at which the remaining spokes meet $C_{\bullet}$ are determined, but $\zeta$ itself can be freely chosen to be any primitive $r$th root of unity.  We encode this choice in the following terminology.

\begin{definition}
Let $\zeta$ be a primitive $r$th root of unity.  Given a stable $(r,n)$-curve, let $p^\ell \in C_{\bullet}$ be the point at which the $y^\ell$-spoke meets the central component, for each $\ell \in \Z_r$.  We say that the curve has {\bf type $\zeta$} if, under the unique automorphism of the central component $C_{\bullet}$ that sends
\[x^+ \mapsto \infty,\;\; x^- \mapsto 0, \;\; p^0 \mapsto 1,\]
we have
\[p^\ell \mapsto \zeta^\ell\] for all $\ell \in \Z_r$.
\end{definition}

\begin{remark}
A stable $(r,n)$-curve of type $\zeta$ can be viewed as a curve with an action of the cyclic group $\mu_r$, in which the generator $\zeta \in \mu_r$ acts by the automorphism $\sigma$.
\end{remark}

Having defined the objects of interest, we now specify the notions of family and morphism of families.

\begin{definition}
A {\bf family of stable $(r,n)$-curves over a base scheme $B$} is a flat, proper morphism $\pi: \mathcal{C} \rightarrow B$ equipped with an order-$r$ automorphism $\sigma$ of $\mathcal{C}$ such that $\pi \circ \sigma = \pi$, and sections $x^{\pm}, \{z_i^j\}$, and $\{y^\ell\}$ of $\pi$ such that for any geometric point $b \in B$, the fiber
\[\bigg( \pi^{-1}(b); \sigma\big|_{\pi^{-1}(b)}; x^{\pm}(b), \{y^\ell(b)\}, \{z_i^j(b)\}\bigg)\]
is a stable $(r,n)$-curve.  If, furthermore, each fiber has type $\zeta$, we say that the family has type $\zeta$.
\end{definition}

\begin{remark}
If the base $B$ is connected, then every fiber has the same type, so the type of the family can be deduced from considering any single fiber.
\end{remark}

\begin{definition}
Given families
\[(\pi: \mathcal{C} \rightarrow B; \sigma; x^{\pm}, \{y^\ell\}, \{z_i^j\}) \, \text{ and } \, (\pi': \mathcal{C}' \rightarrow B; \sigma'; X^{\pm}, \{Y^\ell\}, \{Z_i^j\})\]
over the same base $B$, a {\bf morphism} of families is a morphism $s: \mathcal{C} \rightarrow \mathcal{C}'$ satisfying
\begin{itemize}
    \item $\pi' \circ s = \pi$;
    \item $\sigma' \circ s = s \circ \sigma$;
    \item $X^{\pm} = s \circ x^{\pm}$, $Z_i^j = s \circ z_i^j$, and $Y^\ell = s \circ y^\ell$.
\end{itemize}
\end{definition}

The goal of Section~\ref{sec:construction} is to identify a fine moduli space representing the moduli problem specified above, which we denote as follows.

\begin{definition} \label{def:Lrnzeta}
For any integers $r \geq 2$ and $n \geq 0$, we denote by $\L^r_n(\zeta)$ the moduli space of isomorphism classes of stable $(r,n)$-curves, and we denote by $\L^r_n = \bigsqcup_\zeta \L^r_n(\zeta)$ the moduli space of all isomorphism classes of stable $(r,n)$-curves.
\end{definition}

The reader willing to accept the existence of such a fine moduli space may wish to skip Section~\ref{sec:construction} entirely and proceed directly to the combinatorics in Section ~\ref{sec:chains}, and they are encouraged to do so.  First, however, we must describe the boundary strata in $\L^r_n$, which are critical to the combinatorics that follow.

\subsection{Boundary strata}
\label{subsec:boundarystrata}

In any moduli space of curves, a boundary stratum is defined as the closure of the locus of curves of a fixed topological type.  More precisely, we have the following definition of boundary strata in our case.

\begin{definition}
\label{def:boundarystratum}
Any $(C; \sigma; x^{\pm}, \{y^\ell\}, \{z_i^j\}) \in \L^r_n(\zeta)$ has an associated {\bf dual graph} consisting of
\begin{itemize}
    \item a vertex $v$ for each irreducible component $C_v$ of $C$;
    \item an edge between vertices $v$ and $w$ if the corresponding irreducible components $C_v$ and $C_w$ meet at a node;
    \item a half-edge attached to the vertex $v$ for each marked point on $C_v$, labeled by the name $x^{\pm}, y^\ell$, or $z_i^j$ of the marked point.
\end{itemize}
Given such a dual graph $G$, the corresponding {\bf boundary stratum} $S_G \subseteq \L^r_n(\zeta)$ is defined as the closure of the set of curves with dual graph $G$.
\end{definition}

In particular, passing to the closure means that $S_G$ includes also degenerations of curves with dual graph $G$; for example, see Figure~\ref{fig:dualgraph}.  It follows that one can detect in terms of dual graphs when there is an inclusion of boundary strata: we have $S_G \subseteq S_H$ if and only if $H$ can be obtained from $G$ by edge-contraction of some subset of the edges, a procedure that corresponds geometrically to degeneration of a curve in $S_H$ to a curve in $S_G$.  For example, if $G$ is the top dual graph in Figure~\ref{fig:dualgraph} and $G_1$ and $G_2$ are the two dual graphs depicted below it, then we have $S_G \subseteq S_{G_1}$ and $S_G \subseteq S_{G_2}$, corresponding to the fact that both $G_1$ and $G_2$ can be obtained by edge-contraction from $G$.

\begin{figure}[h]
    \centering
    \includegraphics[scale=0.8]{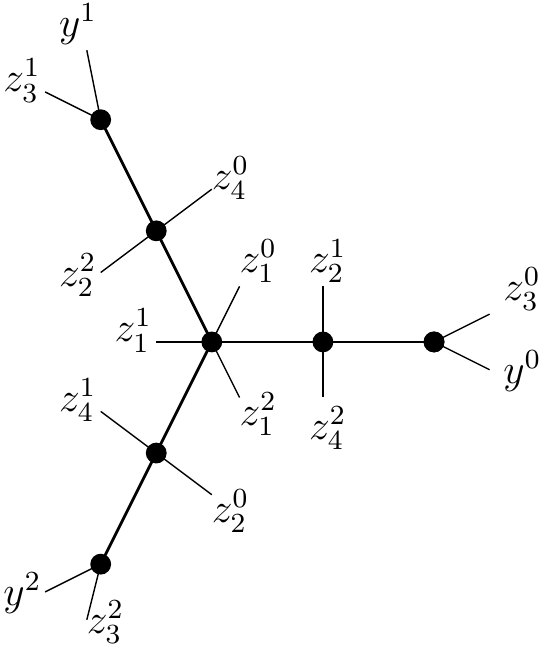}
\begin{subfigure}{.45\textwidth}
  \centering
  \includegraphics[scale=1]{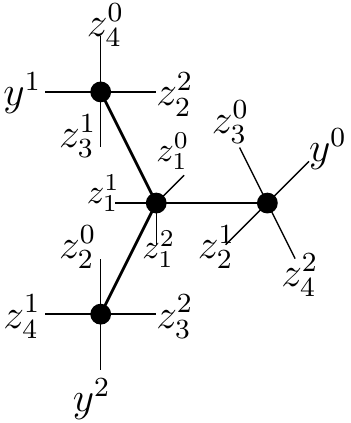}
\end{subfigure}%
\begin{subfigure}{.45\textwidth}
  \centering
  \includegraphics[scale=1]{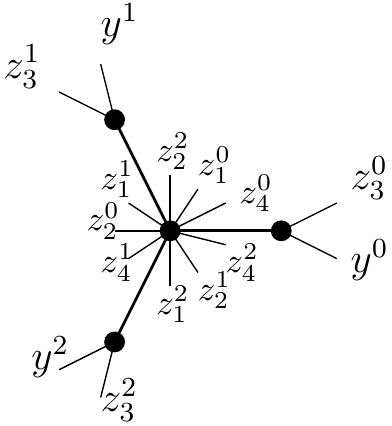}
\end{subfigure}
    \caption{The dual graph $G$ of the curve in Figure~\ref{fig:elementofLrn}, and below it, the dual graphs of two other curves in the boundary stratum $S_G$, corresponding to two edge-contractions of $G$. }
    \label{fig:dualgraph}
\end{figure}

\section{Construction of the moduli space}
\label{sec:construction}

We construct the moduli space of stable $(r,n)$-curves as a closed subscheme of a more well-known moduli space constructed by Hassett in \cite{HASSETT2003316}, so we begin by recalling the necessary definitions from the theory of Hassett spaces.  Throughout, we denote
\[[n] = \{1,2,\ldots, n\}\]
and
\[[n]_0 = \{0, 1, \ldots, n\}.\]

\subsection{Hassett spaces and maps between them}
\label{subsec:hassett-maps}

For any $g \geq 0$, $n \ge 0$ and any weight vector ${\w = (w_1, \ldots, w_n) \in (\Q \cap (0, 1])^{n}}$ such that $2 g + \sum_{i = 1}^{n} w_i > 2$, the associated Hassett space is a smooth Deligne--Mumford stack $\M_{g, \w}$ that is an alternate modular compactification of the moduli space $\calM_{g, n}$ of smooth projective curves of genus $g$ with $n$ distinct marked points.

Specifically, elements of $\overline{\calM}_{g, \w}$ are tuples $(C; q_1, \ldots, q_n)$, where $C$ is a projective curve of arithmetic genus $g$ and only nodes as singularities, and $q_1, \ldots, q_n \in C$ are marked points with weights $w_1, \ldots, w_n$, satisfying the following two conditions:
\begin{itemize}
    \item the sum of the weights of any collection of coinciding marked points is at most $1$;
    \item for each irreducible component $D$ of $C$, if $\{q_i\}_{i \in I_D}$ for $I_D \subseteq [n]$ are the marked points of $D$ and $n_D$ is the number of half-nodes in $D$, then
    \[2g-2+n_D + \sum_{i \in I_D} w_i >0.\]
\end{itemize}
We refer to elements of $\overline{\calM}_{g, \w}$ as {\bf $\w$-stable curves}.  Note that the usual moduli space of curves $\M_{g,n}$ is recovered by taking $w_i=1$ for all $i \in [n]$.

The special case of this construction that is relevant for the current work is when $g = 0$ and the weight vector is
\begin{equation}
\label{eq:w}
\w = 
\left(\frac{1}{2}+ \varepsilon,  \frac{1}{2} + \varepsilon , \underset{\text{\footnotesize $r$}}{\underbrace{1, \dots, 1}},  \underset{\text{\footnotesize $nr$}}{\underbrace{\varepsilon,   \dots,  \varepsilon}}\right)
\end{equation}
for any fixed $0 < \varepsilon < 1/(nr+2)$. Let
\[\overline{\calM}^r_n := \overline{\calM}_{0, \w}\]
denote the Hassett space with this weight vector, which is a smooth projective scheme.  Let $\calC^r_n$ denote the universal curve over $\M^r_n$.  Suggestively, we denote the marked points of $\M^r_n$ with weights $1/2 + \varepsilon$ by $x^+$ and $x^-$, the $r$ marked points with weight $1$ by $y^0, \ldots, y^{r-1}$, and the $nr$ marked points with weight $\varepsilon$ by $z_i^j$ with $i\in [n]$ and $j \in \Z_r$.

\begin{remark}\label{lem:M1nobjects}  It is possible to let $r = 1$ in this construction, in which case $\M^1_n$ is the Hassett space with weight vector $\w = (1/2 + \varepsilon, 1/2 + \varepsilon, 1, \varepsilon, \ldots, \varepsilon)$.  We omit the $\Z_r$-superscripts on the marked points in this case, so an element of $\M^1_n$ is denoted $(C;x^+,x^-,y,z_1,\ldots, z_n)$.  Here, the dual graph of $C$ is a chain with $y$ on one leaf and (if $C$ is reducible) $x^+$ and $x^-$ on the unique other leaf; see an example of an element in $\overline{\calM}^{1}_2$ in Figure \ref{fig:M1n-bar}.  While some readers may recognize such chains of projective lines as the underlying curves of the elements in Losev--Manin space, we stress that $\overline{\calM}^{1}_n$ is {\it not} Losev--Manin space, which can instead be described as $\overline{\calM}_{0,\w'}$ with $\w'=\left(1,1, \varepsilon, \ldots, \varepsilon \right)$.
\end{remark}

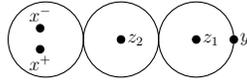
\begin{figure}[h!]
    \begin{center}
    \begin{tikzpicture}[scale=1,  every node/.style={scale=0.6}]
    \filldraw (-0.075,-0.1299) circle (0.05cm) node[below]{$x^{+}$};
    
    \filldraw (-0.075,0.15) circle (0.05cm) node[above]{$x^{-}$};

    \filldraw(1,0) circle (0.05cm) node[right]{$z_2$};
    
    \draw (0,0) circle (0.5cm);
    \draw (1,0) circle (0.5cm);
    \draw (2,0) circle (0.5cm);
    \filldraw (2,0) circle (0.05cm) node[right]{$z_1$};
    \filldraw (2.5,0) circle (0.05cm) node[right]{$y$};
    \end{tikzpicture}
    \end{center}
    \caption{A typical curve in $\overline{\calM}^{1}_2$. }
    \label{fig:M1n-bar}
    \end{figure}

When $r\ge2$, the choice of the weight vector $\w$ ensures that the $r$-pinwheel curves $C$ that underlie stable $(r,n)$-curves are elements of $\M^r_n$.  However, not every element of $\M^r_n$ is such a curve.  Thus, the goal of the next subsection is to identify a closed subscheme
\[\L^r_n \subseteq \M^r_n\]
that is a fine moduli space for stable $(r,n)$-curves.  The definition of this subscheme involves two families of morphisms between Hassett spaces, which we now describe.

First, for every map $\alpha \colon [n]_0  \to \mathbb{Z}_r$, there is a morphism
\[\pi_\alpha \colon \M^r_n \to \M^1_n\] given on $\mathbb{C}$-points by
\[\pi_\alpha (C; x^{\pm}, \{y^\ell\}, \{z_i^j\}) = (C; x^{\pm}, y^{\alpha(0)}, z_1^{\alpha(1)}, \dots, z_n^{\alpha(n)}).\] That is, $\pi_\alpha$ forgets the marked points not in $\{x^{\pm}, y^{\alpha(0)}, z_1^{\alpha(1)}, \ldots, z_{n}^{\alpha(n)}\}$ and contracts any resulting unstable components of $C$; see Figure \ref{fig:pi-alpha}.

\begin{center}
\begin{figure}[h!]
\centering
\begin{tikzpicture}[scale=1,  every node/.style={scale=0.6}]
    
    \filldraw[gray, opacity=0.5] (0.15,0) circle (0.05cm) node[right]{$z^1_1$};
    \filldraw[gray, opacity=0.5] (-0.075,0.1299) circle (0.05cm) node[left]{$z^2_1$};
    \filldraw (-0.075,-0.1299) circle (0.05cm) node[below]{$z^0_1$};
    
    \filldraw[gray, opacity=0.5](1,0.25) circle (0.05cm) node[right]{$z^1_2$};
    \filldraw(-0.7165, 0.741) circle (0.05cm) node[right]{$z^2_2$};
    \filldraw[gray, opacity=0.5](-0.283, -0.991) circle (0.05cm) node[above]{$z^0_2$};
    
    \filldraw(1,-0.25) circle (0.05cm) node[right]{$z^2_4$};
    \filldraw[gray, opacity=0.5](-0.383,0.991) circle (0.05cm) node[right]{$z^0_4$};
    \filldraw[gray, opacity=0.5](-0.717,-0.8) circle (0.05cm) node[above]{$z^1_4$};
    
    \filldraw[gray, opacity=0.5] (2,0) circle (0.05cm) node[right]{$z^0_3$};
    \filldraw (-1,1.732) circle (0.05cm) node[right]{$z^1_3$};
    \filldraw[gray, opacity=0.5] (-1,-1.732) circle (0.05cm) node[right]{$z^2_3$};
    
    \draw (0,0) circle (0.5cm);
    \draw (1,0) circle (0.5cm);
    \draw (2,0) circle (0.5cm);
    \filldraw (2.5,0) circle (0.05cm) node[right]{$y^0$};
    
    \draw[->] (3.5, 0) -- node[above] {$\pi_{\alpha}$}  (4.3,0);
    
    \draw(-0.5,0.866) circle (0.5cm);
    \draw (-1,1.732) circle (0.5cm);
    \filldraw[gray, opacity=0.5] (-1.25,2.165) circle (0.05cm) node[above]{$y^1$};

    \draw (-0.5,-0.866) circle (0.5cm);
    \draw (-1,-1.732) circle (0.5cm);
    \filldraw[gray, opacity=0.5] (-1.25,-2.165) circle (0.05cm) node[below]{$y^2$};

    \end{tikzpicture}
\qquad
\begin{tikzpicture}[scale=1,  every node/.style={scale=0.6}]
    
    \filldraw[gray, opacity=0.5] (0.15,0) circle (0.05cm) node[right]{$z^1_1$};
    \filldraw[gray, opacity=0.5] (-0.075,0.1299) circle (0.05cm) node[left]{$z^2_1$};
    \filldraw (-0.075,-0.1299) circle (0.05cm) node[below]{$z^0_1$};
    
    \filldraw[gray, opacity=0.5](1,0.25) circle (0.05cm) node[right]{$z^1_2$};
    \filldraw(-0.7165, 0.741) circle (0.05cm) node[right]{$z^2_2$};
    \filldraw[gray, opacity=0.5](-0.283, -0.991) circle (0.05cm) node[above]{$z^0_2$};
    
    \filldraw(1,-0.25) circle (0.05cm) node[right]{$z^2_4$};
    \filldraw[gray, opacity=0.5](-0.383,0.991) circle (0.05cm) node[right]{$z^0_4$};
    \filldraw[gray, opacity=0.5](-0.717,-0.8) circle (0.05cm) node[above]{$z^1_4$};
    
    \filldraw[gray, opacity=0.5] (2,0) circle (0.05cm) node[right]{$z^0_3$};
    \filldraw (-1,1.732) circle (0.05cm) node[right]{$z^1_3$};
    \filldraw[gray, opacity=0.5] (-1,-1.732) circle (0.05cm) node[right]{$z^2_3$};
    
    \draw (0,0) circle (0.5cm);
    \draw(1,0) circle (0.5cm);
    \draw[fill=gray, opacity=0.1] (2,0) circle (0.5cm);
    \filldraw (2.5,0) circle (0.05cm) node[right]{$y^0$};

    \draw[fill=gray, opacity=0.1](-0.5,0.866) circle (0.5cm);
    \draw[fill=gray, opacity=0.1] (-1,1.732) circle (0.5cm);
    \filldraw[gray, opacity=0.5] (-1.25,2.165) circle (0.05cm) node[above]{$y^1$};

    \draw[fill=gray, opacity=0.1] (-0.5,-0.866) circle (0.5cm);
    \draw[fill=gray, opacity=0.1] (-1,-1.732) circle (0.5cm);
    \filldraw[gray, opacity=0.5] (-1.25,-2.165) circle (0.05cm) node[below]{$y^2$};

    \end{tikzpicture}
    \qquad
    \begin{tikzpicture}[scale=1,  every node/.style={scale=0.6}]  
    
    %\filldraw (0.15,0) circle (0.05cm) node[right]{$z^2_2$};
    \filldraw (-0.25,0.41) circle (0.05cm) node[above]{$z^1_3=z_2^2$};
    \filldraw (-0.075,-0.25) circle (0.05cm) node[right]{$z^0_1$};
    
    \filldraw[gray, opacity=0](1,0.25) circle (0.05cm) node[right]{$z^1_2$};
    %\filldraw(-0.7165, 0.741) circle (0.05cm) node[right]{$z^2_2$};
    \filldraw[gray, opacity=0](-0.283, -0.991) circle (0.05cm) node[above]{$z^0_2$};
    
    \filldraw(1,-0.25) circle (0.05cm) node[right]{$z^2_4$};
    \filldraw[gray, opacity=0](-0.383,0.991) circle (0.05cm) node[right]{$z^0_4$};
    \filldraw[gray, opacity=0](-0.717,-0.8) circle (0.05cm) node[above]{$z^1_4$};
    
    \filldraw[gray, opacity=0] (2,0) circle (0.05cm) node[right]{$z^0_3$};
    %\filldraw (-1,1.732) circle (0.05cm) node[right]{$z^1_3$};
    \filldraw[gray, opacity=0] (-1,-1.732) circle (0.05cm) node[right]{$z^2_3$};
    
    \draw (0,0) circle (0.5cm);
    \draw(1,0) circle (0.5cm);
    \draw[fill=gray, opacity=0] (2,0) circle (0.5cm);
    \filldraw (1.5,0) circle (0.05cm) node[right]{$y^0$};

    \draw[fill=gray, opacity=0](-0.5,0.866) circle (0.5cm);
    \draw[fill=gray, opacity=0] (-1,1.732) circle (0.5cm);
    \filldraw[gray, opacity=0] (-1.25,2.165) circle (0.05cm) node[above]{$y^1$};

    \draw[fill=gray, opacity=0] (-0.5,-0.866) circle (0.5cm);
    \draw[fill=gray, opacity=0] (-1,-1.732) circle (0.5cm);
    \filldraw[gray, opacity=0] (-1.25,-2.165) circle (0.05cm) node[below]{$y^2$};
        \end{tikzpicture}

\caption{A visualization of the map $\pi_\alpha$ in the case where $\alpha(0)=0$, $\alpha(1)=0$, $\alpha(2) = 2$, $\alpha(3)=1$, $\alpha(4)=2$. The marked points $x^{\pm}$ are on the central component and are not pictured. The faded marked points are forgotten under $\pi_{\alpha}$ and the faded solid circles indicate contracted components as a result of forgetting some of the marked points.}
\label{fig:pi-alpha}
\end{figure}
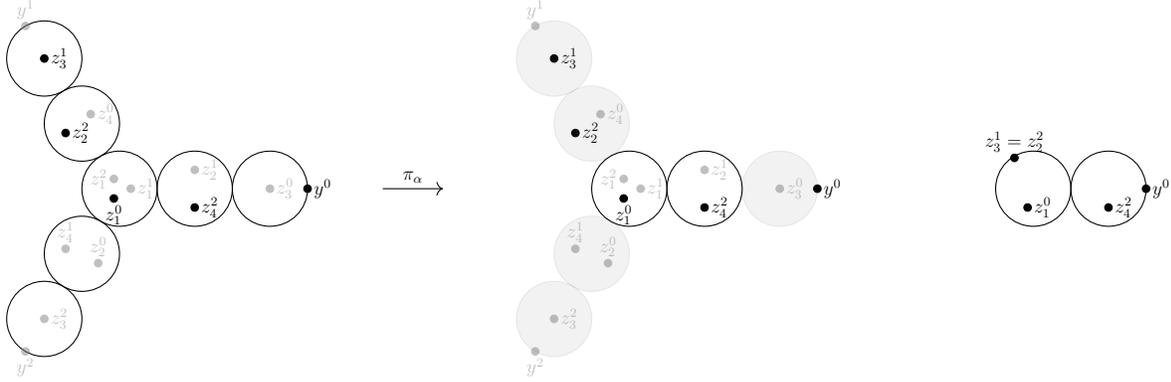
\end{center}

In fact, one can upgrade $\pi_{\alpha}$ from a morphism on $\C$-points to a morphism on families.  Namely, for any base scheme $B$ and any map $f \colon B \to \M^r_n$, there are associated maps \[f_\alpha \colon B \to \M^1_n \qquad \text{ and } \qquad \psi_{\alpha,f} \colon f^*{\mathcal{C}^r_n} \to f_\alpha ^*{\mathcal{C}^1_n}\]
satisfying the following conditions:
\begin{itemize}
\item If $f^*{\mathcal{C}^r_n}$ is marked by the sections $x^{\pm}, \{y^\ell\}, \{z_i^j\}$, then $f_\alpha^*\mathcal{C}^1_n$ is the stabilization of \[(f^*{\mathcal{C}^r_n}; x^{\pm}, y^{\alpha(0)}, z^{\alpha(1)}_1, \dots, z^{\alpha(n)}_n).\]
\item The map $\psi_{\alpha,f}$ is the contraction map, which contracts unstable components in each fiber.\end{itemize}
The key special case for what follows is when $B = \M^r_n$ and $f = \text{id}$, in which case $f_{\alpha} = \pi_{\alpha}$ and the map $\psi_{\alpha, \text{id}}$ is a birational morphism that we denote by $\psi_{\alpha}$ for simplicity. These maps then fit into the following diagram:
\begin{equation}
    \label{eq:def_psialpha}
\xymatrix{
\mathcal{C}^r_n \ar[d]\ar[r]^{\psi_{\alpha}} & \pi_{\alpha}^*\mathcal{C}^1_n \ar[dl]\ar[r] & \mathcal{C}^1_n\ar[d]\\
\M^r_n \ar[rr]^{\pi_{\alpha}} && \M^1_n.
}
\end{equation}

To define the other family of morphisms between Hassett spaces, we first note that there is a natural isomorphism between $\M_0^2$ and $\P^1$ given by associating $q \in \P^1$ to
\[(\P^1; \infty, 0, 1, q) \in \M_0^2.\]
Given this, for every $i \in \mathbb{Z}_r$, we can define a morphism
\[\lambda_i \colon \M^r_n \to \M^2_0 \cong \P^1\]
by sending the point $(C; x^{\pm}, \{y^\ell\}, \{z_i^j\}) \in \M^r_n$ to the element of $\P^1$ corresponding to the point $(C; x^+,x^-, y^{i},y^{i+1}) \in \M_0^2$.

As above, the morphism $\lambda_i$ can be described functorially: to any scheme $B$ and any map $f \colon B \to \M^r_n$, we associate maps \[f_i \colon B \to \M^2_0 \cong \P^1 \qquad \text{ and } \qquad \varphi_{i,f} \colon f^*{\mathcal{C}^r_n} \to f_i^*{\mathcal{C}^2_0}\] satisfying the following conditions:
\begin{itemize} \item If $f^*{\mathcal{C}^r_n}$ is marked by the sections $x^{\pm}, \{y^\ell\}, \{z_i^j\}$, then $f_i^*\mathcal{C}^1_n$ is the stabilization of \[(f^*{\mathcal{C}^r_n}, x^+,x^-, y^i, y^{i+1}).\]
\item The map $\varphi_{i,f}$ is the standard contraction map, which is a birational morphism. \end{itemize}
Again, we will be particularly interested in the case where $B = \M^r_n$ and $f = \text{id}$, in which case $f_i = \lambda_i$ and we denote $\varphi_{i,\text{id}}$ by simply $\varphi_i$.  These maps then fit into a diagram as follows:
\begin{equation}
    \label{eq:def_phii}
\xymatrix{
\mathcal{C}^r_n \ar[d]\ar[r]^{\varphi_{i}} & \lambda_{i}^*\mathcal{C}^2_0 \ar[dl]\ar[r] & \mathcal{C}^2_0\ar[d]\\
\M^r_n \ar[rr]^{\lambda_{i}} && \M^2_0 = \P^1.
}
\end{equation}

Equipped with these morphisms between Hassett spaces, we are prepared to describe our moduli space.  In fact, we define separate moduli spaces
\[\L^r_n(\zeta) \subseteq \M^r_n\]
for each primitive $r$th root of unity $\zeta$, parameterizing stable $(r,n)$-curves of type $\zeta$.  The moduli spaces for different $r$th roots of unity are all isomorphic to one another, and the full moduli space
\[\L^r_n = \bigsqcup_{\zeta} \L^r_n(\zeta)\]
parameterizing all stable $(r,n)$-curves is a disjoint union of these isomorphic components.

\subsection{Construction of the moduli space} Let $\zeta \in \mathbb{C}$ be a primitive $r$th root of unity, and for any $\alpha: [n]_0 \rightarrow \Z_r$ and any $\k \in \Z_r$, denote by $\alpha+\k$ the function $[n]_0 \rightarrow \Z_r$ defined by \[(\alpha+\k)(x)=\Big(\alpha(x)+\k\Big) \!\!\mod r.\]
We define $\L^r_n(\zeta)$ as the subscheme of $\M^r_n$ obtained from the following fiber diagram:

\[
\begin{tikzcd}
\L^r_n(\zeta) \arrow[hookrightarrow]{r} \arrow{d} & \M^r_n \arrow{d}{\prod_{\alpha} \pi_\alpha \times \pi_{\alpha+1}\times \prod\limits_{i \in \Z_r} \lambda_i} \\
\prod\limits_{\alpha: [n]_{0} \to \Z_r} \M^1_n \times \{\zeta\} \arrow{r}{\Delta} & \left( \prod\limits_{\alpha: [n]_{0} \to \Z_r}  \M^1_n \times  \M^1_n \right) \times \prod\limits_{i \in \Z_r} \mathbb{P}^1.
\end{tikzcd}
\]
Here, the map $\Delta$ is the product of the diagonal embeddings of $\M^1_n$ into $\prod_{\alpha} \left( \M^1_n \times  \M^1_n \right)$ and of $\{\zeta\}$ into each $\mathbb{P}^1$, and the upper map is the inclusion of $\overline{\calL}^r_n(\zeta)$ into $\overline{\calM}^{r}_n$.

To interpret this fiber diagram more explicitly, note that elements of $\L^r_n(\zeta) \subseteq \M^r_n$ are defined by the condition that there are isomorphisms of pointed curves
\begin{equation}
    \label{eq:pialphacondition}
    (C; x^{\pm}, y^{\alpha(0)}, z_1^{\alpha(1)}, \ldots, z_n^{\alpha(n)}) \cong (C; x^{\pm}, y^{\alpha(0)+1}, z_1^{\alpha(1)+1}, \ldots, z_n^{\alpha(n)+1})
\end{equation}
for all $\alpha: [n]_0 \to \Z_r$, and
\begin{equation}
    \label{eq:lambdaicondition}
    (C; x^+, x^{-}, y^i, y^{i+1}) \cong (C; \infty, 0,1,\zeta)
\end{equation}
for all $i \in \Z_r$.  The first step in proving that $\L^r_n(\zeta)$ is indeed a fine moduli space for stable $(r,n)$-curves of type $\zeta$ is to verify that its points are in bijection with isomorphism classes of such curves.

\begin{lemma}
\label{lem:Lrnobjects}
If $(C; x^{\pm}, \{y^{\ell}\}, \{z_i^j\}) \in \L^r_n(\zeta)$,
then there exists a unique automorphism $\sigma$ of $C$ making $(C; \sigma; x^{\pm}, \{y^{\ell}\}, \{z_i^j\})$ into a stable $(r,n)$-curve of type $\zeta$, and the elements of $\L^r_n(\zeta)$ are precisely the elements of $\M^r_n$ for which such an automorphism exists.
\end{lemma}
\begin{proof}
Fix an element $(C; x^{\pm}, \{y^{\ell}\}, \{z_i^j\}) \in \L^r_n(\zeta)$, and let $C_{\bullet}$ be the component of $C$ containing $x^+$.  Then each $y^\ell$ lies on some (possibly empty) tree of projective lines attached to $C_{\bullet}$ at a point $p^\ell$, and $x^-$ lies on some (possibly empty) tree of projective lines attached to $C_{\bullet}$ at a point $p^-$.

A priori, some of these trees could coincide with one another.  However, they cannot all be identical---that is, $x^-, y^0, \ldots, y^{r-1}$ cannot all lie on a single tree emanating from $C_{\bullet}$---since this would force the only special points on $C_{\bullet}$ to be a single node, the marked point $x^+$, and possibly some of the $z_i^j$.  With the weights \eqref{eq:w}, such a curve cannot be stable.

Thus, there must be at least one $y^\ell$ on a different tree than $x^-$.  In this case, we can choose an automorphism $s$ of $C_{\bullet}$ such that
\[s(x^+) = \infty, \;\; s(p^-) = 0, \;\; s(p^\ell) = 1,\]
and the condition \eqref{eq:lambdaicondition} ensures that
\[s(p^{\ell+1}) = \zeta.\]
(In particular, note that this means that the $y^{\ell+1}$-tree is distinct from both the $y^\ell$-tree and the $x^-$-tree.)  Applying \eqref{eq:lambdaicondition} again with $\ell$ replaced by $\ell+1$ shows that
\[s(p^{\ell+2}) = \zeta^2.\]
Continuing in this way proves that $x^-$ and the $y^\ell$'s lie on $r+1$ distinct trees emanating from $C_{\bullet}$, attached (under the automorphism $s$) at $0$ and roots of unity.

Each such tree must end in one or more leaves.  However, the weights \eqref{eq:w} ensure that a leaf must either contain at least one $y^\ell$ or both of $x^+$ and $x^-$.  Therefore, the $x^-$-tree must be empty (that is, $x^- \in C_{\bullet}$) and each of the $y^\ell$-trees must be a chain of projective lines ending in a single leaf with $y^\ell$.  What remains to be shown, then, is that each of these $r$ chains has the same length---which, in particular, implies that $C$ admits an automorphism $\sigma$ taking $y^\ell$ to $y^{\ell+1}$ for each $\ell$---and that this automorphism takes $z_i^j$ to $z_i^{j+1}$ for each $i \in [n]$ and $j \in \Z_r$.

To see this, first notice that a repeated application of \eqref{eq:pialphacondition} shows
\[(C; x^{\pm}, y^{\alpha(0)}, z_1^{\alpha(1)}, \ldots, z_n^{\alpha(n)}) \cong (C; x^{\pm}, y^{\alpha(0)+i}, z_1^{\alpha(1)+i}, \ldots, z_n^{\alpha(n)+i})\]
for all $i \in \Z_r$.  This implies that if $z_i^0$ lies on the $y^{\alpha(0)}$-spoke of $C$, then $z_i^j$ lies on the $y^{j + \alpha(0)}$-spoke.  In particular, for each fixed $i \in [n]$, no two of the marked points $z_i^0, \ldots, z_i^{r-1}$ can lie on the same spoke, and furthermore, if one of these lies on the central component, then they all do. 

In light of this, we can define a function $\alpha: [n]_0 \rightarrow \Z_r$ as follows.  First, set $\alpha(0) = 0$.  Then, for each $i \in [n]$ such that the marked points $z_i^0, \ldots, z_i^{r-1}$ do not lie on the central component, set $\alpha(i) \in \Z_r$ so that $z_i^{\alpha(i)}$ lies on the $y^0$-spoke of $C$.  Finally, for each $i \in [n]$ such that the marked points $z_i^0, \ldots, z_i^{r-1}$ do lie on the central component, set $\alpha(i) \in \Z_r$ to be any value.  For example, for the curve in Figure~\ref{fig:elementofLrn}, we have
\begin{align*}
    \alpha(0) =0, \;\; \alpha(1) = \text{anything in } \Z_3, \;\; \alpha(2) = 1, \;\; \alpha(3) = 0, \;\;
    \alpha(4) =2.
\end{align*}
For this $\alpha$, the morphism $\pi_{\alpha}$ does no contraction on the $y^0$-spoke of $C$ but it contracts all of the other spokes; see Figure \ref{fig:pi-alpha-2}. 
    
    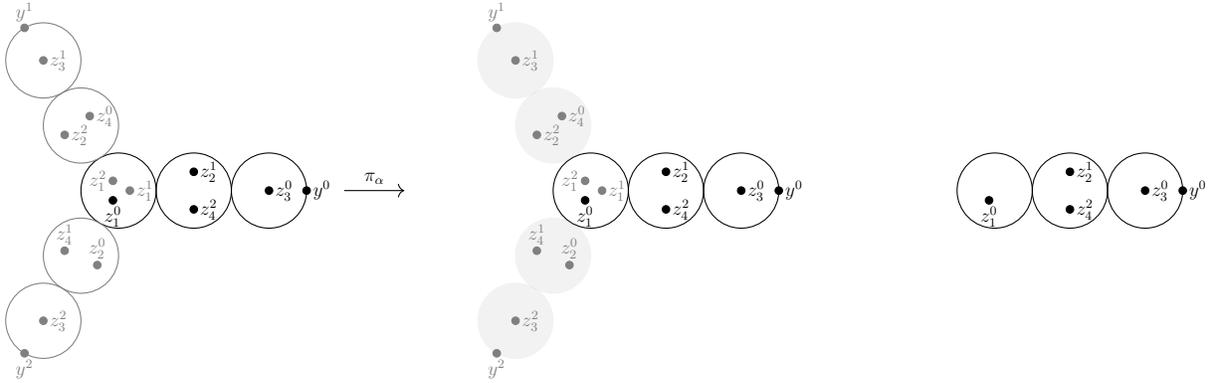
\begin{figure}[h!]
    \begin{center}
    \begin{tikzpicture}[scale=1,  every node/.style={scale=0.6}]
    
    \filldraw[gray] (0.15,0) circle (0.05cm) node[right]{$z^1_1$};
    \filldraw[gray] (-0.075,0.1299) circle (0.05cm) node[left]{$z^2_1$};
    \filldraw (-0.075,-0.1299) circle (0.05cm) node[below]{$z^0_1$};
    
    \filldraw(1,0.25) circle (0.05cm) node[right]{$z^1_2$};
    \filldraw[gray](-0.7165, 0.741) circle (0.05cm) node[right]{$z^2_2$};
    \filldraw[gray](-0.283, -0.991) circle (0.05cm) node[above]{$z^0_2$};
    
    \filldraw(1,-0.25) circle (0.05cm) node[right]{$z^2_4$};
    \filldraw[gray](-0.383,0.991) circle (0.05cm) node[right]{$z^0_4$};
    \filldraw[gray](-0.717,-0.8) circle (0.05cm) node[above]{$z^1_4$};
    
    \filldraw (2,0) circle (0.05cm) node[right]{$z^0_3$};
    \filldraw[gray] (-1,1.732) circle (0.05cm) node[right]{$z^1_3$};
    \filldraw[gray] (-1,-1.732) circle (0.05cm) node[right]{$z^2_3$};
    
    \draw (0,0) circle (0.5cm);
    \draw (1,0) circle (0.5cm);
    \draw (2,0) circle (0.5cm);
    \filldraw (2.5,0) circle (0.05cm) node[right]{$y^0$};
    
    \draw[->] (3, 0) -- node[above] {$\pi_{\alpha}$} (3.8,0);

    \draw[gray] (-0.5,0.866) circle (0.5cm);
    \draw[gray] (-1,1.732) circle (0.5cm);
    \filldraw[gray] (-1.25,2.165) circle (0.05cm) node[above]{$y^1$};

    \draw[gray] (-0.5,-0.866) circle (0.5cm);
    \draw[gray] (-1,-1.732) circle (0.5cm);
    \filldraw[gray] (-1.25,-2.165) circle (0.05cm) node[below]{$y^2$};

    \end{tikzpicture}
    \qquad
    \begin{tikzpicture}[scale=1,  every node/.style={scale=0.6}]
    
    \filldraw[gray] (0.15,0) circle (0.05cm) node[right]{$z^1_1$};
    \filldraw[gray] (-0.075,0.1299) circle (0.05cm) node[left]{$z^2_1$};
    \filldraw (-0.075,-0.1299) circle (0.05cm) node[below]{$z^0_1$};
    
    \filldraw(1,0.25) circle (0.05cm) node[right]{$z^1_2$};
    \filldraw[gray](-0.7165, 0.741) circle (0.05cm) node[right]{$z^2_2$};
    \filldraw[gray](-0.283, -0.991) circle (0.05cm) node[above]{$z^0_2$};
    
    \filldraw(1,-0.25) circle (0.05cm) node[right]{$z^2_4$};
    \filldraw[gray](-0.383,0.991) circle (0.05cm) node[right]{$z^0_4$};
    \filldraw[gray](-0.717,-0.8) circle (0.05cm) node[above]{$z^1_4$};
    
    \filldraw (2,0) circle (0.05cm) node[right]{$z^0_3$};
    \filldraw[gray] (-1,1.732) circle (0.05cm) node[right]{$z^1_3$};
    \filldraw[gray] (-1,-1.732) circle (0.05cm) node[right]{$z^2_3$};
    
    \draw (0,0) circle (0.5cm);
    \draw (1,0) circle (0.5cm);
    \draw (2,0) circle (0.5cm);
    \filldraw (2.5,0) circle (0.05cm) node[right]{$y^0$};

    \filldraw[gray, opacity=0.1] (-0.5,0.866) circle (0.5cm);
    \filldraw[gray, opacity=0.1] (-1,1.732) circle (0.5cm);
    \filldraw[gray] (-1.25,2.165) circle (0.05cm) node[above]{$y^1$};

    \filldraw[gray, opacity=0.1] (-0.5,-0.866) circle (0.5cm);
    \filldraw[gray, opacity=0.1] (-1,-1.732) circle (0.5cm);
    \filldraw[gray] (-1.25,-2.165) circle (0.05cm) node[below]{$y^2$};

    \end{tikzpicture}
    \qquad
    \begin{tikzpicture}[scale=1,  every node/.style={scale=0.6}]
    
    %%\filldraw[gray] (0.15,0) circle (0.05cm) node[right]{$z^1_1$};
    %%\filldraw[gray] (-0.075,0.1299) circle (0.05cm) node[left]{$z^2_1$};
    \filldraw (-0.075,-0.1299) circle (0.05cm) node[below]{$z^0_1$};
    
    \filldraw(1,0.25) circle (0.05cm) node[right]{$z^1_2$};
    %%\filldraw[gray](-0.7165, 0.741) circle (0.05cm) node[right]{$z^2_2$};
    %%\filldraw[gray](-0.283, -0.991) circle (0.05cm) node[above]{$z^0_2$};
    
    \filldraw(1,-0.25) circle (0.05cm) node[right]{$z^2_4$};
    %%\filldraw[gray](-0.383,0.991) circle (0.05cm) node[right]{$z^0_4$};
    %%\filldraw[gray](-0.717,-0.8) circle (0.05cm) node[above]{$z^1_4$};
    
    \filldraw (2,0) circle (0.05cm) node[right]{$z^0_3$};
   %% \filldraw[gray] (-1,1.732) circle (0.05cm) node[right]{$z^1_3$};
   %% \filldraw[gray] (-1,-1.732) circle (0.05cm) node[right]{$z^2_3$};
    
    \draw (0,0) circle (0.5cm);
    \draw (1,0) circle (0.5cm);
    \draw (2,0) circle (0.5cm);
    \filldraw (2.5,0) circle (0.05cm) node[right]{$y^0$};

    \filldraw[gray, opacity=0] (-0.5,0.866) circle (0.5cm);
    \filldraw[gray, opacity=0] (-1,1.732) circle (0.5cm);
    \filldraw[gray, opacity=0] (-1.25,2.165) circle (0.05cm) node[above]{$y^1$};

    \filldraw[gray, opacity=0] (-0.5,-0.866) circle (0.5cm);
    \filldraw[gray, opacity=0] (-1,-1.732) circle (0.5cm);
    \filldraw[gray, opacity=0] (-1.25,-2.165) circle (0.05cm) node[below]{$y^2$};
    \end{tikzpicture}
    \end{center}
    \caption{The image of the curve in Figure \ref{fig:elementofLrn} under $\pi_{\alpha}$, where $\alpha(0) = 0$, $\alpha(1) \in \Z_3$, $\alpha(2) = 1$, $\alpha(3) = 0$ and $\alpha(4) = 2$.}
    \label{fig:pi-alpha-2}
    \end{figure}
    
By the same token, the morphism $\pi_{\alpha+1}$ contracts all spokes except for the $y^1$-spoke, which remains intact.  The condition \eqref{eq:pialphacondition} thus implies that the $y^0$- and $y^1$-spokes are isomorphic pointed curves, meaning they are chains of the same length and the isomorphism between them takes any $z_i^j$ in the $y^0$-spoke to $z_i^{j+1}$.  Repeating this reasoning for each successive spoke shows that $C$ has the appropriate radial symmetry and therefore there exists an automorphism $\sigma$ that makes $C$ into an $(r,n)$-curve of type $\zeta$.  Since the behavior of this automorphism is specified on at least three points of the central component, it is unique.

Conversely, it is clear that \eqref{eq:pialphacondition} and \eqref{eq:lambdaicondition} hold for any $(r,n)$-curve of type $\zeta$, so such curves are precisely the elements of $\L^r_n(\zeta)$.
\end{proof}

The recipe in the proof of Lemma~\ref{lem:Lrnobjects} for producing a function $\alpha: [n]_0 \rightarrow \Z_r$ from a stable $(r,n)$-curve will be useful in what follows, so before proceeding, we take a moment to record it in the following definition.

\begin{definition}
\label{def:compatible}
Let $(C;\sigma; x^{\pm}, \{y^\ell\}, \{z_i^j\})$ be a stable $(r,n)$-curve, which we abbreviate by $C$ for conciseness.   We say that a function $\alpha: [n]_0 \rightarrow \Z_r$ is {\bf compatible} with $C$ if for each $i \in [n]$ such that the $i$th light orbit $z_i^0, \ldots, z_i^{r-1}$ does not lie on the central component, the marked point $z_i^{\alpha(i)}$ lies on the $y^{\alpha(0)}$-spoke of $C$. 
\end{definition}

\begin{remark}
\label{rmk:compatibility}
The notion of compatibility satisfies the following properties:
\begin{enumerate}
    \item If $(C; x^{\pm}, \{y^\ell\}, \{z_i^j\})$ is compatible with $\alpha$, then $(C; x^{\pm}, \{y^\ell\}, \{z_i^j\})$ is compatible with $\alpha+\k$ for all $\k \in \Z_r$.
    \item If $(C; x^{\pm}, \{y^\ell\}, \{z_i^j\})$ is compatible with $\alpha$, then so are all other $(r,n)$-curves with the same dual graph.
    
    \item Let $G$ and $H$ be dual graphs of $(r,n)$-curves such that $H$ is obtained via edge-contraction from $G$.  If
    $\alpha$ is compatible with
    $(r,n)$-curves with dual graph
    $G$, then it is also compatible
    with $(r,n)$-curves with dual
    graph $H$. In particular, a smooth $(r,n)$-curve is compatible with every $\alpha:[n]_0\to\Z_r$.
\end{enumerate}
\end{remark}
Remark~\ref{rmk:compatibility} implies that, for fixed $\alpha:[n]_0 \rightarrow \Z_r$, the locus of curves not compatible with $\alpha$ is a union of boundary strata.  Conversely, the locus of curves that are compatible with $\alpha$ forms an open set, which we denote by
\begin{equation}
    \label{eq:Ualpha}
U_{\alpha} \subseteq \L^n_r(\zeta)
\end{equation}
in what follows.

\subsection{Proof of fine moduli space}

At this point, we have shown via Lemma~\ref{lem:Lrnobjects} that the points of $\L^r_n(\zeta)$ are in bijection with stable $(r,n)$-curves of type $\zeta$.  In order to know that $\L^r_n(\zeta)$ is a fine moduli space for these objects, though, we must also construct a universal family.  This can nearly be bootstrapped from $\M^r_n$: if 
\[\iota: \L^r_n(\zeta) \rightarrow \M^r_n\]
denotes the inclusion, then we can define a universal curve 
$\mathcal{C}^r_n(\zeta)$ over $\L^r_n(\zeta)$ by
\begin{equation}
    \label{eq:univcurve}
    \mathcal{C}^r_n(\zeta) := \iota^*\mathcal{C}^r_n,
\end{equation}
where $\mathcal{C}^r_n$ is the universal curve over $\M^r_n$.  Furthermore, we can define sections $x^{\pm}, y^\ell$, and $z^i_j$ of $\mathcal{C}^r_n(\zeta)$ by pullback of the corresponding sections of $\mathcal{C}^r_n$. What remains, however, is to construct a universal automorphism $\sigma$ of $\mathcal{C}^r_n(\zeta)$.  This is the key content of the following theorem.

\begin{thm}
Let $r \geq 2$ and $n \geq 0$ be integers, and let $\zeta$ be a primitive $r$th root of unity.  Then $\L^r_n(\zeta)$ is a fine moduli space for stable $(r,n)$-curves of type $\zeta$.
\end{thm}
\begin{proof}
Let $\mathcal{C}^r_n(\zeta)$ and its sections $x^{\pm}, y^\ell$, and $z^i_j$ be defined as in \eqref{eq:univcurve} and the subsequent paragraph.  To prove the theorem, it suffices to construct an automorphism $\sigma$ of $\mathcal{C}^r_n(\zeta)$ such that
\begin{equation}
    \label{eq:sigmaproperties}
\sigma \circ x^+ = x^+, \;\; \sigma \circ x^- = x^-, \;\; \sigma \circ y^\ell = y^{\ell+1}, \;\; \text{ and } \sigma \circ z_i^j = z_i^{j+1}
\end{equation}
for all $i \in [n]$ and all $\ell,j \in \Z_r$.  Indeed, if such an automorphism $\sigma$ exists, then it makes $\mathcal{C}^r_n(\zeta)$ into a family of stable $(r,n)$-curves, and this family is of type $\zeta$ by Lemma~\ref{lem:Lrnobjects}. It follows that, for any base scheme $B$, one can restrict the bijection 
\begin{equation}
    \label{eq:modulimap2}
\{\text{morphisms } B\rightarrow \M^r_n\} \leftrightarrow \{\text{families of }\w\text{-stable curves over }B\}/\cong,
\end{equation}
which exists by virtue of $\M^r_n$ being a fine moduli space for $\w$-stable curves, to yield
\begin{align}
\label{eq:modulimap}
\{\text{morphisms } B \rightarrow \L^r_n(\zeta)\} &\rightarrow \{\text{families of stable } (r,n)\text{-curves of type } \zeta \text{ over } B\}/\cong\\
\nonumber f &\mapsto f^*\Big(\mathcal{C}^r_n(\zeta)\Big).
\end{align}
Because it is the restriction of a bijection, the map in \eqref{eq:modulimap} is certainly injective.  It is also surjective, because a family of stable, type-$\zeta$ $(r,n)$-curves over $B$ yields a family of $\w$-stable curves by forgetting $\sigma$, and the latter is the pullback of the universal family on $\M^r_n$ under some $f: B \rightarrow \M^r_n$ by the surjectivity of \eqref{eq:modulimap2}.  The fact that each fiber of the family is in fact a stable $(r,n)$-curve of type $\zeta$ implies, by Lemma~\ref{lem:Lrnobjects}, that $f$ lands in $\L^r_n(\zeta)$.  Thus, the map in \eqref{eq:modulimap} is a bijection, which says precisely that $\L^r_n(\zeta)$ is the requisite fine moduli space.

To construct the automorphism $\sigma$ of $\mathcal{C}^r_n(\zeta)$, we patch together automorphisms defined locally on open subsets of $\mathcal{C}^r_n(\zeta)$.  Toward defining these open sets, let $\alpha: [n]_0 \rightarrow \Z_r$ be any function, and let $U_{\alpha} \subseteq \L^r_n(\zeta)$ be as in \eqref{eq:Ualpha}.  For any $\k \in \Z_r$, consider the morphism
\[\pi_{\alpha+\k}\big|_{U_{\alpha}}: U_{\alpha} \rightarrow \M^1_n.\]
Geometrically, this morphism leaves the $y^{\alpha(0)+\k}$-spoke unchanged but contracts all other spokes.  In particular, since each individual spoke is isomorphic to each other spoke, we see that
\[\pi_{\alpha+\k}\big|_{U_{\alpha}}=\pi_{\alpha+\k'}\big|_{U_{\alpha}}\]
for any $\k,\k' \in \Z_r$.

Now, let $\mathcal{C}_{\alpha} \subseteq \mathcal{C}^r_n(\zeta)$ be the restriction of $\mathcal{C}^r_n(\zeta)$ to $U_{\alpha}$.  Each of the morphisms $\pi_{\alpha+\k}\big|_{U_{\alpha}}$ can be lifted to a birational morphism
\[\psi_{\alpha+\k}: \mathcal{C}_{\alpha} \rightarrow \pi_{\alpha+\k}^*(\mathcal{C}^1_n)\]
as in \eqref{eq:def_psialpha}, which performs the above contraction of all but the $y^{\alpha(0)+\k}$ spoke on each fiber of $\calC_\alpha$.  Note that since the particular spokes being contracted depend on $\k$, we have
\[\psi_{\alpha+\k} \neq \psi_{\alpha+\k'}\]
if $\k \neq \k'$.  

For any $\alpha:[n]_0\to\Z_r$ and any $\k \in \Z_r$, denote by 
\[V_{\alpha,\k} \subseteq \mathcal{C}_{\alpha}\]
the largest open subset of $\mathcal{C}_{\alpha}$ on which the birational morphism $\psi_{\alpha+\k}$ is an isomorphism.  Geometrically, $V_{\alpha,\k}$ is obtained from $\mathcal{C}_{\alpha}$ by removing all spokes in each fiber except for the $y^{\alpha(0)+\k}$-spoke, including removing the points in the central component of the fiber at which these spokes are attached.  In particular, from this geometric description we see that $\{V_{\alpha,\k}\}_{\k \in \Z_r}$ covers $\mathcal{C}_{\alpha}$, and therefore varying over all $\alpha$, we obtain an open cover of $\mathcal{C}^r_n(\zeta)$.

Our goal, now, is to construct an isomorphism
 \[\sigma_{\alpha, \k} \colon V_{\alpha,\k} \to V_{\alpha,\k+1}\]
relative to $U_{\alpha}$, which will serve as the local definition of the universal automorphism $\sigma$.  The key observation is that
\begin{equation}
\label{eq:psikalpha} \psi_{\alpha+\k}(V_{\alpha,\k}) = \psi_{\alpha+\k+1  }(V_{\alpha,\k+1})
\end{equation}
for every $\k \in \Z_r$, from which it follows that $\sigma_{\alpha,\k}$ can be defined as the composition of the two isomorphisms
\begin{equation}
    \label{eq:sigmaSk}
V_{\alpha,\k} \xrightarrow{\psi_{\alpha+\k}|_{V_{\alpha,\k}}} \psi_{\alpha+\k}(V_{\alpha,\k}) = \psi_{\alpha+\k+1}(V_{\alpha,\k+1}) \xrightarrow{(\psi_{\alpha+\k+1}|_{V_{\alpha,\k+1}})^{-1}} V_{\alpha,\k+1}.
\end{equation}

To prove \eqref{eq:psikalpha}, first observe that, by the definition of the maps $\psi_{\alpha+\k}$ and $\pi_\alpha$, we have
    \begin{equation}
        \label{eq:imageVSk}
        \psi_{\alpha+\k}(V_{\alpha,\k}) = \psi_{\alpha + \k} (\mathcal{C}_{\alpha}^{\text{sm}}) \cup \left(\pi_\alpha^* (\mathcal{C}^1_n) \setminus \bigcup_{\ell \neq 0}\text{im}\left(\psi_{\alpha+\k} \circ y^{\alpha(0) + \k + \ell}\right)\right),
    \end{equation}
where $\mathcal{C}_{\alpha}^{\text{sm}} \subseteq \mathcal{C}_{\alpha}$ is the open set whose fibers are smooth curves.  For any $\ell \neq 0$, let us describe the morphism
\[\psi_{\alpha+\k} \circ y^{\alpha(0) + \k + \ell}: U_\alpha \rightarrow \pi_{\alpha}^*(\mathcal{C}^1_n)\]
in geometric terms, by describing its behavior on closed points. Let $C$ be a closed point of $U_\alpha$, which corresponds to a marked curve $(C; x^{\pm}, \{y^{\ell}\}, \{z^j_i\})$.  The image of $C$ under $\pi_\alpha$ is a point that corresponds to a marked curve $C'$, and the image of $C$ under $\psi_{\alpha+\k} \circ y^{\alpha(0) + \k + \ell}$ is a closed point of $C'$ that we will denote $q^{\k,\ell}$. By definition, $C'$ is a chain of projective lines with $x^{\pm}$ together on an end component that we denote by $C_{\bullet}$.  In this notation, we have $q^{\k,\ell} \in C_{\bullet}$, and if we denote the node of $C_{\bullet}$ by $p$, then \eqref{eq:lambdaicondition} implies
\[(C_{\bullet}; x^+, x^-, p, q^{\k,\ell}) \cong (\P^1; \infty, 0, 1, \zeta^{\ell}).\]
Since this result does not depend on $\k$, we see that the morphism $\psi_{\alpha+\k} \circ y^{\alpha(0)+\k+\ell}$ is independent of $\k$ for any $\ell \neq 0$.  Given that $\psi_{\alpha+\k}(\mathcal{C}^{\text{sm}}_\alpha)$ is also manifestly independent of $\k$, it follows from \eqref{eq:imageVSk} that $\psi_{\alpha+\k}(V_{\alpha,\k})$ is independent of $\k$, which proves \eqref{eq:psikalpha}. 

Lastly, we glue the local morphisms $\sigma_{\alpha,\k}$ defined by \eqref{eq:sigmaSk} to give a global morphism \[\sigma:\mathcal{C}_n^r(\zeta)\to\mathcal{C}_n^r(\zeta)\]
over the base $\L_n^r(\zeta)$.  To see that the morphisms $\sigma_{\alpha,\k}$ indeed agree on the overlaps in their domain, we restrict to fibers of $\mathcal{C}_n^r(\zeta)$ and describe the maps $\sigma_{\alpha,\k}$ on closed points.   Given any $(C; x^{\pm}, \{y^{\ell}\}, \{z^j_i\})\in\overline{\mathcal{L}}_{n}^{r}(\zeta)$, let us also denote by $C\subseteq \calC_n^r(\zeta)$ the corresponding fiber of the universal curve.  By Lemma \ref{lem:Lrnobjects}, $C$ is an $(r,n)$-curve; we denote its central component by $C_{\bullet}$ as usual, and we denote by $C(\k)$ the sub-curve obtained by taking the union of $C_{\bullet}$ and the spoke containing $y^{\k}$. Suppose $\alpha:[n]_0\to\Z_r$ is compatible with $C$.  Setting 
\begin{align}
    C^{\circ}(\k):=C(\k)\cap V_{\alpha,\k}=C(\k)\setminus\bigcup_{\k'\ne \k} C(\k'),
\end{align}
we observe that $\sigma_{\alpha,\k}|_{C^{\circ}(\k)}$ defines an isomorphism onto $C^{\circ}(\k+1)$ that fixes $x^{\pm}$, takes $y^{\alpha(0)}$ to $y^{\alpha(0)+1}$, and takes $z_i^{\alpha(i)}$ to $z_i^{\alpha(i)+1}$. As in the proof of Lemma \ref{lem:Lrnobjects}, this forces $\sigma_{\alpha,\k}|_{C_{\bullet}\cap C^{\circ}(\k)}$ to be rotation by $\zeta$, once coordinates are chosen on $C_{\bullet}$ such that $x_+=\infty$, $x_-=0$, and the node connecting $C_{\bullet}$ to $y^{\alpha(0)}$ is $1$. This shows that $\sigma_{\alpha,\k}|_{C^{\circ}(\k)}$ agrees with $\tau|_{C^{\circ}(\k)}$, where $\tau$ is the unique automorphism of $C$ making it into an $(r,n)$-curve of type $\zeta$.  In particular, because this description of $\sigma_{\alpha,\k}$ depends only on the fiber of $\mathcal{C}^r_n(\zeta)$ and not on $\alpha$ or $\k$, the local morphisms $\sigma_{\alpha,\k}$ indeed glue to give a global automorphism  $\sigma$ of $\mathcal{C}^r_n(\zeta)$.

Having equipped the universal curve with a universal automorphism $\sigma$, which (by the argument of the previous paragraph) makes it into a family of $(r,n)$-curves of type $\zeta$, the proof of the theorem is complete.
\end{proof}

\subsection{Geometric observations}
\label{subsec:geometry}

We conclude this section with some geometric observations about the moduli space $\L^r_n(\zeta)$.  Since these are not needed for the current work, we give only brief indications of the proofs.

\begin{observation}
\label{obs:snc}
The moduli space $\L^r_n(\zeta)$ is smooth, and its boundary (the union of the positive-codimension boundary strata) is a simple normal crossings divisor.
\end{observation}

To prove Observation~\ref{obs:snc}, one can leverage the analogous result for the Hassett space $\M^1_n$.  In particular, one can show that the morphisms
\[\pi_{\alpha}|_{U_{\alpha}}: U_{\alpha} \rightarrow \M^1_n\]
are isomorphisms of $U_{\alpha}$ onto an open set $U \subseteq \M^1_n$.  (Specifically, expressing elements of $\M^1_n$ as $(C; x^{\pm}, y, \{z_i\})$ as in Remark~\ref{lem:M1nobjects}, let $C_{\bullet}$ be the component containing $x^{\pm}$.  Then, after choosing coordinates on $C_{\bullet}$ in which $x^+ = \infty$, $x^- = 0$, and the half-node of $C_{\bullet}$ is $1$, the open set $U \subseteq \M^1_n$ consists of curves for which none of the light points $z_i$ lies at an $r$th root of unity in $C_{\bullet}$.)  The local isomorphisms $\pi_{\alpha}$ take the boundary stratification of $\L^r_n(\zeta)$ to the boundary stratification of $\M^1_n$, so Observation~\ref{obs:snc} follows from the analogous statement for Hassett spaces, which is shown in \cite{HASSETT2003316}.

In addition to $\pi_{\alpha}|_{U_{\alpha}}$, there is another natural map from $\L^r_n(\zeta)$ to $\M^1_n$: rather than remembering a single spoke of an $(r,n)$-curve as $\pi_{\alpha}|_{U_{\alpha}}$ does, one can identify all $r$ spokes with each other.  This leads to the following observation.

\begin{observation}
\label{obs:quotient}
There exists a surjective morphism
\[p: \L^r_n(\zeta) \rightarrow \M^1_n\]
that sends an $(r,n)$-curve $C$ with automorphism $\sigma$ to the quotient $C/\sigma$, and $p$ realizes $\M^1_n$ as the quotient\ $\L^r_n(\zeta)/(\Z_r)^n$.  
\end{observation}

This relates $\L^r_n(\zeta)$ with moduli spaces parameterizing coverings of rational curves with marked points: it provides a compactification of the moduli space of coverings with marked orbits that is related to the Harris--Mumford admissible covers spaces \cite{HarrisMumford1982} in the same way in which Hassett spaces are related to $\overline{\calM}_{g, n}$. Other works in this direction can be found in \cite{deopurkar:14}, where  a different choice of weights is made that allows ramification points to collide.

Both $\pi_{\alpha}|_{U_{\alpha}}$ and $p$ can be interpreted from a polytopal perspective, using that $\M^1_n$ is a toric variety (and hence has an associated polytope) whereas $\L^r_n(\zeta)$, as we will see below, has an associated polytopal complex.  We return to this in Remark~\ref{rmk:octants}.

\section{Decorated nested chains}
\label{sec:chains}

Having constructed the moduli space $\L^r_n$ of stable $(r,n)$-curves, we begin the combinatorial heart of the paper: proving that the same combinatorics encodes the boundary strata in $\L^r_n$, the $\T$-cosets in the complex reflection group $S(r,n)$, and the $\Delta$-faces of the polytopal complex $\Delta^r_n$.  The key idea that yields the correspondence between these three types of objects is that all three can be indexed by the discrete data of decorated nested chains, which we now define.

\begin{definition}  Let $r \geq 2$ and $n\geq 0$.  A {\bf decorated nested chain of subsets of $[n]$} (or simply {\bf chain}, for short) is a tuple
\[\tbI = (I_1, \ldots, I_k, \a),\]
where
\[\emptyset =I_0 \subsetneq I_1 \subsetneq \cdots \subsetneq I_k  \subseteq [n]\]
and
\[\a: I_k \rightarrow \Z_r.\]
We refer to $k$ as the {\bf length} of the chain.  The possibility that $k=0$ is allowed, in which case we make the convention that there is a unique length-$0$ chain given by $\tbI = (\emptyset, \a)$ for the unique function $\a: \emptyset \rightarrow \Z_r$. If $n=0$, then the length-$0$ chain is the only chain.
\end{definition}

In the following three sections, we describe a bijective procedure for producing, from a chain $\tbI$, either a boundary stratum $S_{\tbI}$ (Proposition \ref{prop:ItoSI}), a $\T$-coset $C_{\tbI}$ (Proposition \ref{prop:ItoCI}), or a $\Delta$-face $F_{\tbI}$ of $\Delta^r_n$ (Proposition \ref{prop:ItoFI}).  Furthermore, we interpret both the dimension of a stratum (or $\T$-coset, or $\Delta$-face) and the inclusion relation between strata (or $\T$-cosets, or $\Delta$-faces) in terms of corresponding features of the chain $\tbI$.  In particular, the inclusion relation is described in terms of the following relation on chains.

\begin{definition}
Let $\tbI = (I_1, \ldots, I_k, \a)$ and $\tbJ = (J_1, \ldots, J_\ell, \b)$ be chains of length $k$ and $\ell$ respectively.  We say that $\tbI$ {\bf refines} $\tbJ$ if
\[\{J_1, \ldots, J_\ell\} \subseteq \{I_1, \ldots, I_k\}\]
and
\[\b = \a\big|_{J_\ell}.\]
(Note that if a chain $\tbI$ of length $k$ refines a chain $\tbJ$ of length $\ell$, then $J_{\ell} \subseteq I_k$ and so the restriction of $\a$ on $J_{\ell}$ is well-defined.)
\end{definition}

We will find in what follows that 
\begin{itemize}
    \item the boundary stratum $S_{\tbI}$ has codimension $k$, where $\tbI$ is a chain of length $k$,
    \item for boundary strata $S_{\tbI}$ and $S_{\tbJ}$, we have $S_{\tbI} \subseteq S_{\tbJ}$ if and only if $\tbI$ refines $\tbJ$,
\end{itemize}
and the exact same statements hold with boundary strata replaced by $\T$-cosets or $\Delta$-faces.  Thus, passing through chains provides the dimension-preserving, inclusion-preserving bijection of Theorem~\ref{thm:main}.

More precisely, we should note that everything that follows depends on the choice of a primitive $r$th root of unity $\zeta$: the boundary stratum $S_{\tbI}$ lies inside a particular component $\L^r_n(\zeta)$, the definition of the generating set $\T$ of $S(r,n)$ depends on a choice of $\zeta$, and the $\Delta$-faces of $\Delta^r_n$ are described by intersecting with certain hyperplanes whose definition depends on $\zeta$.  Thus, we make the following convention once and for all:

\begin{convention}
\label{convention}
Throughout what follows, $\zeta$ is a fixed choice of primitive $r$th root of unity.
\end{convention}

With this set-up in place, we are ready to flesh out the association of chains to each of the requisite objects.

\section{Combinatorics of the boundary strata}
\label{sec:boundarystrata}

Recall from Section~\ref{subsec:boundarystrata} that the boundary strata in $\L^r_n(\zeta)$ are the closures of the loci of curves of a fixed topological type.  We make use of the following labeling scheme for the components of the underlying curves in a boundary stratum, illustrated in Figure~\ref{fig:labelingscheme}.

\begin{notation}
\label{notation:components}
Let $(C; \sigma; x^{\pm}, \{y^{\ell}\}, \{z_i^j\})$ be a stable $(r,n)$-curve, where $C$ is an $r$-pinwheel curve of length $k$.  (Recall from Definition~\ref{def:pinwheelcurve} that this means that each of the $r$ spokes of $C$ has $k$ components.)  For each $\ell \in \Z_r$, denote by
\[C^\ell_1, \ldots, C^\ell_k\]
the components of the spoke containing $y^\ell$, where $y^\ell \in C^\ell_1$ and the other components are labeled in order from outermost to innermost.  Denote the central component by $C_{k+1}$.
\end{notation}

\begin{figure}[h!]
\begin{center}
    \begin{tikzpicture}[scale=1,  every node/.style={scale=0.6}]
    
    \filldraw (0.15,0) circle (0.05cm) node[right]{$z^1_1$};
    \filldraw (-0.075,0.1299) circle (0.05cm) node[left]{$z^2_1$};
    \filldraw (-0.075,-0.1299) circle (0.05cm) node[below]{$z^0_1$};
    
    \filldraw(1,0.25) circle (0.05cm) node[right]{$z^1_2$};
    \filldraw(-0.7165, 0.741) circle (0.05cm) node[right]{$z^2_2$};
    \filldraw(-0.283, -0.991) circle (0.05cm) node[above]{$z^0_2$};
    
    \filldraw(1,-0.25) circle (0.05cm) node[right]{$z^2_4$};
    \filldraw(-0.383,0.991) circle (0.05cm) node[right]{$z^0_4$};
    \filldraw(-0.717,-0.8) circle (0.05cm) node[above]{$z^1_4$};
    
    \filldraw (2,0) circle (0.05cm) node[right]{$z^0_3$};
    \filldraw (-1,1.732) circle (0.05cm) node[right]{$z^1_3$};
    \filldraw (-1,-1.732) circle (0.05cm) node[right]{$z^2_3$};
    
    \draw (0,0) circle (0.5cm);
    \draw (1,0) circle (0.5cm);
    \draw (2,0) circle (0.5cm);
    \filldraw (2.5,0) circle (0.05cm) node[right]{$y^0$};
    \node at (2,0.75){$C_1^0$};
    \node at (1,0.75){$C_2^0$};

    \draw (-0.5,0.866) circle (0.5cm);
    \draw (-1,1.732) circle (0.5cm);
    \filldraw (-1.25,2.165) circle (0.05cm) node[above]{$y^1$};
    \node at (-1.75, 1.75){$C^1_1$};
    \node at (-1.25, 0.75){$C^1_2$};
    
    \draw (-0.5,-0.866) circle (0.5cm);
    \draw (-1,-1.732) circle (0.5cm);
    \filldraw (-1.25,-2.165) circle (0.05cm) node[below]{$y^2$};
    \node at (-0.25, -2) {$C^2_1$};
    \node at (0.25, -1) {$C^2_2$};
    
    \draw[->] (1.1,-1.1) -- (0.4, -0.4);
    \node at (1.25, -1.25) {$C_3$};
    \end{tikzpicture}
\end{center}
\caption{An element of $\L^4_3$ with components labeled via Notation~\ref{notation:components}.}
\label{fig:labelingscheme}
\end{figure}
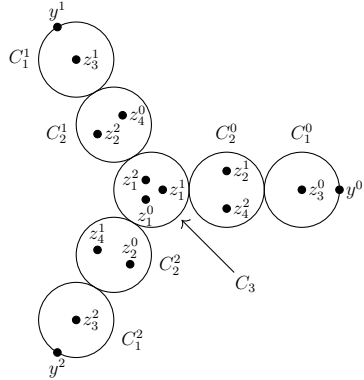

From here, the construction of a boundary stratum in $\L^r_n(\zeta)$ from a chain is as follows.

\begin{definition}
\label{def:SI}
Let $\tbI = (I_1, \ldots, I_k, \a)$ be a chain.  The associated boundary stratum $S_{\tbI} \subseteq \L^r_n(\zeta)$ is the closure of the locus of curves
\[(C; x^{\pm}; y^0, \dots, y^{r-1}; z_1^1, \dots , z_n^{r-1}) \in \L^r_n(\zeta),\] 
where $C$ is a length-$k$ $r$-pinwheel curve and, using Notation~\ref{notation:components}, we have
\begin{enumerate}
    \item for each $j \in \{1, \ldots, k\}$, the light marked points on $C^0_j$ are precisely
    \[\{z_i^{\a(i)} \; | \; i \in I_j \setminus I_{j-1}\},\] where $I_0=\emptyset$; 
    \item the light marked points on the central component $C_{k+1}$ are 
     \[\{z_i^\ell \; | \; i \in [n] \setminus I_k, \; \ell \in \Z_r\} \cup \{x^{\pm}\}.\] 
\end{enumerate}
\end{definition}

\begin{example}
\label{ex:SI}
Let $r=3$ and $n=4$, and consider the chain $\tbI = (I_1, I_2, \a)$ of length $2$, where
\[I_1 = \{3\}, \;\; I_2 = \{2,3,4\}\]
and $\a: I_2 \rightarrow \Z_3$ is given by
 \begin{align*}
    \a(2) &= 1, \quad
    \a(3) = 0, \quad
    \a(4) = 2.
\end{align*}
The associated boundary stratum $S_{\tbI}$ is the closure of the locus of elements of the topological type illustrated in Figure \ref{fig:labelingscheme} above.  In particular, notice that $I_1$ indexes the orbits on the outermost components, $I_2$ indexes the orbits on the two outermost components, and elements of $[n] \setminus I_2$ correspond to orbits in the central component.  The decoration $\a$ indicates the member of each orbit that lies on the $y^0$-spoke of the pinwheel.
\end{example}

The key combinatorial proposition about boundary strata is the following.

\begin{proposition}
\label{prop:ItoSI}
Let $r \geq 2$ and $n\geq 0$.  The association
\[\tbI \mapsto S_{\tbI}\]
is a bijection from the set of decorated nested chains of subsets of $[n]$ to the set of boundary strata of $\L^r_n(\zeta)$.  Furthermore, this bijection satisfies
\begin{enumerate}[label=(\roman*)]
    \item $\text{length}(\tbI) = \text{codim}(S_{\tbI})$,
    \item $S_{\tbI} \subseteq S_{\tbJ}$ if and only if $\tbI$ refines $\tbJ$.
\end{enumerate}
\end{proposition}
\begin{proof}
The surjectivity of $\tbI \mapsto S_{\tbI}$ is clear, since if $S \subseteq \L^r_n(\zeta)$ is a boundary stratum with associated dual graph $G$, then $G$ has a pinwheel shape and thus one can define $I_j \subseteq [n]$ to index the marked points on the outermost $j$ vertices of the $y^0$-spoke of the pinwheel.  The requirement that $z_i^{\a(i)}$ lies on the $y^0$-spoke of the pinwheel for all $i \in I_k$ thus defines a function $\a: I_k \rightarrow \Z_r$ for which $S_{\tbI} = S$.

Injectivity of the association $\tbI \mapsto S_{\tbI}$ follows from item (ii) of the proposition, since it is only possible that $\tbI$ and $\tbJ$ refine one another if $\tbI = \tbJ$.  Item (ii), on the other hand, follows directly from the containment of boundary strata described in Section~\ref{subsec:boundarystrata}.  In particular, if $S_{\tbI}$ and $S_{\tbJ}$ are boundary strata with associated dual graphs $G_{\tbI}$ and $G_{\tbJ}$, then $S_{\tbI} \subseteq S_{\tbJ}$ if and only if $G_{\tbJ}$ can be obtained from $G_{\tbI}$ by edge-contraction of some subset of the edges of $G_{\tbI}$.  Since contracting edges combines the marked points on adjacent vertices, this is the case if and only if $\tbI$ refines $\tbJ$.

Finally, for item (i), let $S = S_{\tbI}$ be a boundary stratum with associated dual graph $G$, where $\text{length}(\tbI) =k$ and therefore $G$ is a pinwheel graph in which each spoke has $k$ vertices.  Choose any
\[\alpha: [n]_0 \rightarrow \Z_r\]
that is compatible with a generic element of $S$ (that is, with any curve with dual graph exactly $G$), where compatibility is defined as in Definition~\ref{def:compatible}.  In the notation of Section~\ref{subsec:hassett-maps}, there is a morphism
\[\pi_\alpha: \L^r_n(\zeta) \rightarrow \M^1_n\]
that maps $S(\zeta)$ birationally onto the boundary stratum $S_0 \subseteq \M^1_n$ whose dual graph $G_0$ consists of only the $y^0$-spoke of $G$ together with the central vertex.  This dual graph $G_0$ has $k$ edges, so the well-known results on Hassett spaces imply that $\text{codim}_{\M^1_n}(S_0)=k$.  Given that $\L^r_n(\zeta)$ is birational to $\M^1_n$, it follows that \[\text{codim}_{\L^r_n(\zeta)}\big(S(\zeta)\big) = \text{codim}_{\M^1_n}(S_0)=k,\]
as claimed.
\end{proof}

In addition to encoding the dimension and inclusion of boundary strata, we note that the chain $\tbI$ also encodes one further piece of geometric information: the decomposition of a boundary stratum into a product of smaller-dimensional moduli spaces.  More precisely, let $\L_n$ denote the Losev--Manin space mentioned in the introduction; in the language of Hassett spaces, this can be described as
\[\L_n = \M_{0,(1,1,\epsilon, \ldots, \epsilon)},\]
where there are $n$ marked points of weight $\epsilon$ and $\epsilon \leq 1/n$.  In particular, elements of $\L_n$ are chains of projective lines with two ``heavy" marked points (one on each end component of the chain) and $n$ ``light" marked points. It is worth stressing that, while we used the Losev--Manin space in the introduction to motivate the present work, the spaces $\L_n$ are not actually the $r=1$ case of the spaces $\L^n_r$.  See Remark~\ref{rmk:r=1}. 

With this notation, the following proposition gives a product decomposition of the boundary stratum $S_{\tbI}$, in which the factors can be read off directly from the chain $\tbI$.

\begin{proposition}\label{lem:productdecompositionofstratum}
For any chain $\tbI = (I_1, \ldots, I_k, \a)$, there is a natural isomorphism between the boundary stratum $S_{\tbI}$ and a product of smaller-dimensional moduli spaces, as follows:
\[S_{\tbI}\cong \L^r_{\abs{[n]\setminus I_k}}\times \prod_{j=1}^k \L_{\abs{I_j\setminus I_{j-1}}}.\]
\end{proposition}
\begin{proof}
This is essentially immediate from the definition of $S_{\tbI}$: the factor of $\L^r_{|[n] \setminus I_k|}$ parameterizes the central component, and the factors $\L_{|I_j \setminus I_{j-1}|}$ each parameterize a component of one spoke of the pinwheel (which determines all of the other spokes).
\end{proof}

\section{Combinatorics of the complex reflection group}
\label{sec:S(r,n)}

We next introduce the group whose structure is combinatorially related to the boundary stratification of $\L^r_n(\zeta)$.  This group, denoted $S(r,n)$, consists of all $n \times n$ matrices whose only nonzero entries are in the group $\mu_r$ of $r$th roots of unity, and with exactly one nonzero entry in each row and column.  (By convention, $S(r,0)$ is a trivial group.)  Note that for any $n \geq 0$, there is a natural group isomorphism
 \[S_n = S(1,n),\]
 where the permutation $\sigma \in S_n$ is identified with the matrix whose $i$th column is the $\sigma(i)$th standard basis vector.

 \begin{remark}
 The group $S(r,n)$ is an example of a  complex reflection group (a finite group acting on $\C^n$ generated by elements whose action fixes a complex hyperplane), and in the classification of complex reflection groups, it is denoted $G(r,1,n)$.  It is sometimes also referred to as the ``generalized symmetric group" and can be equivalently described as the wreath product $\mu_r \wr S_n$.  For more on complex reflection groups, see \cite{Lehrer2009UnitaryRG}.
 \end{remark}

 It is well-known that the symmetric group $S_n$ is generated by the set of adjacent transpositions.  The complex reflection groups $S(r,n)$ have an analogous generating set: define
 \[\mathcal{T} = \{s_0, s_1, \ldots, s_{n-1}\} \subseteq S(r,n),\]
 where
 \[s_0 :=
 \left( \begin{array}{ccccc} \zeta & 0 &0  & \cdots &0\\0& 1 & 0& \cdots & 0\\ 0 & 0 & 1 & \cdots & 0\\ & & & \ddots & \\ 0& 0& \cdots& 0& 1\end{array}\right)\]
 and for $1 \le i \le n - 1$, the element $s_i \in S_n \subseteq S(r,n)$ is the adjacent transposition swapping $i$ and $i+1$, or in other words, the matrix obtained by swapping the $i$th and $(i+1)$st columns of the $n \times n$ identity matrix.  (Note that the definition of $s_0$ makes use of the choice of primitive $r$th root of unity $\zeta$ in Convention~\ref{convention}.)

 It is straightforward to see that $\mathcal{T}$ generates $S(r,n)$: multiplying the identity matrix $I$ by adjacent transpositions can bring any column to the first column, and multiplying by powers of $s_0$ can change the entry in the first column to any power of $\zeta$. More generally, the following remark describes the structure of subgroups of $S(r,n)$ generated by elements of $\mathcal{T}$.

\begin{remark} \label{rmk:subgroupsG} Fix a subset $\calS \subseteq \T$ and let $H_\calS$ be the subgroup of $S(r,n)$ generated by the subset $\calS$. Define $\{s_{j_1}, \dots, s_{j_k}\} := \mathcal{T} \setminus \calS$, where $0 \leq j_1 < \cdots < j_k \leq n-1$. Then $H_\calS$ is equal to the group of block-diagonal matrices \[ \begin{pmatrix}
  S(r,j_1)  & & & & & \\
  &  S_{j_2-j_1}  & & & & \\
  & & S_{j_3-j_2}  & & & \\
  & & & \ddots & &  \\ 
  & & & & S_{j_k-j_{k-1}}  & \\
  & & & & &  S_{n-j_k} 
\end{pmatrix} \subseteq S(r,n). \]
That is, the upper-left block is an element of the complex reflection group $S(r,j_1)$, while the remaining blocks are elements of the indicated symmetric groups.
\end{remark}

The key objects of interest for this paper are the right cosets in $S(r,n)$ of the subgroups described by Remark~\ref{rmk:subgroupsG}.  The following definition establishes the terminology.

\begin{definition} \label{def:Tcoset}
A {\bf $\T$-coset} in $S(r,n)$ is a right coset of the form
\[\langle t_1, \ldots, t_d \rangle \cdot A \subseteq S(r,n)\]
for some $d \geq 0$, where $t_1, \ldots, t_d \in \mathcal{T}$ and $A \in S(r,n)$.  We say that a $\T$-coset as above, where $t_1, \ldots, t_d$ are distinct, has {\bf dimension} $d$ or {\bf codimension} $n-d$.
\end{definition}

In particular, a singleton $\{A\} \subseteq S(r,n)$ is a $0$-dimensional $\T$-coset.  The $1$-dimensional $\T$-cosets are of the form
\[\langle s_i \rangle \cdot A\]
for $s_i \in \mathcal{T}$, and they thus have either two elements or $r$ elements, depending on whether $i \geq 1$ or $i=0$.  Since $\mathcal{T}$ generates $S(r,n)$, the only $n$-dimensional $\T$-coset is the entire group.

Analogously to Definition~\ref{def:SI}, we now describe a procedure for producing a $\T$-coset from a chain $\tbI$.

\begin{definition}
\label{def:CI}
Let $\tbI = (I_1, \ldots, I_k, \a)$ be a chain.  Define the subgroup $H_{\tbI} \subseteq S(r,n)$ by
\[H_{\tbI} := \langle \{ s_i \; | \; n-i \notin \{|I_1|, \ldots, |I_k|\} \rangle,\]
and let $A \in S(r,n)$ be any matrix satisfying the following two conditions:
\begin{enumerate}[label=(\roman*)]
    \item for each $j \in \{1, \ldots, k\}$, the last $|I_j|$ rows of $A$ have nonzero entries in the columns indexed by $I_j$---that is,
    \begin{equation}
        \label{eq:Ij}
I_j = \{ b \in [n] \; | \; A_{a b} \neq 0 \text{ for some } a > n-|I_j|\},
    \end{equation}
    where $A_{ab}$ denotes the entry in the $a$th row and $b$th column of $A$; 
    \item for each $i \in I_k$, the unique nonzero entry in column $i$ of $A$ is $\zeta^{-\a(i)}$.
\end{enumerate}
We define the $\T$-coset associated to $\tbI$ as
\[C_{\tbI} = H_{\tbI} \cdot A.\]
\end{definition}

To make the elements of $C_{\tbI}$ more explicit, note that in the notation of Remark~\ref{rmk:subgroupsG}, we have $H_{\tbI} = H_{\calS(\tbI)}$ for the set
\begin{equation}
    \label{eq:S(I)}
\calS(\tbI) := \{s_i \; | \; n-i \notin \{|I_1|, \ldots, |I_k|\} \subseteq \T.
\end{equation}
Thus, the set $\{s_{j_1}, \ldots, s_{j_k}\} := \mathcal{T} \setminus \calS(\tbI)$ in Remark~\ref{rmk:subgroupsG} is given by
\begin{align*}
    j_1 &= n-|I_k|\\
    j_2 &= n-|I_{k-1}|\\
    &\vdots\\
    j_k &= n-|I_1|,
\end{align*}
and therefore $H_{\tbI}$ is the group of block-diagonal matrices
\[ H_{\tbI} = \begin{pmatrix}
  S(r,|[n] \setminus I_k|)  & & & & & \\
  &  S_{|I_k \setminus I_{k-1}|}  & & & & \\
  & & S_{|I_{k-1} \setminus I_{k-2}|}  & & & \\
  & & & \ddots & &  \\ 
  & & & & S_{|I_2 \setminus I_1|}  & \\
  & & & & &  S_{|I_1|} 
\end{pmatrix} \subseteq S(r,n). \]
From here, one sees from Remark~\ref{rmk:subgroupsG} and the definition of $A$ in Definition~\ref{def:CI} that $C_{\tbI}$ is equal to the set of matrices illustrated in Figure~\ref{fig:CIelements}.  The conditions defining the matrix $A$ are equivalent to requiring that $A$ belong to this set, which in particular implies that the definition of $C_{\tbI}$ does not depend on the choice of $A$ satisfying those conditions.

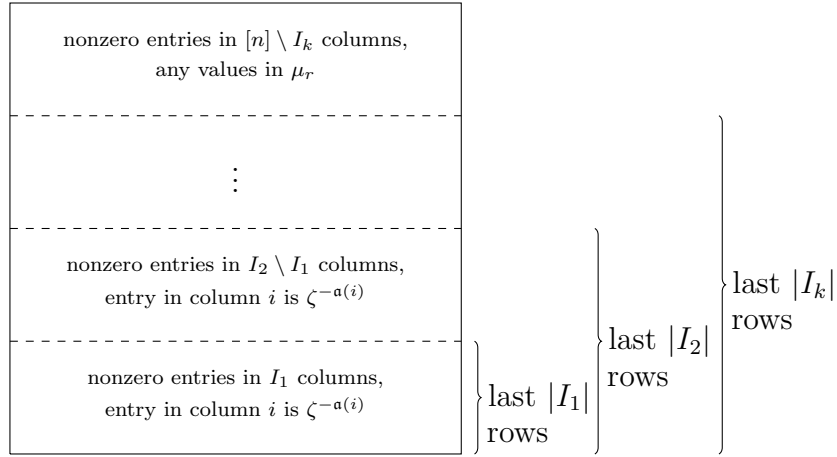
\begin{figure}
    \centering
    \begin{tikzpicture}
    \draw (0,0) -- (6,0) -- (6,6) -- (0,6) -- (0,0);
    \draw[dashed] (0,1.5) -- (6,1.5);
    
    \draw[dashed] (0,3) -- (6,3);
    
    \draw[dashed] (0,4.5) -- (6,4.5);
    
    \node at (3,5.5) {\tiny nonzero entries in $[n] \setminus I_k$ columns,};
    \node at (3,5.1) {\tiny any values in $\mu_r$};
    
        \node at (3, 3.75){$\vdots$};
    
    \node at (3,2.5) {\tiny nonzero entries in $I_2 \setminus I_1$ columns,};
    \node at (3,2.1) {\tiny entry in column $i$ is $\zeta^{-\a(i)}$};
    
    \node at (3,1) {\tiny nonzero entries in $I_1$ columns,};
    \node at (3,0.6) {\tiny entry in column $i$ is $\zeta^{-\a(i)}$};
    
    \draw[decoration={brace,mirror,raise=5pt},decorate]
  (6,0) -- node[right=6pt] {last $|I_1|$} (6,1.5);
  \node at (6.75,0.25){rows};
  
      \draw[decoration={brace,mirror,raise=5pt},decorate]
  (7.6,0) -- node[right=6pt] {last $|I_2|$} (7.6,3);
  \node at (8.35, 1){rows};
  
        \draw[decoration={brace,mirror,raise=5pt},decorate]
  (9.25,0) -- node[right=6pt] {last $|I_k|$} (9.25,4.5);
  \node at (10, 1.75){rows};
  
    \end{tikzpicture}
    \caption{A typical element of the $\T$-coset $C_{\tbI}$.  The action of $H_{\tbI}$ permutes rows within each of the blocks separated by dotted lines, and multiplies elements of the top-most block by $r$th roots of unity.}
        \label{fig:CIelements}
\end{figure}

To illustrate the construction of Definition~\ref{def:CI}, we compute the associated $\T$-coset for the same chain that we considered in Section~\ref{sec:boundarystrata}.

\begin{example}
\label{ex:CI}
As in Example \ref{ex:SI}, let $r=3$ and $n=4$, and let $\tbI$ be the chain \[\tbI = \left(\{3\}, \{2,3,4\}, \a \right),\]
where $\a: I_2 \rightarrow \Z_3$ is given by
\[\a(2) = 1, \;\; \a(3) = 0, \;\; \a(4) = 2.\]
Then
\[H_{\tbI} = \langle s_i \; | \; 4-i \notin \{1,3\} \rangle = \langle s_0, s_2 \rangle,\]
and $C_{\tbI} = H_{\tbI} \cdot A$ for any matrix $A$ such that
\begin{enumerate}[label=(\roman*)]
    \item the last row has its nonzero entry in column $3$, and that entry is $\zeta^0$; and
    \item the last three rows have their nonzero entries in columns $2$, $3$, and $4$, and those entries are $\zeta^{-1}=\zeta^2$, $\zeta^0$, and $\zeta^{-2}=\zeta^1$, respectively. \end{enumerate}
The action of $H_{\tbI}$ on matrices of this form multiplies the first row by roots of unity and swaps the second and third rows.  Thus, we have
\[C_{\tbI} = \left\{\left( 
\begin{array}{cccc} 
\zeta^i & 0 & 0 & 0\\
0 & 0 & 0 & \zeta^1\\
0 & \zeta^2 & 0 & 0\\
0 & 0 & \zeta^0 & 0
\end{array}\right) \Bigg| \; i \in \Z_3\right\}  \; \cup \; \left\{\left( 
\begin{array}{cccc} 
\zeta^i & 0 & 0 & 0\\
0 & \zeta^2 & 0 & 0\\
0 & 0 & 0 & \zeta^1\\
0 & 0 & \zeta^0 & 0
\end{array}\right) \; \Bigg| \; i \in \Z_3\right\}.\]
\end{example}

With the association $\tbI \mapsto C_{\tbI}$ established, we are prepared to state an analogue for $\T$-cosets of Proposition~\ref{prop:ItoSI}.

\begin{proposition}
\label{prop:ItoCI}
Let $r \geq 2$ and $n\geq 0$.  The association
\[\tbI \mapsto C_{\tbI}\]
is a bijection from the set of decorated nested chains of subsets of $[n]$ to the set of $\T$-cosets in $S(r,n)$.  Furthermore, this bijection satisfies
\begin{enumerate}[label=(\roman*)]
    \item $\text{length}(\tbI) = \text{codim}(C_{\tbI})$,
    \item $C_{\tbI} \subseteq C_{\tbJ}$ if and only if $\tbI$ refines $\tbJ$.
\end{enumerate}
\end{proposition}
\begin{proof}
The surjectivity of $\tbI \mapsto C_{\tbI}$ is clear, since given an arbitrary $\T$-coset
\[C=\langle s_{\ell_1}, \ldots, s_{\ell_{n-k}} \rangle \cdot A,\]
one can define sets $I_1 \subseteq \cdots \subseteq I_k$ by \eqref{eq:Ij}.  Condition (ii) of Definition~\ref{def:CI} then uniquely defines a function $\a: I_k \rightarrow \Z_r$, and by construction, setting $\tbI = (I_1, \ldots, I_k, \a)$ gives $C = C_{\tbI}$.

Injectivity of this association will follow from item (ii), while item (i) is immediate from the definition of $\text{codim}(C_{\tbI})$.  Thus, what remains is to prove item (ii).

Let
\[\tbI = (I_1, \ldots, I_k, \a)\]
and
\[\tbJ = (J_1, \ldots, J_{\ell}, \b),\]
and suppose that $\tbI$ refines $\tbJ$.  Then
\[\{I_1, \ldots, I_k\} \supseteq \{J_1, \ldots, J_\ell\}\]
and hence
\[\{|I_1|, \ldots, |I_k|\} \supseteq \{|J_1|, \ldots, |J_\ell|\}.\]
From here, it is straightforward to unpack that
\[\calS(\tbI) \subseteq \calS(\tbJ),\]
where $\calS(\tbI)$ is defined by \eqref{eq:S(I)} and $\calS(\tbJ)$ is defined analogously.  Furthermore, a matrix $A$ satisfying the conditions of Definition~\ref{def:CI} for $\tbI$ will satisfy the same conditions for $\tbJ$, so the same matrix can be chosen to represent both cosets.  It follows that $C_{\tbI} \subseteq C_{\tbJ}$.

Conversely, suppose that $C_{\tbI} \subseteq C_{\tbJ}$.  This means that one can choose the same representative for both cosets, so we have
\begin{equation}
    \label{eq:CIcontainedCJ}
 H_{\tbI}  \cdot A  \subseteq H_{\tbJ} \cdot A
\end{equation}
for a matrix $A \in C_{\tbI}$.  The elements of $C_{\tbI}$ differ from $A$ by permuting the rows within blocks as in Figure~\ref{fig:CIelements}, so in order for the containment \eqref{eq:CIcontainedCJ} to hold, the corresponding blocks for $C_{\tbJ}$ must be unions of the blocks for $C_{\tbI}$.  That is, we must have
\[\{|I_1|, \ldots, |I_k|\} \supseteq \{|J_1|, \ldots, |J_\ell|\}.\]
In particular, for any $i \in \{1, \ldots, \ell\}$, we must have $|J_i| = |I_j|$ for some $j \in \{1, \ldots, k\}$, so the set $J_i$ indexing the columns with nonzero entries in the last $|J_i|$ rows of $A$ is equal to the set $I_j$ indexing the columns with nonzero entries in the last $|I_j|$ rows of $A$.  That is, we have
\[\{I_1, \ldots, I_k\} \supseteq \{J_1, \ldots, J_\ell\}.\]
Since $\a$ and $\b$ are determined by the same matrix $A$, we also have $\b = \a|_{I_{k}}$, concluding the proof that $\tbI$ refines $\tbJ$.
\end{proof}

We end this section by noting that, analogously to the way in which Proposition~\ref{lem:productdecompositionofstratum} gives a product decomposition of a boundary stratum in terms of the combinatorial data of a chain, the $\T$-cosets in $S(r,n)$ have decompositions as products of groups dictated by their associated chains.  To state the decomposition, for any chain $\tbI$, let
\begin{equation}
    \label{eq:PhiI}
\Phi_{\tbI}: S(r,|[n] \setminus I_k|) \times S(r,|I_k \setminus I_{k-1}|) \times \cdots \times S(r,|I_2 \setminus I_1|) \times S(r,|I_1|) \hookrightarrow S(r,n)
\end{equation}
be the embedding of the left-hand side as block-diagonal matrices in $S(r,n)$, where the first factor is embedded as the first $|[n] \setminus I_k|$ rows and the columns indexed by $[n] \setminus I_k$, and similarly for the remaining factors.

\begin{proposition} \label{cor:prodCI}  For any chain $\tbI=(I_1, \cdots, I_k, \a)$, let $C_{\tbI} = H_{\tbI} \cdot A$ be the associated $\T$-coset.  Then there is a natural isomorphism between $H_{\tbI}$ and a product of complex reflection groups, as follows:
\begin{equation}
\label{eq:0}
H_{\tbI} \cong S(r,|[n] \setminus I_k|) \times \prod_{j=1}^{k} S_{|I_j\setminus I_{j-1}|}.
\end{equation}
More specifically, let $\Phi_{\tbI}$ be as in \eqref{eq:PhiI} and let $A_j \in S(r, |I_j\setminus I_{j-1}|)$ be any matrix whose entry in the $\ell$th column is $\zeta^{-\a(\ell)}$ for each $\ell \in I_j \setminus I_{j-1}$.  Then
\begin{equation}
    \label{eq:CIdecomp}
C_{\tbI} = \Phi_{\tbI}\left( S(r,|[n] \setminus I_k|) \times \prod_{j=1}^{k} 
S_{|I_j \setminus I_{j-1}|} \cdot A_j\right).
\end{equation}
\end{proposition}
\begin{proof}
Recall from Remark~\ref{rmk:subgroupsG} that, if $\calS \subseteq \T$ and $H_\calS$ denotes the subgroup of $S(r,n)$ generated by $\calS$, then the block-diagonal decomposition gives an isomorphism
\[H_\calS \cong S(r,j_1) \times S_{j_2-j_1} \times S_{j_3-j_2} \times \cdots \times S_{j_k- j_{k-1}} \times S_{n-j_k},\]
in which $0 \leq j_1 < \cdots < j_k \leq n-1$ are defined by 
\begin{equation}
    \label{eq:1}
j \in \{j_1, \ldots, j_k \} \; \Leftrightarrow \; s_j \notin \calS.
\end{equation}
The group $H_{\tbI}$ is equal to $H_{\calS(\tbI)}$ for the set $\calS(\tbI) \subseteq \T$ defined by
\begin{equation}
    \label{eq:2}
    s_i \in \calS(\tbI) \; \Leftrightarrow n-i \notin \{|I_1|, \ldots, |I_k|\}.
\end{equation}
Combining \eqref{eq:1} and \eqref{eq:2} shows that the isomorphism \eqref{eq:0} is given by the block-diagonal decomposition.  The elements of $C_{\tbI}$ are precisely the matrices with this block-diagonal decomposition and entry $\zeta^{-\a(i)}$ in column $i$ for all $i \in I_k$, which proves \eqref{eq:CIdecomp}.
\end{proof}

\section{Combinatorics of the permutohedral complex}
\label{sec:permutohedral-complex}

The third setting that is combinatorially related to the boundary stratification of $\L^r_n(\zeta)$ and to the $\T$-coset structure of $S(r,n)$ is the permutohedral complex.  To define it, we set
\[\Y:= \R^{\geq 0} \cdot \mu_r \subseteq \C\]
for any $r \geq 2$, and we set
\[\delta^n_k:=n + (n-1) + \cdots + (n-k+1)\]
for any $n \geq 1$ and $k \in [n]$.

\begin{definition}
Let $r \geq 2$ and $n \geq 0$.  The {\bf $n$-dimensional $r$-permutohedral complex} is defined as
\[\Delta_n^r := \left\{(x_1, \ldots, x_n) \in \Y^n \; \left| \; \sum_{i \in I} |x_i| \leq \delta^n_{|I|} \text{ for all } I \subseteq [n]\right.\right\}\]
if $n \geq 1$, or as a single point if $n=0$.
\end{definition}

Because in general $\Delta^r_n$ is not a polytope (rather, as we will prove in Corollary~\ref{cor:facedim} below, it is a polytopal complex), defining the appropriate notion of ``face" requires some care.  We carry this out in the following definition.

\begin{definition}
\label{def:face}
A {\bf decorated subset} of $[n]$ is $\widetilde{I} = (I, a)$, where $I \subseteq [n]$ and $a: I \rightarrow \Z_r$ is any function; in other words, it is a length-$1$ decorated nested chain.  Any decorated subset has an associated hyperplane
\[H_{\widetilde{I}} := \left\{(x_1, \ldots, x_n) \in \C^n \; \left| \; \sum_{i \in I} \zeta^{a(i)} \cdot x_i = \delta^n_{|I|} \right.\right\} \subseteq \C^n.\]
A {\bf $\Delta$-face} of $\Delta^r_n$ is defined as any nonempty intersection
\[\Delta^r_n \cap H_{\widetilde{I}_1} \cap \cdots \cap H_{\widetilde{I}_k},\]
where $\widetilde{I}_1, \ldots, \widetilde{I}_k$ are a choice of distinct decorated subsets of $[n]$.
\end{definition}

When $r=2$, the complex $\Delta^2_n$ is in fact a polytope in $\R^n$ (in particular, $\Delta^2_2$ is the octagon illustrated in Figure~\ref{fig:Delta22example}, and more generally, $\Delta^2_n$ is known as the type-$B$ permutohedron) and its $\Delta$-faces are precisely its faces in the usual sense.  When $r \ge3$, on the other hand, the $\Delta$-faces of $\Delta^r_n$ are themselves polytopal complexes.  We illustrate this in an example before proving it in general.

\begin{example}
\label{ex:facestructure}
Let $r=3$ and $n=2$.  Given $\widetilde{I} = (I, a)$ with $I = \{1, 2\}$ and $a: I \rightarrow \Z_3$ defined by
\[a(1)=2, \;\;\;a(2)=0,\]
the associated $\Delta$-face is
 \[\Delta^3_2 \cap H_{\widetilde{I}} = \{(x_1, x_2) \in \Delta^3_2 \; | \; \zeta^2 x_1 + \zeta^0 x_2 = 3\}.\]
Geometrically, this is a line segment, and there are nine such $\Delta$-faces of $\Delta^3_2$ given by changing the powers of $\zeta$ in the coefficients of the defining equations; see the nine green line segments labeled in  Figure~\ref{fig:Delta32example}.  On the other hand, given $\widetilde{I'} = (I', a')$ where $I' = \{1\}$ and $a'(1)= 2$, the $\Delta$-face associated to $\widetilde{I'}$ is 
\[\Delta^3_2 \cap H_{\widetilde{I'}} = \{(x_1, x_2) \in \Delta^3_2 \; | \; \zeta^2 x_1 = 2\}.\]
One can check that this is equivalent to
\[\{2\zeta^1\} \times \{x_2 \in \Y \; | \; |x_2| \leq 1\},\]
which is the union of three line segments in a ``$\mathsf{Y}$" shape.  There are three such $\Delta$-faces given by changing the value of $a'(1)$, labeled in red in Figure~\ref{fig:Delta32example}, and there are three similar $\Delta$-faces given by the equations $\zeta^i x_2 = 2$ for $0 \le i \le 2$, labeled in blue in Figure~\ref{fig:Delta32example}. 

In addition, $\Delta^3_2$ has $0$-dimensional $\Delta$-faces given by intersecting two $1$-dimensional $\Delta$-faces; it is straightforward to check that each such nonempty intersection is a single point.  Finally, although Figure~\ref{fig:Delta32example} shows $\Delta^3_2$ as the union of nine pentagons, these together constitute just a single $2$-dimensional face.

\end{example}

This example illustrates that $\Delta^r_n$ is a polytopal complex glued from $n$-dimensional polytopes, and that a $\Delta$-face given by intersecting $\Delta^r_n$ with $k$ distinct hyperplanes is a polytopal complex glued from $(n-k$)-dimensional polytopes.  Before proving these observations in general, it is useful to draw on another key observation: a condition on the decorated subsets $\widetilde{I}_1, \ldots, \widetilde{I}_k$ must be satisfied in order to ensure that the corresponding $\Delta$-face is nonempty.

\begin{lemma}
\label{lem:nonemptyface}
Let $\widetilde{I}_1, \ldots, \widetilde{I}_k$ be 
 decorated subsets of $[n]$, where $\widetilde{I}_j = (I_j, a_j)$ for each $j$.  Then
\[\Delta^r_n \cap H_{\widetilde{I}_1} \cap \cdots \cap H_{\widetilde{I}_k} \neq \emptyset\]
if and only if, after possibly reordering $\widetilde{I}_1, \ldots, \widetilde{I}_k$, the tuple $(I_1, \ldots, I_k, a_k)$ is a decorated nested chain of subsets of $[n]$.
\end{lemma}
\begin{proof}
Since $n$ is fixed throughout this proof, we write $\delta^n_k$ as simply $\delta_k$ to avoid cluttering the notation.

Without loss of generality, we assume that the decorated sets
$\widetilde{I}_1,\cdots,\widetilde{I}_k$
are distinct.  To prove the forward direction, it suffices to show that
\[\Delta^r_n \cap H_{\widetilde{I}} \cap H_{\widetilde{J}} \neq \emptyset \; \Longrightarrow \; \widetilde{I} \subseteq \widetilde{J}.\]
Suppose, then, that $x \in \Delta^r_n \cap H_{\widetilde{I}} \cap H_{\widetilde{J}}$, and let $\widetilde{I} = (I,a)$ and $\widetilde{J} = (J,b)$.  The fact that $x \in \Delta^r_n$ means in particular that $x \in \Y^n$, so 
\[x  = (\lambda_1 c_1, \ldots, \lambda_n c_n)\]
for some $\lambda_i \in \R^{\geq 0}$ and $c_i \in \mu_r$. We claim that
\[c_i = \zeta^{-a(i)}\]
for all $i \in I$.

To see this, note that the fact that $x \in H_{\widetilde{I}}$ means that
\begin{equation}
    \label{eq:inHI}
\sum_{i \in I}\lambda_i \cdot \zeta^{a(i)}c_i  = \delta_{|I|},
\end{equation}
and the fact that $x \in \Delta^r_n$ means that
\[\sum_{i \in I} \lambda_i \leq \delta_{|I|}.\]
From the triangle inequality we obtain
\[\delta_{|I|} = \left| \sum_{i \in I}\lambda_i \cdot \zeta^{a(i)}c_i \right| \leq \sum_{i \in I} \Big| \lambda_i \cdot \zeta^{a(i)}c_i  \Big| = \sum_{i \in I} \lambda_i \leq \delta_{|I|}.\]
Thus, the triangle inequality is in fact an equality, which is only possible if the complex numbers $\lambda_i \cdot \zeta^{a(i)}c_i$ are all non-negative real scalar multiples of one another.  Since their sum is a positive real number by \eqref{eq:inHI}, they must each individually be non-negative real numbers; that is,
\begin{equation}
    \label{eq:Rgeq0}
\lambda_i \cdot \zeta^{a(i)}c_i \in \R^{\geq 0}
\end{equation}
for each $i$.  Furthermore, $\lambda_i\neq 0$ for all $i\in I$, since if $\lambda_j=0$ for some $j\in I$, then
\[\sum_{i\in I\setminus\{j\}}\lambda_i=\delta_{|I|}>\delta_{|I|-1},\]
violating one of the inequalities in the definition of $\Delta^r_n$.  Thus, we have $\lambda_i > 0$ for all $i$, so \eqref{eq:Rgeq0} implies $\zeta^{a(i)} c_i \in \R^{\geq 0} \cap \mu_r$.  We conclude that $c(i) = \zeta^{-a(i)}$, as claimed.

This proves that the decorations on $I$ are given by the inverses of the coefficients $c(i)$ on $x \in \Delta^r_n \cap H_{\widetilde{I}} \cap H_{\widetilde{J}}$, and the exact same argument shows that the decorations on $J$ are given by the same formula.  Thus, what remains to be proved is that either $I \subseteq J$ or $J \subseteq I$.  To see this, notice that the fact that $x \in H_{\widetilde{I}} \cap H_{\widetilde{J}}$ can now be expressed as
\[\sum_{i \in I} \lambda_i = \delta_{|I|} \;\;\text{ and }\;\; \sum_{i \in J} \lambda_i = \delta_{|J|}.\]
From these equations, we deduce that 
\[\sum_{i \in I \cup J} \lambda_i = \sum_{i \in I \setminus (I \cap J)} \lambda_i + \sum_{i \in J \setminus (I \cap J)} \lambda_i + \sum_{i \in I \cap J} \lambda_i = \delta_{|I|} + \delta_{|J|} - \sum_{i \in I \cap J} \lambda_i.\]

By the defining inequalities of $\Delta^r_n$, we have
\[\sum_{i \in I \cap J} \lambda_i \leq \delta_{|I \cap J|},\]
so
\[\sum_{i \in I \setminus (I \cap J)} \lambda_i + \sum_{i \in J \setminus (I \cap J)} \lambda_i + \sum_{i \in I \cap J} \lambda_i \geq \delta_{|I|} + \delta_{|J|} - \delta_{|I \cap J|},\]
or in other words,
\begin{equation}
    \label{eq:biginequality}
\sum_{i \in I \cup J} \lambda_i \geq \delta_{|I|} + \delta_{|J|} - \delta_{|I \cap J|}.
\end{equation}
But a straightforward calculation shows that if $|I \cap J|$ is strictly less than both $|I|$ and $|J|$, then
\[\delta_{|I|} + \delta_{|J|} - \delta_{|I\cap J|} > \delta_{|I|+|J|-|I \cap J|}.\]
Given that $\delta_{|I|+|J|-|I \cap J|} = \delta_{|I \cup J|}$, we would then obtain from \eqref{eq:biginequality} that
\[\sum_{i \in I \cup J} \lambda_i > \delta_{|I \cup J|},\]
contradicting one of the defining inequalities of $\Delta^r_n$.  Thus, we must have either $|I \cap J| = |I|$ or $|I \cap J| = |J|$, meaning that either $I \subseteq J$ or $J \subseteq I$.  This concludes the proof of the forward direction of the lemma.

For the reverse direction, it suffices to show the statement when $k=n$ and $\widetilde{I}_1, \ldots, \widetilde{I}_n$ are distinct, because any chain can be extended to a maximal one.  In light of this, let
 \[\tbI = (I_1, \ldots, I_n, \a)\]
be a maximal chain, which can equivalently be expressed as 
\begin{align*}
    I_1 &= \{i_1\}\\
    I_2 &= \{i_1, i_2\}\\
    \vdots\\
    I_n &= \{i_1, i_2, \ldots, i_n\} = [n]
\end{align*}
for some $i_1, \ldots, i_n$.  Setting
 \begin{equation}
     \label{eq:xj}
     x_{i_j} = \zeta^{-\a(i_j)} \cdot (n+1 - j)
 \end{equation}
 for each $j$, it is straightforward to check that $(x_1, \ldots, x_n) \in \Delta^r_n \cap H_{\widetilde{I}_1}\cap \cdots \cap H_{\widetilde{I}_n}$, where $\tilde{I}_j := (I_j, \a|_{I_j})$.  Therefore, $\Delta^r_n\cap H_{\widetilde{I}_1}\cap \cdots \cap H_{\widetilde{I}_n} \neq \emptyset$, so the reverse direction of the lemma is proved.
 \end{proof}

The key upshot of Lemma~\ref{lem:nonemptyface} is that the $\Delta$-faces of $\Delta^r_n$, like the boundary strata and $\T$-cosets, can be indexed by chains.  The notation, analogously to the previous two sections, is as follows.

\begin{definition}
\label{def:FI} Let $\tbI = ( I_1, \ldots, I_k, \a)$ be a chain.  Then the $\Delta$-face of $\Delta^r_n$ associated to $\tbI$ is
\[F_{\tbI} := \Delta^r_n \cap H_{\widetilde{I}_1} \cap \cdots \cap H_{\widetilde{I}_k},\]
where $\widetilde{I}_j = \left(I_j, \a|_{I_j}\right)$ for each $j$.
\end{definition}

In the special case where the chain $\tbI$ is maximal, the $\Delta$-face $F_{\tbI}$ is a single point of $\Delta^r_n$, which we refer to as a {\bf vertex}.  Note that equation~\eqref{eq:xj} makes the coordinates of the vertex associated to a maximal chain $\tbI$ explicit.  More generally, the proof of Lemma~\ref{lem:nonemptyface} gives an explicit description of the elements of $F_{\tbI}$ for any chain $\tbI$, which we collect in the following remark for future reference.

\begin{remark}\label{cor:settheoreticdescriptionofface}
Suppose ${\tbI}=(I_1,\ldots,I_k, \a)$ is a chain. Then $(x_1,\ldots ,x_n)\in \Y^n$ lies in $F_{\tbI}$ if and only if the following conditions are satisfied:
\begin{enumerate}[label={\bf(C\arabic*)}]
    \item $(x_1,\ldots ,x_n)\in \Delta_n^r$, or in other words,
    \[\sum_{i\in I}\abs{x_i}\le\delta_{\abs{I}}^n\]
    for all $I \subseteq [n]$;
    \label{it:isincomplex}
    \item $ x_i \in \R^{\geq 0} \cdot \zeta^{-\a(i)}$ for all $i \in I_k$; \label{it:whichoctant}
    \item for all $j \in \{1, \ldots, k\}$,
    \[\sum_{i\in I_j} \abs{x_i}=\delta^n_{\abs{I_j}}.\]
    \label{it:equalities}
\end{enumerate}
\end{remark}

Let us illustrate the passage from a chain to its associated $\Delta$-face for the same chain considered in Examples~\ref{ex:SI} and \ref{ex:CI} above.

\begin{example}
\label{ex:FI}
Let $r=3$ and $n=4$, and consider again the chain $\tbI = (I_1, I_2, \a)$, where
\[I_1 = \{3\}, \;\; I_2 = \{2,3,4\}\]
and $\a: I_2 \rightarrow \Z_3$ is given by
 \begin{align*}
    \a(2) &= 1, \quad
    \a(3) = 0, \quad
    \a(4) = 2.
\end{align*}
The associated $\Delta$-face $F_{\tbI} \subseteq \Delta^3_4$ is, by definition,
\[F_{\tbI} = \left\{(x_1, \ldots, x_4) \in \Y^4 \; \left| \; \substack{\textstyle \sum_{i \in I} |x_i| \leq \delta^4_{|I|} \; \text{ for all } I \subseteq [4]\\\\ \textstyle \zeta^0 x_3 = 4 \\\\ \textstyle \zeta^1 x_2 + \zeta^0 x_3 + \zeta^2 x_4 = 9}\right.\right\}.\]
It is illuminating to divide the four coordinates according to the decomposition
\begin{align*}
[4] &=  I_1 \cup \Big(I_2 \setminus I_1\Big) \cup \Big([4] \setminus I_2\Big)\\
&=\{3\} \cup \{2,4\} \cup \{1\}.
\end{align*}
For the coordinates in each of these sets, we have the following conditions:
\begin{itemize}
    \item The coordinate $x_3$ must satisfy
    \[x_3 \in \mathbb{R}^{\geq 0}\cdot \zeta^0,\]
    and by the second equality in the above expression for $F_{\tbI}$, we have
    \begin{equation}
        \label{eq:x3}
        |x_3|=4.
    \end{equation}
    Thus, the value of $x_3$ is completely determined.
    \item The coordinates $x_2$ and $x_4$ must satisfy
    \[x_2 \in \mathbb{R}^{\geq 0}\cdot \zeta^2 \text{ and } x_4 \in \mathbb{R}^{\geq 0} \cdot \zeta^1,\]
    and by the second and third equalities in the above expression for $F_{\tbI}$, we have
    \begin{equation}
        \label{eq:x24}
        |x_2| + |x_4| = 5.
    \end{equation}
    Thus, $(x_2, x_4)$ lie in a line segment.
    \item The coordinate $x_1$ can lie in any of the three branches of $\mathsf{Y}$.  Combining equations \eqref{eq:x3} and \eqref{eq:x24} with the first inequality of $F_{\tbI}$, we see that $|x_1| \leq 1$.  Thus,
    \[x_1 \in \Delta^3_1,\]
    which is the $\mathsf{Y}$-shaped region depicted in Figure~\ref{fig:exFI}.
\end{itemize}
Combining these conditions, we find that $F_{\tbI}$ is the product shown in Figure~\ref{fig:exFI}.  
\end{example}

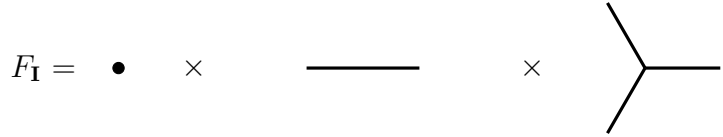
\begin{figure}[h!]
    \centering
    \begin{tikzpicture}
    \node at (-1,0) {$F_{\tbI}$ =};
    
    \filldraw (0,0) circle (0.075);
    
    \node at (1,0) {$\times$};
     
    \draw[very thick] (2.5,0) --  (4,0)  ;
    
    \node at (5.5,0) {$\times$};
    
    \draw[very thick] (7,0) -- (8,0);
    \draw[very thick] (7,0) -- (6.5,0.866);
    \draw[very thick] (7,0) -- (6.5,-0.866);

    \end{tikzpicture}
    \caption{The $\Delta$-face $F_{\tbI} \subseteq \Delta^3_4$ of Example~\ref{ex:FI}, decomposed as a product according to its $x_3$-coordinate, its $(x_2,x_4)$-coordinates, and its $x_1$-coordinate.}
    \label{fig:exFI}
\end{figure}

The product decomposition of $F_{\tbI}$ illustrated in Figure~\ref{fig:exFI} can be interpreted as a product of the standard permutohedra of dimensions zero and one, and the permutohedral complex $\Delta^3_1$.  (We recall the definition of the standard permutohedra in Definition~\ref{def:permutohedron} below.)  To see the permutohedron arising more clearly, it is illuminating to consider one additional example.

\begin{example}
\label{ex:FI2}
    Let $\tbI'$ be the chain obtained by removing $I_1$ from the chain $\tbI$ of Example~\ref{ex:FI}.  Then
    \[F_{\tbI'} = \left\{(x_1, \ldots, x_4) \in \Y^4 \; \left| \; \substack{\textstyle \sum_{i \in I} |x_i| \leq \delta^4_{|I|} \; \text{ for all } I \subseteq [4]\\\\  \textstyle \zeta^1 x_2 + \zeta^0 x_3 + \zeta^2 x_4 = 9}\right.\right\}.\]
As before, $x_1\in\Delta_1^3$ and
is independent from $x_2,x_3,x_4$.  The coordinates $x_2,x_3,x_4$, on the other hand, must satisfy
\[x_2 \in \mathbb{R}^{\geq0} \cdot \zeta^2, \;\; x_3 \in \mathbb{R}^{\geq0} \cdot 1, \;\; x_4 \in \mathbb{R}^{\geq0} \cdot \zeta\]
as well as
\[|x_2|+ |x_3| + |x_4| = 9,\]
One can check that these conditions, together with the inequalities in $F_{\tbI'}$, shows that $x_2, x_3, x_4$ lie in a hexagon with vertices
\[(4\zeta^2, 3, 2\zeta), \; (4\zeta^2, 2, 3\zeta), \; (3\zeta^2, 4, 2\zeta), \; (3\zeta^2, 2, 4\zeta), \; (2\zeta^2, 4, 3\zeta), \; (2\zeta^2, 3, 4\zeta).\]
See Figure~\ref{fig:exFII}.

\begin{figure}[h!]
    \centering
    \begin{tikzpicture}

    \filldraw[fill=gray, opacity=0.5, very thick] (2.15,0) -- (2.85,-1) -- (4.15,-1) -- (4.85,0) -- (4.15,1) -- (2.85,1) -- (2.15,0);
    
    \node at (5.5,0) {$\times$};
    
    \draw[very thick] (7,0) -- (8,0);
    \draw[very thick] (7,0) -- (6.5,0.866);
    \draw[very thick] (7,0) -- (6.5,-0.866);

    \end{tikzpicture}
    \caption{The $\Delta$-face $F_{\tbI'} \subseteq \Delta^3_4$ of Example~\ref{ex:FI2}, decomposed as a product according to its $(x_2,x_3,x_4)$-coordinates and its $x_1$-coordinate.}
    \label{fig:exFII}
\end{figure}
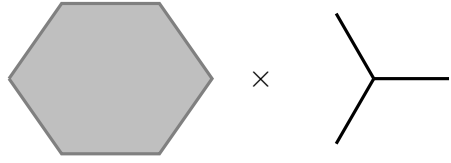
\end{example}

Generalizing the above examples, we now prove that each $\Delta$-face of $\Delta^r_n$ is equal to a product of smaller permutohedral complexes and permutohedra; this will provide the analogue of Propositions \ref{lem:productdecompositionofstratum} and \ref{cor:prodCI} (the product decompositions of boundary strata and $\T$-cosets, respectively), and furthermore, the product decomposition will be used to verify the dimension of each $\Delta$-face.

First, we must recall the definition of the standard permutohedron.

\begin{definition}
\label{def:permutohedron}
The {\bf permutohedron} $\Pi_n$ is the set of points $(x_1, \ldots, x_n) \in \mathbb{R}^n$ satisfying the inequalities
\[\sum_{i \in I} x_i \leq \delta^n_{|I|}\]
for all $I \subsetneq [n]$ and the equality
\begin{equation*}
\sum_{i \in [n]} x_i = \delta^n_n.
\end{equation*}
(Analogously to the situation for Losev--Manin space, we point out again that while the analogy to the permutohedron motivates the definition of $\Delta^r_n$, it is not literally the case that setting $r=1$ in the definition of $\Delta^r_n$ recovers $\Pi_n$.  See Remark~\ref{rmk:r=1}.)
\end{definition}

It is well-known that $\Pi_n$ is a polytope of dimension $n-1$.  More generally, one obtains other polytopes of the same dimension by shifting $\Pi_n$, as follows.

\begin{definition}
Let $\gamma \in \mathbb{R}$. The {\bf $\gamma$-shifted permutohedron} $\Pi_n +\gamma$ is the set of points $(x_1, \ldots, x_n) \in \R^n$ satisfying the inequalities
\[\sum_{i \in I} x_i \leq \delta^n_{|I|} + |I|\cdot\gamma\]
for all $I \subsetneq [n]$ and the equality
\[\sum_{i \in [n]} x_i = \delta^n_n + n\gamma.\]
\end{definition}

Given these definitions, we have the following alternative description of a face $F_{\tbI}$ of $\Delta^r_n$.

\begin{proposition}
\label{prop:faceproduct}
Let $\tbI = (I_1, \ldots, I_k, \a)$ be a chain.  Then $F_{\tbI}$ is equal to the set of points $(x_1, \ldots, x_n) \in \C^n$ that satisfy the following conditions:
\begin{enumerate}[label={\bf($\text{C}\mathbf{\arabic*'}$)}]
    \item $(x_i)_{i \in [n] \setminus I_k} \in \Delta^r_{|[n] \setminus I_k|}$, or in other words, 
    \[\sum_{i \in I} |x_i| \leq \delta_{|I|}^{|[n] \setminus I_k|}\]
    for all $I \subseteq [n] \setminus I_k$;
    \item $x_i \in \R^{\geq 0} \cdot \zeta^{-\a(i)}$ for all $i \in I_k$;
    \item for all $j \in \{1, \ldots, k\}$, the point $(|x_i|)_{i \in I_j \setminus I_{j-1}}$ lies in the shifted permutohedron
    \[\Pi_{|I_j \setminus I_{j-1}|} + \gamma_j\]
    for $\gamma_j := |[n] \setminus I_j|$.
\end{enumerate}
\end{proposition}
\begin{proof}
By Remark~\ref{cor:settheoreticdescriptionofface}, it suffices to show that conditions ({\bf C1}) -- ({\bf C3}) are equivalent to conditions ({\bf C$\mathbf{1}'$}) -- ({\bf C$\mathbf{3}'$}).

First, suppose that $(x_1, \ldots, x_n) \in \C^n$ satisfies ({\bf C1}) -- ({\bf C3}).  Then, given $I \subseteq [n] \setminus I_k$, we have
\begin{align*}
    \sum_{i \in I} |x_i| &= \sum_{i \in I \cup I_k} |x_i| - \sum_{i \in I_k} |x_i|\\
    &\leq \delta^n_{|I|+|I_k|} - \sum_{i \in I_k} |x_i|\\
    &= \delta^n_{|I| + |I_k|} - \delta^n_{|I_k|}\\
    &= \delta^{|[n] \setminus I_k|}_{|I|},
\end{align*}
  where the first equality follows from the fact that $I$ and $I_k$ are disjoint, the inequality from ({\bf C1}), and the second equality from ({\bf C3}).  Thus, $(x_1, \ldots, x_n)$ satisfies ({\bf C$\mathbf{1}'$}).

Condition ({\bf C$\mathbf{2}'$}) is identical to condition ({\bf C2}), so there is nothing to check.  Finally, for condition ({\bf C$\mathbf{3}'$}), let $j \in \{1, \ldots, k\}$.  Then
\begin{align*}
    \sum_{i\in I_j\setminus I_{j-1}}\abs{x_i}&=\sum_{i\in I_j}\abs{x_i}-\sum_{i\in I_{j-1}}\abs{x_i}\\
    &=\delta_{\abs{I_j}}^n-\delta_{\abs{I_{j-1}}}^n\\
    &=\delta_{\abs{I_j\setminus I_{j-1}}}^{\abs{[n]\setminus I_{j-1}}}\\
    &=\delta_{\abs{I_j\setminus I_{j-1}}}^{\abs{I_j\setminus I_{j-1}}}+\abs{[n]\setminus I_j}\cdot \abs{I_j\setminus I_{j-1}},
\end{align*}
which is the equality in the definition of the shifted permutohedron $\Pi_{|I_j \setminus I_{j-1}|} + \gamma_j$.  Furthermore, for any $I \subsetneq I_j \setminus I_{j-1}$, we have
\begin{align*}
    \sum_{i\in I} \abs{x_i}&=\sum_{i\in I_{j-1}\cup I} \abs{x_i}-\sum_{i\in I_{j-1}}\abs{x_i}
    \\&\le \delta^n_{\abs{I}+\abs{I_{j-1}}}-\delta^n_{\abs{I_{j-1}}}
    \\&=\delta_{\abs{I}}^{\abs{[n]\setminus I_{j-1}}}
    \\
    &=\delta^{\abs{I_j\setminus I_{j-1}}}_{\abs{I}}+\abs{[n]\setminus I_j}\cdot \abs{I},
\end{align*}
which are the inequalities in the definition of the shifted permutohedron.  Thus, $(x_1, \ldots, x_n)$ satisfies ({\bf C$\mathbf{3}'$}).

For the reverse direction, suppose that $(x_1, \ldots, x_n) \in \C^n$ satisfies ({\bf C$\mathbf{1}'$}) -- ({\bf C$\mathbf{3}'$}). It is automatic that $(x_1,\ldots, x_n)\in \Y^n$ satisfies ({\bf C$\mathbf{2}$}).  For ({\bf C$\mathbf{1}$}), note that for any $I\subseteq [n]$, we have
\begin{align*}
    \sum_{i\in I}\abs{x_i}&=\sum_{i\in I\setminus I_{k}} \abs{x_i}+\sum_{j=1}^k \sum_{i\in I\cap(I_j\setminus I_{j-1})}\abs{x_i}\\
    &\le \delta^{\abs{[n]\setminus I_k}}_{\abs{I\setminus I_k}}+ \sum_{j=1}^k \Big(\delta^{\abs{I_j\setminus I_{j-1}}}_{\abs{I\cap(I_j\setminus I_{j-1})}}+\abs{[n]\setminus I_j}\cdot \abs{I\cap(I_j\setminus I_{j-1})}\Big)
    \\
    &= \delta^{\abs{[n]\setminus I_k}}_{\abs{I\setminus I_k}}+\sum_{j=1}^k \left(\delta^{\abs{[n]\setminus I_{j-1}}}_{\abs{I\cap (I_j\setminus I_{j-1})}}\right),
    \end{align*}  
For notational convenience, we define
   \[n_j := |[n] \setminus I_{j-1}| \;\text{ and }\; a_j := |I \cap (I_j \setminus I_{j-1})|\]
   for each $j \in \{1, \ldots, k+1\}$, where $I_0 = \emptyset$ and $I_{k+1} = [n]$.  In this notation, the above is expressed as
   \[\sum_{i \in I} |x_i| \leq \sum_{j=1}^{k+1} \delta^{n_j}_{a_j}.\]
   From here, using the fact that $n_1 \geq n_2 \geq \cdots \geq n_{k+1}$ and $a_j \leq n_{j+1} - n_j$ for each $j$, one can check that
   \[\sum_{j=1}^{k+1} \delta^{n_j}_{a_j}\le \delta^{n_1}_{a_1 + \cdots + a_{k+1}}\]
   and hence
\[\sum_{i\in I}\abs{x_i} \leq \delta^{n_1}_{a_1 + \cdots + a_{k+1}}=\delta^n_{\abs{I}},\]
    proving that $(x_1,\ldots, x_n)\in \Y^n$ satisfies ({\bf C$\mathbf{1}$}). We also have, for $\ell\in \{1,\ldots, k\}$, that 
    \begin{align*}
        \sum_{i\in I_{\ell}}\abs{x_i}&=\sum_{j=1}^{\ell}\sum_{i\in I_{j}\setminus I_{j-1}}\abs{x_i}\\&=\sum_{j=1}^{\ell}\left(\delta_{\abs{I_j\setminus I_{j-1}}}^{\abs{I_j\setminus I_{j-1}}}+\abs{[n]\setminus I_j}\cdot \abs{I_j\setminus I_{j-1}}\right)\\
        &=\sum_{j=1}^{\ell}\left(\delta_{\abs{I_j}}^n-\delta_{\abs{I_{j-1}}}^n\right)\\
        &=\delta^n_{\abs{I_{\ell}}},
        \end{align*}
    proving that $(x_1,\ldots, x_n)$ satisfies ({\bf C$\mathbf{3}$}) and completing the proof of the proposition.
\end{proof}

As a corollary, we find that the $\Delta$-faces of $\Delta^r_n$ have a product decomposition, which shows that they are polytopal complexes of the expected dimension.

\begin{corollary}
\label{cor:facedim}
For any chain $\tbI$, we have
 \begin{equation}
     \label{eq:faceprod}
F_{\tbI}= \Delta_{\abs{[n]\setminus I_k}}^r \times \varphi_{\a}\left(\prod_{j=1}^k (\Pi_{\abs{I_j\setminus I_{j-1}}} + \gamma_j)\right),
 \end{equation}
 where $\varphi_{\a}: \R^{|I_k|} \rightarrow \C^{|I_k|}$ multiplies the $i$th coordinate by $\zeta^{-\a(i)}$, and $\gamma_j := |[n] \setminus I_j|$.  In particular, $F_{\tbI}$ is a polytopal complex of dimension $n-k$.
\end{corollary}
\begin{proof}
The product decomposition \eqref{eq:faceprod} follows immediately from Proposition~\ref{prop:faceproduct}.  For the ``in particular," it suffices to prove that $\Delta^r_n$ itself is a polytopal complex of dimension $n$.  If this is the case, then \eqref{eq:faceprod} implies that $F_{\tbI}$ is a product of polytopal complexes and polytopes, so it is a polytopal complex.  Furthermore, using the fact that $\dim(\Pi_m) = m-1$, we find that the dimension of $F_{\tbI}$ is
\[|[n] \setminus I_k| + \sum_{j=1}^k \Big( |I_j \setminus I_{j-1}| -1 \Big) = n-k,\]
as claimed.

To prove that $\Delta^r_n$ is a polytopal complex of dimension $n$, let $\mathfrak{c}: [n] \rightarrow \Z_r$, and define
\[\Y^n_{\mathfrak{c}} = (\R^{\geq 0} \cdot \zeta^{c(1)}) \times \cdots \times (\R^{\geq 0} \cdot \zeta^{c(n)}) \subseteq \Y^n.\]
Then $\Y^n_{\mathfrak{c}}$ is naturally identified with $\R^n$, and under this identification, we have
\[\Delta^r_n \cap \Y^n_{\mathfrak{c}} = \left\{ (x_1, \ldots, x_n) \in \R^n \; \left| \; \substack{\displaystyle \sum_{i \in I} x_i \leq \delta^n_{|I|} \; \text{ for all }\; I \subseteq [n],\\ \displaystyle x_i \geq 0 \; \text{ for all }\;  i \in [n]} \right.\right\}.\]
In particular, $\Delta^r_n \cap \Y^n_{\mathfrak{c}}$ is a polytope in $\R^n$, and it contains the origin as well as the $n$ standard basis vectors $e_1, \ldots, e_n$.  Since these are $n+1$ affinely independent points in $\R^n$, it follows that $\Delta^r_n \cap \Y^n_{\mathfrak{c}}$ has dimension $n$.  Thus, we have an expression
\[\Delta^r_n = \bigcup_{\mathfrak{c}: [n] \rightarrow \Z_r} (\Delta^r_n \cap \Y^n_{\mathfrak{c}})\]
as a union of $n$-dimensional polytopes intersecting only along the faces where some subset of the coordinates is equal to zero.  That is, $\Delta^r_n$ is a polytopal complex of dimension $n$.
\end{proof}

\begin{remark}
The proof of Corollary~\ref{cor:facedim} shows that the individual $n$-dimensional polytopes $\Delta^r_n \cap \Y^n_{\mathfrak{c}}$ that are glued to form $\Delta^r_n$ are independent of $r$.  Setting $r=2$, we find that $\Delta^2_n$ is, in fact, a polytope (the type-$B$ permutohedron) and the polytopes that comprise any $\Delta^r_n$ are the intersection of $\Delta^2_n$ with an octant.  When $n=2$, for example, $\Delta^2_2$ is an octagon whose intersection with each quadrant is a pentagon, and these pentagons are visible in Figure~\ref{fig:Delta32example} as the building blocks of $\Delta^3_2$.
\end{remark}

\begin{remark}
\label{rmk:octants}
The polytopes $\Delta^r_n \cap \Y^n_{\mathfrak{c}}$ for different choices of $\mathfrak{c}$ are all isomorphic to one another, and they are precisely the polytope of the toric variety $\M^1_n$ that arose in Section~\ref{sec:construction}.  In particular, one can interpret the maps
\[\pi_\alpha|_{U_\alpha}:U_{\alpha} \rightarrow \M^1_n\]
and
\[p: \L^r_n(\zeta) \rightarrow \M^1_n\]
of Section~\ref{subsec:geometry} in polytopal terms: the map $\pi_\alpha|_{U_\alpha}$ amounts to looking locally at a single octant $\Delta^r_n \cap \Y^n_{\mathfrak{c}}$ of $\Delta^r_n$, while the map $p$ amounts to identifying all of the octants of $\Delta^r_n$ with one another.
\end{remark}

We are now ready to prove that---just as in the settings of boundary strata and $\T$-cosets---the $\Delta$-face structure of $\Delta^r_n$ is precisely captured by chains.

\begin{proposition}
\label{prop:ItoFI}
Let $r \geq 2$ and $n\geq 0$.  The association
\[\tbI \mapsto F_{\tbI}\]
is a bijection from the set of decorated nested chains of subsets of $[n]$ to the set of $\Delta$-faces of $\Delta^r_n$.  Furthermore, this bijection satisfies
\begin{enumerate}[label=(\roman*)]
    \item $\text{length}(\tbI) = \text{codim}(F_{\tbI})$,
    \item $F_{\tbI} \subseteq F_{\tbJ}$ if and only if $\tbI$ refines $\tbJ$.
\end{enumerate}
\end{proposition}
\begin{proof}
The surjectivity of $\tbI \mapsto F_{\tbI}$ is is the content of Lemma~\ref{lem:nonemptyface}.  The injectivity will follow from item (ii) of the proposition, since the only way that $\tbI$ is a refinement of $\tbJ$ and vice versa is if $\tbI = \tbJ$.  Item (i) is immediate from Corollary~\ref{cor:facedim}.

Thus, all that remains is to prove item (ii).  One direction is clear: if $\tbI$ is a refinement of $\tbJ$, then the hyperplanes intersected to form $F_{\tbJ}$ are a subset of the hyperplanes intersected to form $F_{\tbI}$, so $F_{\tbI} \subseteq F_{\tbJ}$.

Conversely, suppose that $F_{\tbI} \subseteq F_{\tbJ}$, where
\[\tbI = ( I_1, \ldots, I_k, \a),\]
\[\tbJ = (J_1, \ldots, J_\ell, \b).\]
By completing $\tbI$ to a maximal chain (which, in particular, involves extending $\a$ to a function $\a: [n] \rightarrow \Z_r$), one can find a vertex
\[(x_1, \ldots, x_n) \in F_{\tbI}.\]
Equation~\eqref{eq:xj} implies that, for each $i \in \{1, \ldots, n\}$, we have
\[x_i = \zeta^{-\a(i)} \lambda_i\]
for some $\lambda_i \in [n]$.  The fact that $v \in F_{\tbI} \subseteq F_{\tbJ}$ then implies that
\[\sum_{i \in J_j} \zeta^{b(i)} \cdot \zeta^{-\a(i)} \lambda_i = \delta_{|J_j|}\]
for each $j$, and from here, the same triangle inequality argument from Lemma~\ref{lem:nonemptyface} shows that $\a(i) = \b(i)$ for all $i \in J_j$.  Thus, the decorations on $\tbJ$ agree with the decorations on $\tbI$ where both are defined, and what remains is to prove that
\begin{equation}
    \label{eq:JsubsetI}
\{J_1, \ldots, J_{\ell}\} \subseteq \{I_1, \ldots, I_k\}.
\end{equation}

If not, then one possibility is that $J_{\ell} = [n]$ whereas $I_k \neq [n]$.  In this case, however, if $i \in [n] \setminus I_k$, then for any $s \in \Z_r$ one can construct a vertex $(x_1, \ldots, x_n) \in F_{\tbI}$ with $x_i = \zeta^s \lambda_i$.  By contrast, any vertex of $F_{\tbJ}$ has $x_i = \zeta^{-b(i)} \lambda_i$, so it cannot be the case that $F_{\tbI} \subseteq F_{\tbJ}$.

Having ruled out this possibility, the failure of \eqref{eq:JsubsetI} implies that
\begin{equation}
    \label{eq:JnotinI}
\{J_1, \ldots, J_\ell, [n]\} \not\subseteq \{I_1, \ldots, I_k, [n]\},
\end{equation}
and from here, we can cite the known face structure of the permutohedron $\Pi_n$ (see, for example, \cite[Proposition 2.6]{Postnikov} or \cite[Section 4.1]{AguiarArdila}).  In particular, faces of $\Pi_n$ are indexed by subsets of $[n]$ in which the largest is $[n]$ itself, and if we set
\[F_{J_1, \ldots, J_{\ell}, [n]} := \left\{(x_1, \ldots, x_n) \in \Pi_n \; \left| \; \sum_{i \in J_j} x_i = \delta_{|J_j|} \; \text{ for all } j\right.\right\}\]
and similarly
\[F_{I_1, \ldots, I_{k}, [n]} := \left\{(x_1, \ldots, x_n) \in \Pi_n \; \left| \; \sum_{i \in I_j} x_i = \delta_{|I_j|} \; \text{ for all } j\right.\right\},\]
then it is known that \eqref{eq:JnotinI} implies
\[F_{J_1, \ldots, J_{\ell}, [n]} \not\subseteq F_{I_1, \ldots, I_{k}, [n]}.\]
That is, there exists $(x_1, \ldots, x_n) \in F_{J_1, \ldots, J_{\ell}, [n]}$ with $(x_1, \ldots, x_n) \notin F_{I_1, \ldots, I_{k}, [n]}$.  It is straightforward to see that, for any extension of $a$ to a function $[n] \rightarrow \Z_r$, we have
\[(\zeta^{-a(1)}x_1, \ldots, \zeta^{-a(n)} x_n) \in F_{\tbJ}\]
but
\[(\zeta^{-a(1)}x_1, \ldots, \zeta^{-a(n)} x_n) \notin F_{\tbI}.\]
This contradicts our assumption that $F_{\tbI} \subseteq F_{\tbJ}$ and thus completes the proof.
\end{proof}

The previous two sections concluded with a product decomposition of the relevant objects, and we close this section by briefly noting that the analogous product decomposition also holds for $\Delta$-faces.

\begin{remark}
\label{rem:productdecompositionofface} For any chain $\tbI = (I_1, \ldots, I_k, \a)$, there is an isomorphism
\[F_{\tbI} \cong \Delta^r_{|[n] \setminus I_k|} \times \prod_{j=1}^k \Pi_{|I_j \setminus I_{j-1}|}.\]
(The word ``isomorphism" here can be taken to mean ``combinatorial equivalence," or more strongly, ``isometry" under the standard inner products on $\R^n$ and $\C^n$.)  This follows directly from Corollary~\ref{cor:facedim}.
\end{remark}

\section{Proof Theorem~\ref{thm:main} and Enhancements}
\label{sec:conclusions}

We are now positioned to complete the proof of the main theorem:

\main*
\begin{proof}
Propositions~\ref{prop:ItoSI}, \ref{prop:ItoCI}, and \ref{prop:ItoFI} give bijections between each of these three sets and the set of decorated nested chains of subsets of $[n]$, and in each case, the dimension is encoded by the co-length of a chain and the inclusion relation is encoded by refinement of chains.  
\end{proof}

In fact, the statement of Theorem~\ref{thm:main} can be enhanced to incorporate two pieces of additional structure: product decompositions of the three types of objects and an action of $S(r,n)$ on each.  The remainder of this last section of the paper is devoted to carrying out these enhancements.

\subsection{Product decompositions}\label{subsec:productdecompositionsarecompatible} 

Theorem~\ref{thm:main} is an analogue of results pertaining to the Losev--Manin moduli spaces $\L_n$ studied in \cite{losev2000}.  There, the relevant group is the symmetric group $S_n$ (in which there is a precisely analogous definition of $\T$-cosets with $\T$ the set of adjacent transpositions), and the relevant polytopal complex is the permutohedron $\Pi_n$ (which, in this case, is actually a polytope).  The theorem, then, is that there are dimension-preserving and inclusion-preserving bijections
\begin{equation}
    \label{eq:r=1}
\left\{ \substack{\textstyle\text{boundary }\\ \textstyle\text{strata in } \L_n}\right\} \longleftrightarrow \left\{ \substack{\textstyle \T\text{-cosets}\\ \textstyle\text{ in } S_n}\right\} \longleftrightarrow \left\{ \substack{\textstyle \text{faces}\\ \textstyle\text{ of }\Pi_n}\right\}.
\end{equation}

\begin{remark}
\label{rmk:r=1}The analogy between \eqref{eq:r=1} and Theorem~\ref{thm:main}, and the fact that $S_n$ is the $r=1$ case of $S(r,n)$, suggests that $\L_n$ and $\Pi_n$ should be viewed ``morally" as the $r=1$ cases of $\L_n^r$ and $\Delta_n^r$, respectively, despite the fact that these objects are not literally recovered by setting $r=1$ in the higher-$r$ construction.
\end{remark}

With this analogy established, we note that the bijections of Theorem \ref{thm:main} preserve rich geometric structures of $\L_n^r$ in a way that incorporates the corresponding structures in Losev--Manin spaces encoded by \eqref{eq:r=1}.  More precisely, we have seen in Proposition~\ref{lem:productdecompositionofstratum}, Proposition~\ref{cor:prodCI}, and Remark~\ref{rem:productdecompositionofface} that
\begin{itemize}
    \item a boundary stratum in $\L_n^r$ is isomorphic to a product with one factor $\L_{n'}^r$ for some $n'$ and all other factors  Losev--Manin spaces;
    \item a $\T$-coset $C_{\tbI}$ in $S(r,n)$ is a coset of a subgroup isomorphic to a product with one factor $S(r,n')$ for some $n'$ and the other factors symmetric groups;
    \item a $\Delta$-face $F_{\tbI}$ in $\Delta^r_n$ is isomorphic to a product with one factor $\Delta^r_{n'}$ for some $n'$ and all other factors permutohedra.
\end{itemize}
The following theorem says that the bijections of Theorem \ref{thm:main} are compatible with these product decompositions. 

\begin{theorem}
\label{thm:proddecomps}
Under the bijections of Theorem \ref{thm:main}, a boundary stratum $S_{\tbI}$ naturally isomorphic (via Proposition~\ref{lem:productdecompositionofstratum}) with 
\begin{align*}
\L_{n_{k+1}}^r\times\prod_{j=1}^k\L_{n_j}    \end{align*} 
corresponds to a $\mathcal{T}$-coset $C_{\tbI}$ naturally identified (via Proposition~\ref{cor:prodCI}) with
\begin{align*}\label{eq:cosetdecomposition1}
     S(r,n_{k+1})\times \prod_{j=1}^k  S_{n_j} \cdot A_j,
\end{align*}
 and to a $\Delta$-face $F_{\tbI}$ naturally isomorphic (via Remark~\ref{rem:productdecompositionofface}) with 
 \begin{align*}
 \Delta^r_{n_{k+1}}\times \prod_{j=1}^{k}\Pi_{n_j}.   
 \end{align*}
In particular, if $\tbI = (I_1, \ldots, I_k, \a)$, then
\[n_{k+1} = |[n] \setminus I_k| \text{ and } n_j = |I_j \setminus I_{j-1}|\]
for all $j \in \{1, \ldots, k\}$.
\end{theorem}

\begin{proof}
This follows from Proposition \ref{lem:productdecompositionofstratum}, Proposition \ref{cor:prodCI}, and Remark \ref{rem:productdecompositionofface}, which describe the product decomposition of a boundary stratum, $\mathcal{T}$-coset, or $\Delta$-face corresponding to a given chain $\tbI$.
\end{proof}

\subsection{Actions and equivariance}\label{subsec:groupactionamdequivariance} Another key feature of each of the three settings of interest is the existence of a right action by $S(r,n)$.  In particular:
\begin{itemize}
    \item $S(r,n)$ acts on $\L^r_n(\zeta)$, because an element of $\L^r_n(\zeta)$ is determined by the choice of the curve $C$ and the first element $z_1^0, \ldots, z_n^0$ in each light orbit.  Thus, for any $A \in S(r,n)$ we can view the matrix-vector product
    \[(z_1^0, \ldots, z_n^0) \cdot A \]
    as a new tuple of elements of $C$ by identifying $\zeta^k \cdot z_j^i$ with $\sigma^k(z_j^i)$, so setting
    \[(C; z_1^0, \ldots, z_n^0) \cdot A := (C; (z_1^0, \ldots, z_n^0) \cdot A )\]
    gives an action of $S(r,n)$ on $\L^r_n(\zeta)$.
    \item $S(r,n)$ acts on itself by multiplication on the right. 
    \item $S(r,n)$ acts on $\Delta^r_n$ by matrix-vector multiplication
    \[(x_1, \ldots, x_n) \cdot A\]
    for $(x_1, \ldots, x_n) \in \Delta^r_n$.
\end{itemize}

We will see in Lemma~\ref{lem:0dim} that the action on $\L^r_n(\zeta)$ takes $0$-dimensional boundary strata to $0$-dimensional boundary strata, and the action on $\Delta^r_n$ takes $0$-dimensional $\Delta$-faces to $0$-dimensional $\Delta$-faces.  Moreover, the bijections of Theorem~\ref{thm:main} are equivariant under these actions.  Before proving this in general, let us illustrate it in two examples.

\begin{example}
\label{ex:action}
    In the case of $(r, n) = (2, 2)$, let
    \[
    A = \begin{pmatrix}
    0 & 1 \\
    -1 & 0 
    \end{pmatrix}. 
    \]
    Then the action of $A$ on $\Delta^2_2$ rotates by $\pi/4$ counterclockwise.  On $S(2,2)$, the action is simply right-multiplication, whereas on $\L^2_2$, the action is described by
    \[
    (z_1^0 , \; z_2^{0}) \cdot A = (z_2^{1}, \; z_1^{0}); 
    \]
    in other words, if $(C; z_1^0, z_2^0)$ specifies an element of $\L^2_2$, then
    \[(C; z_1^0, z_2^0) \cdot A = (C; \tilde{z}_1^0, \tilde{z}_2^0)\]
    in which $\tilde{z}_1^0 = z_2^1$ and $\tilde{z}_2^0 = z_1^0$.  From here, consulting Figure~\ref{fig:Delta22example-boundary} shows that the bijection between $0$-dimensional boundary strata in $\L_2^2$ and vertices of $\Delta^2_2$ is $S(2,2)$-equivariant: if $(C; z_1^0, z_2^0)$ is a $0$-dimensional boundary stratum corresponding to a vertex $v \in \Delta^2_2$, then $(C; z_0^1, z_0^2)\cdot A$ is the $0$-dimensional boundary stratum corresponding to the vertex $v \cdot A$.  Similarly, from Figure~\ref{fig:Delta22example-cosets}, one sees that if $B \in S(2,2)$ is a group element (that is, a $0$-dimensional $\T$-coset) corresponding to a vertex $v \in \Delta^2_2$, then $B \cdot A$ is the group element corresponding to $v \cdot A$.
\end{example}

\begin{example} 
    In the case of $(r, n) = (3, 4)$, let
    \begin{equation}
        \label{eq:A}
    A= \begin{pmatrix}
    0 & \zeta^2 & 0 & 0 \\
    0 & 0 & 0 & \zeta \\
    0 & 0 & \zeta^2 & 0 \\
    1 & 0 & 0 & 0
    \end{pmatrix}. 
    \end{equation}
    Then we have
    \[
    (z_1^0, z_2^0, z_3^0, z_4^0) \begin{pmatrix}
    0 & \zeta^2 & 0 & 0 \\
    0 & 0 & 0 & \zeta \\
    0 & 0 & \zeta^2 & 0 \\
    1 & 0 & 0 & 0 
    \end{pmatrix} = (z_4^0, z_1^{2},z_3^2, z_2^1), 
    \] which means that after the action of $A$, the first elements of the light orbits are located at the points where $z_4^{0}, z_1^{2}, z_3^{2}, z_2^{1}$ were located before the action of $A$. For instance, the action of $A$ sends the element of $\L^3_4(\zeta)$ in Figure~\ref{fig:r3n4-zero-dim-action-before} to the element in Figure \ref{fig:r3n4-zero-dim-action-after}. 
    
    \begin{figure}[h!]
    \centering 
    \begin{subfigure}{7cm}
    \centering
    \begin{tikzpicture}[scale=0.6]
    \draw (0,0) circle (0.5cm);
    \node at (0,0) {$x^{\pm}$};
    \draw (1,0) circle (0.5cm);
    \node at (1,0) {$z_1^0$};
    \draw (2,0) circle (0.5cm);
    \node at (2,0) {$z_2^0$};
    \draw (3,0) circle (0.5cm);
    \node at (3,0) {$z_3^0$};
    \draw (4,0) circle (0.5cm);
    \node at (4,0) {$z_4^0$};
    \filldraw (4.5,0) circle (0.075cm) node[right]{$y^0$};
    
    \draw (-0.5,0.866) circle (0.5cm);
    \node at (-0.5,0.866) {$z_1^1$};
    \draw (-1, 1.732) circle (0.5cm);
    \draw (-1.5,2.598) circle (0.5cm);
    \node at (-1, 1.732) {$z_{2}^1$};
    \node at (-1.5,2.598) {$z_{3}^1$};
    \draw (-2,3.464) circle (0.5cm);
    \node at (-2,3.464) {$z_4^1$};
    \filldraw (-2.25,3.897) circle (0.075cm) node[above]{$y^1$};
    
    \draw (-0.5,-0.866) circle (0.5cm);
    \node at (-0.5,-0.866) {$z_1^2$};
    \node at (-1, -1.732) {$z^{2}_2$};
    \draw (-1, -1.732) circle (0.5cm);
    \draw (-1.5,-2.598) circle (0.5cm);
    \node at (-1.5,-2.598) {$z_3^2$};
    \draw (-2,-3.464) circle (0.5cm);
    \node at (-2,-3.464) {$z_4^2$};
    \filldraw (-2.25,-3.897) circle (0.075cm) node[below]{$y^2$};
    \end{tikzpicture}
    \caption{An element $C$ in $\overline{\calL}^{3}_4(\zeta)$. }
    \label{fig:r3n4-zero-dim-action-before}
    \end{subfigure} \qquad 
    \begin{subfigure}{7cm}
    \centering
    \begin{tikzpicture}[scale=0.6]
    \draw (0,0) circle (0.5cm);
    \node at (0,0) {$x^{\pm}$};
    \draw (1,0) circle (0.5cm);
    \node at (1,0) {$z_2^1$};
    \draw (2,0) circle (0.5cm);
    \node at (2,0) {$z_4^2$};
    \draw (3,0) circle (0.5cm);
    \node at (3,0) {$z_3^1$};
    \draw (4,0) circle (0.5cm);
    \node at (4,0) {$z_1^0$};
    \filldraw (4.5,0) circle (0.075cm) node[right]{$y^0$};
    
    \draw (-0.5,0.866) circle (0.5cm);
    \node at (-0.5,0.866) {$z_2^2$};
    \draw (-1, 1.732) circle (0.5cm);
    \draw (-1.5,2.598) circle (0.5cm);
    \node at (-1, 1.732) {$z_4^0$};
    \node at (-1.5,2.598) {$z_{3}^2$};
    \draw (-2,3.464) circle (0.5cm);
    \node at (-2,3.464) {$z_1^1$};
    \filldraw (-2.25,3.897) circle (0.075cm) node[above]{$y^1$};
    
    \draw (-0.5,-0.866) circle (0.5cm);
    \node at (-0.5,-0.866) {$z_2^0$};
    \node at (-1, -1.732) {$z^4_1$};
    \draw (-1, -1.732) circle (0.5cm);
    \draw (-1.5,-2.598) circle (0.5cm);
    \node at (-1.5,-2.598) {$z_3^0$};
    \draw (-2,-3.464) circle (0.5cm);
    \node at (-2,-3.464) {$z_1^2$};
    \filldraw (-2.25,-3.897) circle (0.075cm) node[below]{$y^2$};
    \end{tikzpicture}
    \caption{The element $C \cdot A$ in $\overline{\calL}^{3}_4(\zeta)$. }
    \label{fig:r3n4-zero-dim-action-after}
    \end{subfigure}
    \caption{The action of $A \in S(3, 4)$ defined by \eqref{eq:A} on an element of $\overline{\calL}^{3}_4(\zeta)$.}
    \end{figure}
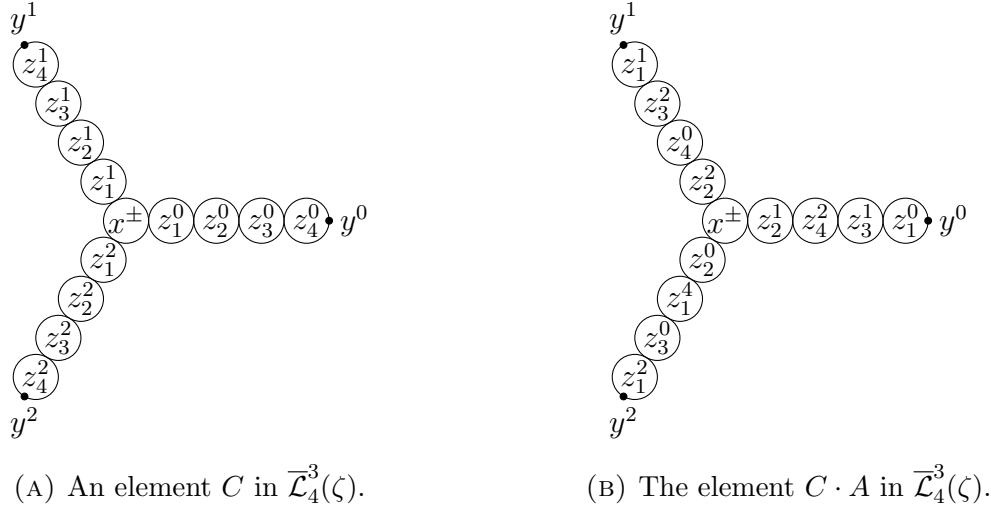
    
Unpacking the bijections of Theorem~\ref{thm:main}, one sees that the curve in Figure~\ref{fig:r3n4-zero-dim-action-before} is the $0$-dimensional boundary stratum corresponding to the $\T$-coset $\{I\}$ containing only the identity matrix, and corresponding to the vertex $(1,2,3,4) \in \Delta^3_2$.  On the other hand, the curve in Figure~\ref{fig:r3n4-zero-dim-action-after} is the $0$-dimensional boundary stratum corresponding to the chain
    \[\tbI = (\{1\}, \{1,3\}, \{1,3,4\}, \{1,2,3,4\}; \a),\]
where
\[\a(1) = 0, \; \a(2) = 1, \; \a(3) = 1, \; \a(4) = 2,\]
from which it is straightforward to check that it corresponds to the $\T$-coset $\{A\} = \{A \cdot I\}$ and to the vertex $(4,\zeta^2, 3\zeta^2, 2\zeta) = (1,2,3,4) \cdot A$ of $\Delta^3_2$.  Thus, in this case, we again see that the bijections of Theorem~\ref{thm:main} are $S(2,3)$-equivariant.
\end{example}

To confirm that the bijections between $0$-dimensional objects are $S(r,n)$-equivariant in general, we denote by $S_0 \in \L^r_n(\zeta)$ the zero-dimensional boundary stratum corresponding under Theorem~\ref{thm:main} to the $0$-dimensional $\T$-coset $\{I\} \subseteq S(r,n)$.  Specifically, this means that the $y^0$-spoke of $S_0$ contains the light marked points $z_1^0, \ldots, z_n^0$, with one on each component in order from innermost to outermost, or in other words that
\[z^i_j \in C^i_{n+1-j}.\]
For instance, Figure~\ref{fig:r3n4-zero-dim-action-before} illustrates $S_0 \subseteq \L^3_4(\zeta)$. 

\begin{lemma}
\label{lem:0dim}
The bijections 
\[ \left\{ \substack{\textstyle\text{zero-dimensional}\\\textstyle\text{boundary }\\ \textstyle\text{strata in } \L^r_n(\zeta)}\right\} \longleftrightarrow  S(r,n) \longleftrightarrow \left\{ \substack{\textstyle \text{vertices}\\ \textstyle \text{ of }\Delta^r_n}\right\}\]
of Theorem~\ref{thm:main} identify $A \in S(r,n)$ with the boundary stratum $S_0 \cdot A$ and with the vertex $(1,\ldots, n) \cdot A \in \Delta^r_n$.
\end{lemma}
\begin{proof}
All of these objects correspond to maximal chains
\[\tbI = (I_1, \ldots, I_n,\a)\]
with $\a: [n] \rightarrow \Z_r$, for which the nested sets can be expressed as
\begin{align*}
    I_1 &= \{i_1\}\\
    I_2 &= \{i_1, i_2\}\\
    &\vdots\\
    I_n &= \{i_1, i_2, \ldots, i_n\} = [n]
\end{align*}
for some $i_1, \ldots, i_n \in [n]$.  In this notation:
\begin{enumerate}[label=(\roman*)]
    \item the associated boundary stratum $S_{\tbI}$ is defined by the condition that
    \[\text{ the unique light marked point on the component } C^0_j \text{ is } z_{i_j}^{\a(i_j)}\]
    for each $j \in [n]$;
    \item the associated $\T$-coset is the singleton $C_{\tbI} = \{A\}$, where $A$ is the matrix defined by the condition that
    \[\text{row } n+1-j \text{ of } A \text{ has nonzero entry } \zeta^{-\a(i_j)} \text{ in column } i_j\]
    for each $j \in [n]$;
    \item the associated $\Delta$-face is the vertex $F_{\tbI} = \{(x_1, \ldots, x_n)\}$ with coordinates defined by
    \[x_{i_j} = \zeta^{-\a(i_j)} \cdot (n+1-j)\]
    for each $j \in [n]$.  (See equation~\eqref{eq:xj}.)
\end{enumerate}
In particular, it is straightforward to see that the vertex in (iii) is equal to 
\[(x_1, \ldots, x_n) = (1, \ldots, n) \cdot A\]
for the matrix $A$ in (ii), which verifies one half of the lemma. 

For the other half of the lemma, we must show that $S_{\tbI} = S_0 \cdot A$.  To see this, let $\{z^i_j\}$ denote the elements of the light orbits in $S_0$, so that, by the definition of $S_0$, we have
\[z^i_j \in C^i_{n+1-j}.\]
If $\{\tilde{z}^i_j\}$ denote the elements of the light orbits in $S_0 \cdot A$, then the definition of the action and of $A$ implies that
\[\tilde{z}^{0}_{i_j} = z^{-\a(i_j)}_{n+1-j} \in C^{-\a(i_j)}_j,\]
In particular, $S_0 \cdot A$ has just one light marked point on each of the components $C^0_j$ for $j \in [n]$, and that marked point is $\tilde{z}^{\a(i_j)}_{i_j}$.  This exactly agrees with the above description of $S_{\tbI}$, so $S_{\tbI} = S_0 \cdot A$.
\end{proof}

From here, the fact that the bijections of Theorem~\ref{thm:main} are inclusion-preserving gives a concise reinterpretation of the bijections in general.

\begin{proposition}
\label{prop:reinterpretation}
The bijections
\[ \left\{ \substack{\textstyle\text{boundary }\\ \textstyle\text{strata in } \L^r_n(\zeta)}\right\} \longleftrightarrow \left\{ \substack{\textstyle \T\text{-cosets}\\ \textstyle\text{ in } S(r,n)}\right\} \longleftrightarrow \left\{ \substack{\textstyle\Delta\text{-faces}\\ \textstyle\text{ of }\Delta^r_n}\right\}\]
of Theorem~\ref{thm:main} identify a boundary stratum $S$ with
\[\{A \in S(r,n) \; | \; S_0 \cdot A \in S\} \subseteq S(r,n)\]
and identify a $\Delta$-face $F$ with
\[\{A \in S(r,n) \; | \; (1,\ldots, n) \cdot A \in F\} \subseteq S(r,n),\]
both of which are $\T$-cosets.
\begin{proof}
Let $S$ be a boundary stratum, and let $C_S \subseteq S(r,n)$ be the $\T$-coset associated to it via the bijection of Theorem~\ref{thm:main}.  By Lemma~\ref{lem:0dim}, the bijection associates each $A \in S(r,n)$ to the zero-dimensional boundary stratum $S_0 \cdot A$.  And since it is inclusion-preserving, we have
\[A \in C_S \; \Leftrightarrow \; S_0 \cdot A \in S.\]
This proves that
\[C_S = \{A \in S(r,n) \; | \; S_0 \cdot A \in S\},\]
so in particular, the latter is indeed a $\T$-coset.  The argument for the case of a $\Delta$-face $F$ is identical.
\end{proof}
\end{proposition}

One reason to like this interpretation---in addition to the fact that it is much simpler to state than how we initially constructed the bijections of Theorem~\ref{thm:main}, and in particular does not require the auxiliary machinery of chains---is that it immediately shows that the bijections are $S(r,n)$-equivariant.  To see this, we first should note that there is a right action of $S(r,n)$ on the sets of boundary strata, $\T$-cosets, and $\Delta$-faces, in each case by setting
\[X \cdot A := \{x \cdot A \; | \; x \in X\}\]
for a boundary stratum, $\T$-coset, or $\Delta$-face $X$.  Here, the fact that $S(r,n)$ indeed acts on each of these sets is a result of the following observations:
\begin{itemize}
    \item The action of $S(r,n)$ on $\L^r_n$ preserves the topological type of $C$ (in fact, it preserves $C$ itself) while permuting marked points, so it takes boundary strata to boundary strata.
    \item The action of $S(r,n)$ on $\T$-cosets is equivalently described by
    \[\Big(\langle s_{\ell_1}, \ldots, s_{\ell_d} \rangle B\Big) \cdot A =\langle s_{\ell_1}, \ldots, s_{\ell_d} \rangle \cdot (BA),\]
    so it takes $\T$-cosets to $\T$-cosets.
    \item The action of $S(r,n)$ on $\C^n$ on $\Delta^r_n$ takes points satisfying the conditions of Remark~\ref{cor:settheoreticdescriptionofface} to points satisfying an analogous set of conditions, so it takes $\Delta$-faces to $\Delta$-faces.
\end{itemize}
From here, it is essentially immediate from Proposition~\ref{prop:reinterpretation} that the bijections of Theorem~\ref{thm:main} respect these actions.

\begin{remark}
One way to confirm that the bijections of Theorem~\ref{thm:main} are equivariant is to verify that the above three $S(r,n)$-actions all correspond, under Theorem~\ref{thm:main}, to an action on chains.  Indeed, this is the case: the image of a chain $(I_1, \dots, I_k, \a)$ under the action of the element $A \in S(r,n)$ is the chain $(I'_1, \dots, I'_k, \a')$ characterized by 
    \[I'_j = \bigcup_{i\in I_j} \{ \ell \in [n] \; | \; A_{i\ell} \neq 0\}\]
and with
\[\a'(\ell)=\a(i)-m_{i\ell} \, \text{ for all } \,  \ell \in I'_k,\] 
where $i \in [n]$ and ${m_{i\ell}} \in \Z_r$ are uniquely determined by the condition that $A_{i\ell} = \zeta^{m_{i \ell}}$.  However, it requires some care to check that this action on chains indeed matches the three actions above, so we will instead prove the equivariance of Theorem~\ref{thm:main} directly as a corollary of Proposition~\ref{prop:reinterpretation}.
\end{remark}

\begin{corollary}
The bijections of Theorem~\ref{thm:main} are $S(r,n)$-equivariant.
\end{corollary}
\begin{proof}
Let $S$ be a boundary stratum, and let $C_S$ be the associated $\T$-coset.  Under the action of $B \in S(r,n)$, we have
\[C_S \cdot B = \{AB \in S(r,n) \; | \; A \in C_S\}=\{A \in S(r,n) \; | \; AB^{-1} \in C_S\}.\]
The element $AB^{-1} \in S(r,n)$ corresponds to the boundary stratum $S_0 \cdot AB^{-1} \in \L^r_n(\zeta)$, by Lemma~\ref{lem:0dim}.  Together with the fact that the bijection from $\T$-cosets to boundary strata is inclusion-preserving, this implies that
\[AB^{-1} \in C_S \; \text{ if and only if }  \; S_0 \cdot AB^{-1} \in S.\]
Thus, we have
\[C_S \cdot B = \{A \in S(r,n) \; | \; S_0 \cdot AB^{-1} \in S\} = \{ A \in S(r,n) \; | \; S_0 \cdot A \in S \cdot B\},\]
which, by Proposition~\ref{prop:reinterpretation}, is precisely equal to $C_{S \cdot B}$.

This proves that the bijection between the sets of boundary strata and $\T$-cosets is $S(r,n)$-equivariant, and an identical proof shows the same statement for the bijection between $\Delta$-faces of $\Delta^r_n$ and $\T$-cosets.
\end{proof}

\bibliographystyle{alpha}
\bibliography{bibliography.bib}
\end{document}